\documentclass[11pt,a4paper]{article}
\usepackage[utf8]{inputenc}

\usepackage{amsmath} 
\usepackage{amsthm} 
\usepackage{amssymb}	
\usepackage{graphicx}
\usepackage{eucal}
\usepackage{stmaryrd}
\usepackage[shortlabels]{enumitem}
\usepackage{pdfpages}
\usepackage[textwidth=16cm,centering,headheight=24pt]{geometry} 
\usepackage{comment}
\usepackage{scalerel,stackengine}
\stackMath
\newcommand\reallywidehat[1]{%
\savestack{\tmpbox}{\stretchto{%
  \scaleto{%
    \scalerel*[\widthof{\ensuremath{#1}}]{\kern-.6pt\bigwedge\kern-.6pt}%
    {\rule[-\textheight/2]{1ex}{\textheight}}
  }{\textheight}%
}{0.5ex}}%
\stackon[1pt]{#1}{\tmpbox}%
}
\usepackage{url}

\newcommand{\abs}[1]{\left| #1 \right|} 
\newcommand{\norm}[1]{\left\| #1 \right\|}
 
\newcommand{\R}{\mathbb R}
\newcommand{\C}{\mathbb C}
\newcommand{\N}{\mathbb N}

\newcommand{\Z}{\mathbb Z}

\renewcommand{\H}{\mathcal{H}}
\newcommand{\T}{\mathrm{T}}
\renewcommand{\S}{\mathbb{S}}

\newcommand{\eps}{\varepsilon}
\newcommand{\ip}[2]{\left\langle #1, \, #2\right\rangle}

\newcommand{\csubset}{\subset\!\subset}

\DeclareMathOperator{\vol}{vol}
\DeclareMathOperator{\On}{O}

\DeclareMathOperator{\End}{End}

\DeclareMathOperator{\D}{D}
\let\d\relax
\DeclareMathOperator{\d}{d}

\DeclareMathOperator{\spt}{spt}
\DeclareMathOperator{\Harm}{Harm}

\DeclareMathOperator{\dist}{dist}

\DeclareMathOperator{\inj}{inj}
\DeclareMathOperator{\EndSym}{End_{sym}}
\newcommand{\mres}{\mathbin{\vrule height 1.6ex depth 0pt width
0.13ex\vrule height 0.13ex depth 0pt width 1.3ex}}

\newtheorem{theorem}{Theorem}
\newtheorem*{theorem*}{Theorem}

\newtheorem{lemma}[theorem]{Lemma}
\newtheorem{prop}[theorem]{Proposition}
\newtheorem{corollary}[theorem]{Corollary}

\newtheorem{mainthm}{Theorem}           

\theoremstyle{definition}

\newtheorem{remark}[theorem]{Remark}

\theoremstyle{remark}
\newtheorem{step}{Step}

\title{The Yang-Mills-Higgs functional on complex line bundles:\\
asymptotics for critical points}

\date{\today}
\author{Giacomo Canevari, Federico Luigi Dipasquale and Giandomenico Orlandi}

\newcommand{\Addresses}{{
  \bigskip
  \footnotesize

  Giacomo~Canevari, \textsc{Dipartimento di Informatica, Universit\`{a} di Verona, Strada Le Grazie 15, 37134 Verona, Italy}\par\nopagebreak
  \textit{E-mail address}: \texttt{giacomo.canevari@univr.it}

  \medskip

  Federico Luigi Dipasquale (Corresponding author), \textsc{Dipartimento di Matematica e Applicazioni ``Renato Caccioppoli'', Universit\`{a} degli studi di Napoli ``Federico II'', Via Cintia, Monte S. Angelo, 80126 Napoli, Italy}\par\nopagebreak
  \textit{E-mail address}: \texttt{federicoluigi.dipasquale@unina.it}

  \medskip

  Giandomenico Orlandi, \textsc{\textsc{Dipartimento di Informatica, Universit\`{a} di Verona, Strada Le Grazie 15, 37134 Verona, Italy}}\par\nopagebreak
  \textit{E-mail address}: \texttt{giandomenico.orlandi@univr.it}

}}

\begin{document}

\maketitle

\begin{abstract}
	We consider a gauge-invariant Ginzburg-Landau functional 
	(also known as Abelian Yang-Mills-Higgs model), 
	on Hermitian line bundles over closed Riemannian manifolds
	of dimension $n \geq 3$. 
	Assuming a logarithmic energy bound in the coupling parameter, 
	we study the asymptotic behaviour of critical points 
	in the London limit.
	After a convenient choice of the gauge, we show compactness 
	of finite-energy critical points 
	in Sobolev norms. 
	Moreover, 
	thanks to a suitable monotonicity formula,
	we prove that the energy densities of critical points, 
	rescaled by the logarithm of the coupling parameter,  
	converge to 
	the weight measure of a stationary, rectifiable 
	varifold of codimension 2. 
\end{abstract}


\section*{Introduction}
Let $(M,\,g)$ be a smooth, compact, connected, oriented Riemannian manifold 
without boundary, of dimension $n \geq 3$. 
Let $E\to M$ be a Hermitian line bundle over $M$, equipped with a 
(smooth) reference metric connection $\D_0$. 
For any $\eps > 0$, we consider the Ginzburg-Landau-type functional
\begin{equation} \label{magneticGL}
 G_\eps(u, \, A) := \int_M\left(\frac{1}{2}\abs{\D_A u}^2 
 + \frac{1}{2}\abs{F_A}^2 + \frac{1}{4\eps^2}
 \left(1 - \abs{u}^2\right)^2\right)\vol_g.
\end{equation}
Here~$u\colon M\to E$ is a section of the bundle, $A$ is a real-valued 
$1$-form on~$M$,
$\D_A := \D_0 - iA$ and~$F_A$ is the curvature $2$-form of~$\D_A$.
We denote the integrand of~$G_\eps$ as
\begin{equation} \label{energydensity}
 e_\eps(u_\eps, \, A_\eps) := \frac{1}{2}\abs{\D_A u}^2 
 + \frac{1}{2}\abs{F_A}^2 + \frac{1}{4\eps^2}
 \left(1 - \abs{u}^2\right)^2\!.
\end{equation}
The functional $G_\eps$ is also known as the 
\emph{Abelian Yang-Mills-Higgs energy}. 
One prominent feature of~$G_\eps$ 
is \emph{gauge-invariance}: for any $\Phi \in W^{2,2}(M,\,\S^1)$ and any 
pair $(u,\,A)\in  W^{1,2}(M,\,E) \times W^{1,2}(M,\,\T^*M)$ with finite energy,
each term of the energy density~$e_\eps$ 
is invariant under the \emph{gauge transformation}
\begin{equation}\label{eq:gauge-transf}
	(u,\,A) \mapsto \left(\Phi u,\, \Phi \cdot A \right) := 
	\left(\Phi u,\, A -  i \Phi^{-1} \d \Phi \right) \!,
\end{equation}
where $\Phi u$ is defined by the fibre-wise action
of~$\S^1\simeq\mathrm{U}(1)$ on $E$.

For any~$u$ and~$A$, we define the forms 
(see e.g.~\cite{SS-book})
\begin{equation} \label{Jac}
 j(u, \, A) := \ip{\D_A u}{i u}, \qquad
 J(u, \, A) := \frac{1}{2}\d j(u, \, A) + \frac{1}{2} F_A,
\end{equation}
called, respectively, the \emph{gauge-invariant pre-Jacobian} and the 
\emph{gauge-invariant Jacobian} of the pair $(u,\,A)$. 
If~$(u, \, A)$ is a pair with finite energy, both $j(u,\,A)$ 
and $J(u,\,A)$ are well-defined, both in the sense of distributions and 
pointwise, and invariant under gauge transformation.

In the present paper, we consider  
\emph{finite-energy critical points} of $G_\eps$.  
More precisely, we notice that the functional $G_\eps$ 
is well-defined and finite for $(u,\,A)\in \mathcal{E}$, where 
\[
\mathcal{E} := (W^{1,2} \cap L^\infty)(M,\,E) \times W^{1,2}(M,\,\T^*M)
\]  
and critical points $(u_\eps,\,A_\eps)$ of $G_\eps$ in $\mathcal{E}$ satisfy 
the Euler-Lagrange equations
\begin{align}
 \D_A^*\D_A u_\eps 
  + \frac{1}{\eps^2} \left(\abs{u_\eps}^2 - 1\right) u_\eps &= 0 \label{EL-u}\\
 \d^* F_\eps &= j(u_\eps, \, A_\eps) \label{EL-A}
\end{align}
where~$\D^*_A$ is the~$L^2$-adjoint of~$\D_A$,
$F_\eps := F_{A_\eps}$ 
and~$\d^*$ denotes the codifferential. 
By taking the differential in both sides of~\eqref{EL-A}, 
we obtain the \emph{London equation} (at level $\eps$)
\begin{equation} \label{London}
 -\Delta F_\eps + F_\eps = 2J(u_\eps, \, A_\eps)
\end{equation}
Throughout the paper, we \emph{assume} 
that~$\{(u_\eps, \, A_\eps)\} \subset \mathcal{E}$ 
is a sequence of critical points of $G_\eps$ satisfying 
the logarithmic energy bound
\begin{equation} \label{hp:logenergy}
 G_\eps(u_\eps, \, A_\eps) \leq \Lambda \abs{\log\eps}
\end{equation}
for some constant~$\Lambda>0$ that does not depend on~$\eps$.
This assumption is natural, as the energy of minimisers 
of~$G_\eps$ is precisely of order $\abs{\log\eps}$,
whenever the bundle~$E\to M$ is non-trivial
\cite[Theorem~A and Remark~2]{CDO1}.
Under this assumption, we have this (preliminary) convergence result:

\setcounter{mainthm}{-1}
\begin{mainthm}[{\cite{CDO1}}] 
\label{th:zero}
 Let $\{(u_\eps,\,A_\eps)\} \subset \mathcal{E}$ be a sequence of 
 critical points of $G_\eps$ satisfying the logarithmic energy 
 bound~\eqref{hp:logenergy}. Then, there exist a (non-relabelled)
 subsequence,  $2$-forms~$J_*$, $F_*$ and a~$1$-form~$j_*$
 such that the following properties hold:
 \begin{enumerate}[label=(\roman*)]
  \item $J(u_\eps, \, A_\eps)\to J_*$ strongly 
  in~$W^{-1,p}(M, \ \Lambda^2 \T^*M)$
  for any~$p\in \left[1, \, \frac{n}{n-1}\right)$;
  \item $F_{A_\eps}\to F_*$ strongly 
  in~$W^{1,p}(M, \ \Lambda^2 \T^*M)$
  for any~$p\in \left[1, \, \frac{n}{n-1}\right)$;
  \item $j(u_\eps, \, A_\eps)\to j_*$ strongly 
  in~$L^p(M, \ \Lambda^2 \T^*M)$
  for any~$p\in \left[1, \, \frac{n}{n-1}\right)$;
  \item the Hodge dual~$\star J_*$ is a boundary-less,
  integer-multiplicity rectifiable $(n-2)$-current in~$M$, 
  whose integral homology class~$\mathcal{C}\subseteq H_{n-2}(M; \, \Z)$
  is Poincar\'e dual to the first Chern class of~$E\to M$;
  \item the limit curvature~$F_*$ satisfies
  \begin{equation} \label{London*}
   \d^* F_* = j_*, \qquad -\Delta F_* + F_*= 2\pi J_*,
  \end{equation}
  in the sense of distributions.
 \end{enumerate}
\end{mainthm}

The proof of Theorem~\ref{th:zero} is contained in~\cite{CDO1}
(see, in particular, Theorem~A and Remark~3.12).
Property~(iii) is not mentioned explicitly in~\cite{CDO1},
but it is an immediate consequence of~\eqref{EL-A} and~(ii).
Note that Property~(i) is completely independent of 
the Euler-Lagrange equations~\eqref{EL-u}--\eqref{EL-A};
it holds true for any sequence~$\{(u_\eps, \, A_\eps)\}$
that satisfies~\eqref{hp:logenergy}.

\paragraph*{Main results.}
In this paper, we will address two more questions about 
the asymptotic behaviour, as~$\eps\to 0$, of critical points 
satisfying~\eqref{hp:logenergy}.
Our first main result 
describes the concentration of the \emph{rescaled energy densities}.
We denote by~$\spt\mu$ the support of a Radon measure~$\mu$
and by~$\H^{n-2}$ the $(n-2)$-dimensional Hausdorff measure.

\begin{mainthm}\label{th:energy_density}
	Let $\{(u_\eps,\,A_\eps)\} \subset \mathcal{E}$ be any sequence of 
	critical points of $G_\eps$ satisfying the logarithmic energy 
	bound~\eqref{hp:logenergy}.
	Then, up to extraction of a subsequence, we have
	\[
	 \frac{e_\eps(u_\eps,\,A_\eps)}{\abs{\log\eps}} \,\vol_g
	 \rightharpoonup^* \mu_* \qquad \textrm{as } \eps\to 0
	\]
	in the sense of Radon measures in $M$.
	The limit measure~$\mu_*$ satisfies~$\H^{n-2}(\spt\mu_*) < +\infty$
	and, moreover, it is the weight measure of a stationary, 
	rectifiable $(n-2)$-varifold in~$M$. 
\end{mainthm}

Recall that \emph{varifolds} are measure-theoretic generalisations of 
smooth submanifolds (see e.g.~\cite{Simon-GMT} and Section~\ref{sec:varifold}
below for details). 
In particular, stationary rectifiable varifolds are generalisations of 
minimal surfaces. Constructing non-trivial stationary, rectifiable varifolds in 
Riemannian manifolds is a fundamental problem in Geometric Measure 
Theory~\cite{Almgren,Pitts}.

\begin{remark} \label{rk:mu*}
 By $\Gamma$-convergence~\cite[Theorem~A and Remark~3.13]{CDO1},
 the limit measure~$\mu_*$ satisfies~$\mu_*\geq\pi\abs{J_*}$
 where~$J_*$ is the limit Jacobian, given by Theorem~\ref{th:zero}.
 In particular, $\mu_*$ it is always nonzero,
 so long as the bundle~$E\to M$ is nontrivial.
 Moreover, if~$\{(u_\eps, \, A_\eps)\}$ is a sequence of minimisers,
 then~$\mu_* = \pi\abs{J_*}$ and the Hodge dual~$\star J_*$
 is area-minimising in its integral homology class~$\mathcal{C}$
 \cite[Corollary~B]{CDO1}.
 However, we do not know whether the inequality
 $\mu_*\geq \pi\abs{J_*}$ may be strict, in general.
 A related question, which we do not address here, 
 is whether or not~$\mu_*$ is an integral varifold
 (see the paragraph on open questions below).
\end{remark}
\begin{remark}
 A more detailed description of~$\mu_*$ can be found 
in Theorem~\ref{thm:rectifiability} in Section~\ref{sec:varifold},
of which Theorem~\ref{th:energy_density} is an abridged version.
Contrary to what is observed in related problems
(e.g.~\cite{BethuelOrlandiSmets-Annals, Stern2021}),
the measure~$\mu_*$ contains no diffuse part;
see Remark~\ref{rk:nodiffuse}.
\end{remark}

The proof of Theorem~\ref{th:energy_density} relies 
on a suitable monotonicity formula 
(see Theorem~\ref{thm:monotonicity} below). 



In our second main result, 
we prove compactness for the sequence  
$\{(u_\eps,\, A_\eps)\}$ in
$W^{1,p}(M,\,E) \times W^{2,p}(M,\,\T^*M)$, for any $p$ with 
$1 \leq p < \frac{n}{n-1}$, 
so long as $A_\eps$ is in Coulomb gauge, that is, 
\begin{equation} \label{globalCoulomb}
 \begin{split}
  A_\eps = \d^*\psi_\eps + \zeta_\eps 
 \end{split}
\end{equation}
where~$\psi_\eps$ is an exact~$2$-form, $\zeta_\eps$
is a harmonic~$1$-form and $\norm{\zeta_\eps}_{L^\infty(M)}\leq C_M$
for some constant~$C_M$ that depends on~$M$ only
(see e.g.~\cite[Lemma~2.10]{CDO1}).
%

 \begin{mainthm} \label{th:uA}
  Let $\{(u_\eps,\,A_\eps)\} \subset \mathcal{E}$ 
  be a sequence of critical points of $G_\eps$. 
  Assume that $\{(u_\eps,\,A_\eps)\}$ satisfies the logarithmic energy bound 
  \eqref{hp:logenergy} and 
  that each $A_\eps$ satisfies~\eqref{globalCoulomb}.
  Then, up to extraction of a (non-relabelled) subsequence, there holds
  \begin{equation*}
   u_\eps\to u_* \quad \textrm{strongly in } W^{1,p}(M, \, E), \qquad 
   A_\eps\to A_* \quad \textrm{strongly in } W^{2,p}(M, \, \T^*M) 
  \end{equation*}
  as~$\eps\to 0$, for any~$p$ with~$1\leq p < \frac{n}{n-1}$.
  Moreover, $u_*$ and~$A_*$ are smooth in~$M\setminus\spt\mu_*$,
  they satisfy~$\abs{u_*} = 1$ in~$M\setminus\spt\mu_*$ and
  \begin{equation} \label{weak-A*harmonic}
   j(u_*, \, A_*)= \d^* F_{A_*}
  \end{equation}
  in the sense of distributions in~$M$.
 \end{mainthm}

\begin{remark} \label{rk:A*harmonic}
 By continuity of the Jacobian and the differential operator, 
 the limits~$u_*$ and~$A_*$ are compatible with~$J_*$, $F_*$, $j_*$
 given by Theorem~\ref{th:zero}
 --- that is, $F_* = F_{A_*}$, $J_* = J(u_*, \, A_*)$, 
 and~$j_* = j(u_*, \, A_*)$.
 Moreover, Equation~\eqref{weak-A*harmonic} implies that~$u_*$ 
 is an $A_*$-harmonic unit section away from the support of~$\mu_*$,
 where both~$u_*$ and~$A_*$ are smooth.
 In other words, there holds
 \begin{equation} \label{Aharmonic}
  \D_{A_*}^*\D_{A_*} u_* = \abs{\D_{A_*} u_*}^2 u_* 
  \qquad \textrm{pointwise in } M\setminus\spt\mu_*
 \end{equation}
 (see Remark~\ref{rk:strongA*harmonic} for details).
 Finally, by passing to the limit in~\eqref{globalCoulomb},
 it follows that 
 $A_* = \d^*\psi_* + \zeta_*$, where~$\psi_*$
 is an exact $2$-form and~$\zeta_*$ is a harmonic~$1$-form.
\end{remark}

\begin{remark}\label{rk:ThA-minimisers}
	Theorem~\ref{th:uA} applies, in particular, 
	to any sequence of \emph{minimisers} of~$G_\eps$
	in the class
	$W^{1,2}(M,\,E) \times W^{1,2}(M,\,\T^*M)$. 
	Indeed, by~\cite[Lemma~2.1]{CDO1}, minimisers belong 
	to~$\mathcal{E}$ and, by \cite[Remark~2]{CDO1},
	they satisfy the assumption~\eqref{hp:logenergy}.
\end{remark}


Before illustrating the ideas of the proof, let us 
motivate the interest towards our results.

\paragraph{Background and motivation.}
Ginzburg-Landau functionals of the form~\eqref{magneticGL} 
were originally proposed as a model of 
superconductivity. This theory gained 
much popularity, as it accounts for most commonly observed 
effects in superconductors,
and it gradually became relevant to other areas of physics,
such as particle physics and gauge theories.
The asymptotic regime as~$\eps\to 0$, which is known as
the \emph{London limit} in 
superconductivity theory or the \emph{strongly repulsive limit}
in particle physics, is characterised by the emergence of
\emph{topological singularities}:
due to topological obstructions set by the structure of the bundle~$E\to M$,
critical points develop singularities, in the asymptotic limit, 
which are supported on a set of dimension $n-2$.
The topologically-driven energy concentration phenomenon
can be detected by the distributional Jacobian,
which both enjoys remarkable compactness properties and 
identifies the topological singularities that emerge in the limit 
as~$\eps\to 0$ (see e.g.~\cite{JerrardSoner-GL, ABO1, ABO2}).

The asymptotic analysis of~\eqref{magneticGL} in 
the limit as~$\eps\to 0$ relies on analogous
results for the ``non-magnetic'' version of the functional,
in which the variable~$A$ is set to zero and the energy 
is considered as a functional of~$u$ only. 
The analysis of critical points of Ginzburg-Landau-type functionals without magnetic field was 
initiated by Bethuel, Brezis, and H\'{e}lein~\cite{BBH} in
the 2-dimensional, Euclidean setting. The analysis was then extended to critical points on $2$-dimensional Riemann surfaces~\cite{Baraket},
higher-dimensional Euclidean domains~\cite{LinRiviere2, BethuelBrezisOrlandi}
and, more recently, higher-dimensional manifolds~\cite{Cheng2020,Stern2021, ColinetJerrardSternberg, DePhilippisPigati}.
In particular, our Theorems~\ref{th:uA} and~\ref{th:energy_density}
generalise results that were first obtained in~\cite{BethuelBrezisOrlandi}
in the Euclidean, non-magnetic case.
In all these works, a main difficulty lies in the 
lack of uniform energy bounds, 
as the natural energy scaling considered is
the logarithmic one~\eqref{hp:logenergy}.



As for the gauge-invariant, ``magnetic'' version of the functional,
the first results addressed the asymptotic
behaviour of minimisers in two-di\-men\-sio\-nal Euclidean domains~\cite{BethuelRiviere} and on Riemann surfaces~\cite{Orlandi,Qing}.
A detailed analysis of the asymptotics 
of critical points for the gauge-invariant 
functional in 2 dimensions, in the Euclidean setting and
in the logarithmic energy regime, has been 
carried out by Sandier and Serfaty 
--- see e.g.~\cite[Chapter~13]{SandierSerfaty-book}
and the references therein. 
Apparently, 
less effort has been put in studying the asymptotic behaviour of critical
points in higher dimensions. Recently, this problem has been 
addressed in~\cite{PigatiStern},
in the context of Hermitian line bundles, for the 
\emph{self-dual} version $G_\eps^{\rm self}$ of $G_\eps$, in which the 
curvature term $\int_M \abs{F_\eps}^2\,\vol_g$ is replaced by 
$\eps^2 \int_M \abs{F_\eps}^2\,\vol_g$.
The main outcome of~\cite{PigatiStern} is that,
for a sequence $\{(u_\eps,\,A_\eps)\}$ of critical points with  
\emph{bounded} $G_\eps^{\rm self}$-energy, 
the energy densities $e_\eps^{\rm self}(u_\eps,\,A_\eps)$ 
converge, up to subsequences, to the weight measure of a 
limiting stationary, rectifiable, \emph{integral} 
varifold of codimension 2 in~$M$. 
Conversely, in~\cite{DePhilippisPigati}, it is shown that any 
non-degenerate minimal submanifold of codimension 2 in $M$ can be obtained as 
energy concentration set of a sequences of bounded energy critical points of 
$G_\eps^{\rm self}$.

\paragraph{Proofs of the main results: a sketch.}
In the paper, we first address the proof of Theorem~\ref{th:uA},
which ultimately relies
on the $L^p$-compactness for the prejacobians~$j(u_\eps, \, A_\eps)$
given by Theorem~\ref{th:zero}.
Instead, the proof of Theorem~\ref{th:energy_density} combines ideas from~\cite{BethuelBrezisOrlandi} 
in order to deal with the logarithmic energy regime  and from~\cite{Orlandi,PigatiStern} 
to deal with the Hermitian line bundle setting. 
The most difficult part is to identify $\mu_*$ as the weight measure of a 
stationary, rectifiable $(n-2)$-varifold. 
The key point is showing that, if $\mu_*(\mathcal{B}_R(x_0))$ is smaller 
than $\eta_0 \, R^{n-2}$, where $\eta_0$ is suitable
uniform ``small'' constant, 
then $\mu_*(\mathcal{B}_{R/2}(x_0)) = 0$. 
This is shown in Lemma~\ref{lemma:out-of-support}. The proof goes 
along the lines of that of \cite[Proposition~{VIII.1}]{BethuelBrezisOrlandi} 
and it is based on a \emph{clearing-out} result
(see the analysis in Section~\ref{sec:clearing-out},
in particular Proposition~\ref{prop:small-ball-small-energy} 
and Proposition~\ref{prop:clearingout}). 

A crucial r\^{o}le is played by the following (almost) 
$(n-2)$-\emph{monotonicity inequality}.
Let~$\inj(M)$ be the injectivity radius of~$M$.
For~$x_0 \in M$ and~$r \in (0, \inj(M))$, we denote
by~$\mathcal{B}_r(x_0)$ the geodesic ball of radius~$r$
centered at~$x_0$. Given~$(u_\eps, \, A_\eps)\in\mathcal{E}$,
we define
\begin{equation}\label{eq:def-E-rho-intro}
		E_\eps(x_0,\,\rho) := 
		\int_{\mathcal{B}_\rho(x_0)} e_\eps(u_\eps,\,A_\eps)\,\vol_g.
\end{equation} 
and
\begin{equation}\label{eq:def-X-rho-intro}
	X_\eps(x_0,\,\rho) := \int_{\partial \mathcal{B}_\rho(x_0)} 
	\left( \abs{\D_{A_\eps,\nu} u_\eps}^2 + \abs{{\rm i}_\nu F_{\eps}}^2 \right) \vol_{\hat{g}}
	+ \frac{1}{2\eps^2\, \rho} \int_{\mathcal{B}_\rho(x_0)} \left(1 - \abs{u_\eps}^2\right)^2 \vol_g
\end{equation}
where~$\nu$ is the exterior unit normal 
to~$\partial B_\rho(x_0)$, ${\rm i}_\nu$
denotes the contraction against~$\nu$ (see~\eqref{innerprod} below),
$F_\eps := F_{A_\eps}$ and~$\hat{g}$ is the metric induced by~$g$
on $\partial \mathcal{B}_\rho(x_0)$. 

\begin{mainthm}\label{thm:monotonicity} 
	There exist $\alpha \in (0,\,2)$, $R_0 \in (0,\,\inj(M))$, and $C > 0$, 
	all depending only on $M$ and $g$, such that the following holds. 
	If $(u_\eps,\,A_\eps) \in \mathcal{E}$ 
	is any critical point of $G_\eps$ and~$x_0$ is any point of $M$,
	then the inequality
	\begin{equation} \label{eq:monotonicity}
	\begin{split}
		\frac{e^{Cr^2}}{r^{n-2}} E_\eps(x_0,\,r)
		&+ \int_r^R \frac{e^{C\rho^2}}{\rho^{n-2}} X_\eps(x_0,\,\rho)\,\d \rho \\ 
		& \leq \frac{e^{CR^2}}{R^{n-2}} E_\eps(x_0,\,R) 
		+ \frac{C e^{CR^2}(R^\alpha -r^\alpha)}{R^{n-2+\alpha}} 
		E_\eps(x_0,\,R)
	\end{split}
	\end{equation}
	holds for any $0 < r < R < R_0$. 
\end{mainthm}

The proof of Theorem~\ref{thm:monotonicity} is delicate. 
A classical argument 
(see~e.g.~Proposition~\ref{prop:n-4-monotonicity} below),
based on the Pohozaev identity, proves that
\begin{equation} \label{monotonicity-base}
 \frac{\d}{\d\rho}\left(\frac{e^{C\rho^2} E_\eps(x_0, \, \rho)}{\rho^{n-4}}\right) \geq 0
\end{equation}
for any~$x_0\in M$, $\rho\in (0, \, \inj(M))$ and some uniform constant~$C>0$.
The exponent of~$n - 4$ at the denominator
appears for scaling reasons
(heuristically, the curvature~$F_\eps$ scales as
a second derivative, so the integral of~$\abs{F_\eps}^2$
scales as a length to the power of~$n - 4$).
The inequality~\eqref{monotonicity-base} is known 
to be sharp in the context of non-Abelian Yang-Mills-Higgs
theories \cite{SmithUhlenbeck}.
However, in our Abelian setting, Theorem~\ref{thm:monotonicity}
gives an (almost) monotonicity formula
for the energy rescaled by~$\rho^{n-2}$.
A similar result is obtained in~\cite[Section~4]{PigatiStern},
for the Ginzburg-Landau functional in the self-dual scaling.
However, the arguments of~\cite{PigatiStern} fall apart
in our setting, because we work with a different scaling 
of the energy. Instead, we exploit the properties of
the Euler-Lagrange equations. In particular, our arguments
rely on Equation~\eqref{EL-A}, which can be reformulated as a 
\emph{uniformly elliptic} equation for the variable~$A_\eps$.
By elliptic regularity estimates, we can control the 
curvature terms and improve on the monotonicity, 
obtaining in the end Theorem~\ref{thm:monotonicity}.
As a byproduct of our arguments, we obtain
a bound for the Morrey-Campanato $L^{q,2}$-norm
of~$F_\eps$, in terms of the rescaled 
energy~\eqref{eq:def-E-rho-intro}, for any~$q < n$
(see Proposition~\ref{prop:almostLinftycurvature} below).


Back on the proof of Theorem~\ref{th:energy_density},
the clearing-out property combined with the 
energy bound~\eqref{hp:logenergy} implies
a two-sided estimate for the $(n-2)$-density of $\mu_*$, see 
Lemma~\ref{lemma:mu_*}. 
Following ideas in \cite[pp.~498--499]{BethuelBrezisOrlandi} (extended first 
in the gauge-invariant setting in \cite{SandierSerfaty-book} and 
to the context of Hermitian line bundles in \cite[Section~6]{PigatiStern}), 
the conclusion is achieved once we characterise~$\mu_*$ 
in terms of the \emph{stress-energy tensor fields} $T_\eps$, defined 
in~\eqref{stresstensor} below. 
Indeed, by criticality of $(u_\eps,\, A_\eps)$, each tensor field $T_\eps$ 
is stationary (i.e., divergence-free). 
Since stationarity is a linear condition, 
their weak limit~$T_*$ is still stationary. 
On the other hand, $T_*$ is naturally 
identified with a $(n-2)$-generalised varifold in~$M$
whose weight measure~$\abs{T_*}$ is absolutely continuous 
with respect to~$\mu_*$. 
Then, by Lemma~\ref{lemma:mu_*}, we can apply a suitable version of 
the Ambrosio-Soner rectifiability criterion \cite[Theorem~3.1]{AmbrosioSoner} 
to deduce that $T_*$ is also rectifiable and that $\mu_* = \abs{T_*}$,
whence the conclusion of Theorem~\ref{th:energy_density} follows.

\paragraph{Further comments and open questions.} \label{openquestions}
We did not address specifically the problem 
of the existence, in full generality, of non-minimising 
critical points of $G_\eps$.
We believe that they could be constructed by adapting the general framework 
developed in~\cite{JerrardSternberg}.

Another interesting question that we have not addressed in this paper
is whether the varifold~$\frac{1}{\pi}\mu_*$ is \emph{integral} or not.
For the magnetic Ginzburg-Landau functional in the self-dual scaling,
it is known that the measure~$\frac{1}{\pi}\mu_*$ is integral~\cite{PigatiStern}.
By contrast, for critical points of the non-magnetic 
Ginzburg-Landau functional~\eqref{nonmagnetic_GL} in Euclidean domains,
the situation is more involved~\cite{PigatiStern-Integral}
--- that is, the density of the limiting varifold $\frac{1}{\pi}\mu_*$ at a generic point
is either~$1$ or possibly a non integer number larger or equal than 2.
Unfortunately, the arguments 
in~\cite{PigatiStern} are really tailored to the 
bounded energy regime and do not carry over to our setting.

Finally, it would be interesting to 
study, along the lines of~\cite{DePhilippisPigati}, 
whether any non-degenerate minimal codimension 2 
submanifold of~$M$ can be obtained as energy 
concentration set of critical points of~$G_\eps$, 
for an appropriate choice of the total space $E$ of the bundle $E \to M$.  
We leave all these questions to future investigation.

\paragraph*{Organisation of the paper.}

The paper is organised as follows. In Section~\ref{sec:estimates} we derive 
several \emph{a priori} estimates for gauge-invariant quantities, such as  
$\abs{u_\eps}$, $\abs{\D_{A_\eps} u_\eps}$, and $\abs{F_\eps}$, 
that are used throughout the whole paper. 
In Section~\ref{sec:compactness} we prove Theorem~\ref{th:uA}. 
In Section~\ref{sec:monotonicity} we prove the almost $(n-2)$-monotonicity 
inequality, i.e., Theorem~\ref{thm:monotonicity}. 
In Section~\ref{sec:clearing-out}, the longest 
and by far the most technical of the paper, we prove the
clearing-out property with its consequences
(Propositions~\ref{prop:small-ball-small-energy} and~\ref{prop:clearingout}).
In the last section, Section~\ref{sec:varifold},
we complete the proof of Theorem~\ref{th:energy_density}. 
 
The paper is completed by two appendices containing technical results 
concerning Poincar\'{e} and trace-type inequalities for differential forms 
(Appendix~\ref{app:Poincare}) 
and elliptic estimates for the Hodge Laplacian 
(Appendix~\ref{app:elliptic}). These results are certainly known to experts 
but we provide full proofs because they are crucial to our arguments and we 
could not find the exact statements we needed in the available literature.

\paragraph*{Notation.} 
In this paper, we use the symbol $\{X_\eps\}$ to denote a family of 
objects indexed by the parameter $\eps > 0$. Although in some instances 
$\{X_\eps\}$ may also be considered as a continuous family, we always look at 
it as a sequence. In other words, $\left\{ X_\eps \right\}$ is a shorthand for 
the sequence $\left\{X_{\eps_k}\right\}_k$, where $\eps_k \to 0$ as 
$k \to +\infty$. 
Usually, for the sake of a lighter notation, we do not relabel subsequences.

In inequalities like $A \lesssim B$, the symbol $\lesssim$ means that there 
exists a constant $C$, independent of $A$ and $B$, such that $A \leq C B$. 
In particular, 
dealing with sequences indexed by $\eps$, we use  
$\lesssim$ to denote inequality up to a constant independent of $\eps$. 
Whenever it is relevant, we keep track of the dependences
of the implicit constants. In most cases, they will depend on 
the manifold $(M,\,g)$ and on the constant $\Lambda$ in the
energy bound~\eqref{hp:logenergy}. 
Since we only consider Riemannian manifolds~$(M,\,g)$ as base spaces of 
$E \to M$ and the 
metric~$g$ on~$M$ is fixed from the very beginning, we will sometimes 
write~$M$ as a shorthand for~$(M,\,g)$.

We denote $\Lambda^k \T^* M$ the bundle of $k$-forms on $M$. 
Sobolev spaces of sections of a bundle $E \to M$ are denoted $W^{k,p}(M,\,E)$. 
In particular, $W^{1,2}(M,\,\T^*M)$ is the Sobolev space of one-forms on $M$ 
of class $W^{1,2}$. More details on such spaces can be found, for instance, 
in Appendix~A of~\cite{CDO1}.
If $\omega \in \Lambda^k \T^*M$ is a $k$-form on $M$ and 
$X$ is a vector field, we denote by ${\rm i}_X \omega$ the \emph{contraction} 
of $\omega$ with $X$, defined by setting 
\begin{equation} \label{innerprod}
 {\rm i}_X \omega(X_1,\dots,X_{k-1}) := \omega(X,X_1,\dots, X_{k-1})
\end{equation}
for any vector fields $X_1$, \dots{,}$X_{k-1}$ on $M$.

Finally, as mentioned above, we will always write~$\inj(M)$
for the injectivity radius of~$M$ and~$\mathcal{B}_r(x_0)$
for the geodesic ball in~$M$ of centre~$x_0\in M$ and~$r>0$.

\numberwithin{equation}{section}
\numberwithin{definition}{section}
\numberwithin{theorem}{section}
 
\section{Preliminary estimates for gauge-invariant quantities}\label{sec:estimates}

We recall from~\cite{CDO1} the following decomposition of~$\D_A u$ which we 
will use several times.

\begin{lemma}[{\cite[Lemma~2.2]{CDO1}}]\label{lemma:dec-D_Au}
	For any $u \in W^{1,2}(M,\,E)$ and any $A \in W^{1,2}(M,\,\T^*M)$, there 
	holds
	\begin{equation}\label{eq:dec-D_Au}
		\D_A u = \frac{\d(\abs{u})}{\abs{u}}u + \frac{j(u,\,A)}{\abs{u}^2}iu
	\end{equation}
	a.e. on the set $\{u \neq 0\}$.
\end{lemma}

Below, we collect some useful consequences of the equations 
\eqref{EL-u}, \eqref{EL-A}, and~\eqref{London} satisfied by 
critical points of the functionals $G_\eps$. All the results 
in this section are gauge-invariant and will be used repeatedly 
in the rest of the paper.

\paragraph{$L^\infty$-bounds.}

Let~$(u_\eps, \, A_\eps)\in \mathcal{E}$ be a critical point of~$G_\eps$.
Elliptic regularity theory implies that,
for each~$\eps > 0$, the pair~$(u_\eps, \, A_\eps)$
is smooth up to a suitable gauge transformation 
(see e.g.~\cite{JaffeTaubes} or~\cite[Appendix]{PigatiStern}). 
In particular, gauge-invariant quantities
such as~$\abs{u_\eps}$ or~$\abs{\D_{A_\eps} u_\eps}$ are continuous 
(and their squares are smooth).
Below, we prove several {\em a priori} estimates on $u_\eps$, 
$\D_{A_\eps}u_\eps$, and $F_{A_\eps}$. We will apply these estimates 
to obtain compactness results. 
For the reader's convenience, we provide 
self-contained proofs of the estimates we need.
We recall that~$\mathcal{E} := (W^{1,2}\cap L^\infty)(M, \, E)
\times W^{1,2}(M, \, \T^*M)$.
First, we prove that
\begin{prop}\label{prop:L-infty-bound-u}
Let $\eps > 0$ and let $(u_\eps,\,A_\eps) \in (W^{1,2}\cap L^\infty)(M,\,E) \times W^{1,2}(M,\, \T^*M)$ be any critical point of $G_\eps$.
Then, $\abs{u_\eps}$ satisfies
\begin{equation} \label{EL-abs}
 -\frac{1}{2}\Delta(\abs{u_\eps}^2) 
 + \frac{1}{\eps^2} \left(\abs{u_\eps}^2 - 1\right) \abs{u_\eps}^2
 + \abs{\D_{A_\eps} u_\eps}^2 = 0
\end{equation}
in the sense of distributions on $M$. As a consequence, 
\begin{equation} \label{maxprinc}
 \abs{u_\eps} \leq 1 \qquad \textrm{a.e. in } M.
\end{equation}
\end{prop}

\begin{remark}\label{rk:test-functions}
	The assumption~$u_\eps\in (W^{1,2}\cap L^\infty)(M, \, E)$
	implies that~$\abs{u_\eps}^2\in W^{1,2}(M)$. Moreover,
	since $G_\eps(u_\eps,\,A_\eps) < +\infty$, we have 
	$\abs{u_\eps} \in L^4(M)$ and $\abs{\D_{A_\eps} u_\eps}^2 \in L^1(M)$.
	In particular, it makes sense to test~\eqref{EL-abs} 
	against test functions in $(L^\infty \cap W^{1,2})(M)$.
\end{remark}

In order to prove~\eqref{maxprinc}, we write the equation for~$\abs{u_\eps}$, 
as in~\cite[Proposition~2.4]{Orlandi},
then apply the maximum principle. 
To this purpose, we first establish the following auxiliary result.
\begin{lemma}\label{lemma:maxprinc1}
For any section $u \in L^2(M,\,E)$ and any 1-form $A$ on $M$, 
such that $\D_Au\in L^2(M, \, \T^*M\otimes E)$, there holds
\begin{equation} \label{maxprinc1}
 \begin{split}
  -\frac{1}{2}\Delta(\abs{u}^2)
  = \ip{\D^*_{A}\D_{A} u}{u}
   - \abs{\D_{A} u}^2
 \end{split}
\end{equation}
in the sense of distributions on $M$. 
\end{lemma}

\begin{proof}
 Let~$u\in L^2(M, \, E)$ be a section and
 let~$A$ be~$1$-form on~$M$, such that $\D_Au\in L^2(M, \, \T^*M\otimes E)$.
 For any smooth function~$\varphi\colon M\to\R$, we have
 \begin{equation*} 
  \begin{split}
   \int_M \ip{\D^*_A\D_A u}{u} \, \varphi \, \vol_g
    - \int_M \abs{\D_A u}^2 \, \varphi \, \vol_g
   &= \int_M \ip{\D_A u}{\D_A(u\varphi)} \, \vol_g
    - \int_M \abs{\D_A u}^2 \, \varphi \, \vol_g \\
   &= \int_M \ip{\D_A u}{\d\varphi\otimes u} \, \vol_g 
  \end{split}
 \end{equation*}
 Let~$x\in M$, and let~$\{\tau_1, \, \ldots, \, \tau_n\}$
 be an orthonormal basis of~$\T_x M$. The scalar product
 on~$\T^*M\otimes E$ is the one induced by the Hermitian form on~$E$
 and the Riemannian metric on~$M$; therefore, at the given point~$x$,
 we have
 \begin{equation} \label{maxprinc2}
   \ip{\D_A u}{\d\varphi\otimes u} 
   = \sum_{j=1}^{n} \ip{\D_{A \, \tau_j} u}{u} \,
    \frac{\partial\varphi}{\partial\tau_j} 
 \end{equation}
 However, the right-hand side of~\eqref{maxprinc2}
 is also the scalar product between the (real-valued) $1$-forms
 $\ip{\D_A u}{u}$ and~$\d\varphi$, evaluated at~$x$.
 As a consequence, we have
 \begin{equation*} 
  \begin{split}
   \int_M \left(\ip{\D^*_A\D_A u}{u} - \abs{\D_A u}^2\right) \varphi \, \vol_g
   = \int_M \ip{\ip{\D_A u}{u}\!}{\d\varphi} \, \vol_g
   &= \int_M \varphi \, \d^*\ip{\D_A u}{u} \, \vol_g
  \end{split}
 \end{equation*}
 Since~$\varphi$ is arbitrary, we deduce $\ip{\D^*_A\D_A u}{u} - \abs{\D_A u}^2
 = \d^*\ip{\D_A u}{u}$. Finally, as $\ip{\D_A u}{u} \in L^1(M,\,\T^*M)$ 
 and the connection~$\D_A$ is compatible 
 with the Hermitian form on~$E$, we have
 $\ip{\D_A u}{u} = \d(\abs{u}^2/2)$ in the sense of distributions 
 and hence,
 \[
  \ip{\D^*_A\D_A u}{u} - \abs{\D_A u}^2
  = \frac{1}{2}\d^*\d\left(\abs{u}^2\right)
  = - \frac{1}{2}\Delta\left(\abs{u}^2\right)
 \]
 in the sense of distributions on $M$. The conclusion follows.
\end{proof}

\begin{proof}[Proof of Proposition~\ref{prop:L-infty-bound-u}]
Since $G_\eps(u_\eps,\,A_\eps) < +\infty$, 
it follows that $\D_{A_\eps} u_\eps \in L^2(M,\,\T^*M \otimes E)$. 
By Lemma~\ref{lemma:maxprinc1}, $(u_\eps,\,A_\eps)$ satisfies \eqref{maxprinc1} 
in the sense of distributions on $M$. 
On the other hand, being a critical point of 
$G_\eps$, $(u_\eps,\,A_\eps)$ satisfies 
\eqref{EL-u}, too. Combining \eqref{maxprinc1} with \eqref{EL-u}, 
we immediately obtain~\eqref{EL-abs}. 
Then, taking Remark~\ref{rk:test-functions} into 
account,~\eqref{maxprinc} follows 
by testing~\eqref{EL-abs} against 
$(\abs{u_\eps}^2-1)^+ := \max\{\abs{u_\eps}^2-1, \, 0\}$.
\end{proof}

Next, we prove a $L^\infty$-bound for $\D_{A_\eps} u_\eps$. 
Precisely, 
\begin{lemma} \label{lemma:L-infty-bound-D_Au}
Let $\eps > 0$ and let $(u_\eps,\,A_\eps) \in (W^{1,2}\cap L^\infty)(M,\,E) \times W^{1,2}(M,\, \T^*M)$ be a critical point of $G_\eps$.
Then,
	\begin{equation}\label{eq:L-infty-bound-D_Au}
		\norm{\D_{A_\eps} u_\eps}_{L^\infty(M)} \lesssim \eps^{-1}.
	\end{equation}
\end{lemma}

\begin{proof}
	Let $\eps > 0$ be arbitrary and 
	let $(u_\eps,\,A_\eps)$ be any critical
	point of $G_\eps$. In particular,  
	$\D_{A_\eps} u_\eps \in L^2(M,\,\T^*M \otimes E)$, whence almost every 
	$x_0 \in M$ is a Lebesgue point for $\abs{\D_{A_\eps} u_\eps}^2$. 
	Testing~\eqref{EL-u} against any $v \in W^{1,2}(M,\,E)$, we get
	\begin{equation}\label{eq:EL-u-integral}
		\int_M \ip{\D_{A_\eps} u_\eps}{\D_{A_\eps} v}\,\vol_g 
		= \frac{1}{\eps^2} \int_M \ip{v}{u_\eps}(1-\abs{u_\eps}^2) \,\vol_g.
	\end{equation}	  
	Let~$x_0 \in M$ be a Lebesgue point for~$\abs{\D_{A_\eps} u_\eps}^2$.
	Let $\mathcal{B}_r(x_0)$ be the geodesic ball of  center $x_0$ 
	and radius $r > 0$, with $r$ smaller than the injectivity radius of $M$. 
	For any $s > 0$	such that $r+s$ is still smaller than the injectivity 
	radius of $M$, we pick $\varphi_s \in C^\infty(M)$ such that
	\begin{equation}\label{eq:phi_s}
		\spt(\varphi_s) \subset \mathcal{B}_{r+s}(x_0), \quad 0 \leq \varphi_s \leq 1, 
		\quad \varphi_s \equiv 1 \mbox{ on } \mathcal{B}_r(x_0), \quad 
		\abs{\d \varphi_s} \lesssim \frac{1}{s}. 
	\end{equation}	 
	Since $v := u_\eps \varphi_s$ belongs to $W^{1,2}(M,\,E)$, it 
	is admissible to test~\eqref{EL-u}. By \eqref{eq:EL-u-integral}, we get 
	\begin{multline}\label{eq:test-phi_s}
		\int_{\mathcal{B}_r(x_0)} \abs{\D_{A_\eps} u_\eps}^2 \,\vol_g 
		- \frac{1}{\eps^2} \int_{\mathcal{B}_s(x_0)} \varphi_s \abs{u_\eps}^2(1-\abs{u_\eps}^2) \, \vol_g \\
		= - \int_{\mathcal{B}_{r+s}(x_0) \setminus \mathcal{B}_r(x_0)} \ip{\D_{A_\eps} u_\eps}{\d\varphi_s \otimes u_\eps + \varphi_s \D_{A_\eps} u_\eps}\,\vol_g.
	\end{multline}
	Since $\norm{u_\eps}_{L^\infty(M)} \leq 1$ by 
	Proposition~\ref{prop:L-infty-bound-u}, \eqref{eq:phi_s} and Schwarz' 
	inequality imply that the last term at right hand side of~\eqref{eq:test-phi_s} tends 
	to zero as $s \to 0$. Hence, from \eqref{eq:test-phi_s} we obtain
	\[
	\begin{split}
		\int_{\mathcal{B}_r(x_0)} & \abs{\D_{A_\eps}u_\eps}^2\,\vol_g 
		- \frac{1}{\eps^2} \int_{\mathcal{B}_r(x_0)}\abs{u_\eps}^2(1-\abs{u_\eps}^2) \,\vol_g \\
		&\leq \liminf_{s \to 0} \frac{C}{s}
		\int_{\mathcal{B}_{r+s}(x_0) \setminus \mathcal{B}_r(x_0)} \abs{\D_{A_\eps}u_\eps} \,\vol_g \\
		&= \liminf_{s \to 0} \frac{C}{s} \int_{r}^{r+s} 
		\left( \int_{\partial \mathcal{B}_r(x_0)} 
		\abs{\D_{A_\eps}u_\eps} \,\d \mathcal{H}^{n-1} \right)\,\d t \\
		&=C \int_{\partial \mathcal{B}_r(x_0)} \abs{\D_{A_\eps} u_\eps} \,\d \mathcal{H}^{n-1} \\
		&\leq C (r^{n-1})^{1/2} \left( \int_{\partial \mathcal{B}_r(x_0)} 
		\abs{\D_{A_\eps}u_\eps}^2 \,\d \mathcal{H}^{n-1} \right)^{1/2}
	\end{split}
	\]
	Using once again $\norm{u_\eps}_{L^\infty(M)} \leq 1$, we estimate
	\begin{equation}\label{eq:est-loc-abs-cov-dev}
	\int_{\mathcal{B}_r(x_0)} \abs{\D_{A_\eps}u_\eps}^2\,\vol_g 
	\leq C r^n \eps^{-2} + C r^{n/2 -1/2} \left( \int_{\partial \mathcal{B}_r(x_0)} 
		\abs{\D_{A_\eps}u}^2 \,\d \mathcal{H}^{n-1} \right)^{1/2}
	\end{equation}
	Now, we define 
	\[
		y(r) \equiv y(x_0, r) := \int_{\mathcal{B}_r(x_0)} 
		\abs{\D_{A_\eps}u_\eps}^2 \,\vol_g.
	\]
	Recasting~\eqref{eq:est-loc-abs-cov-dev} in terms of $y$ and rearranging, 
	we see that have shown
	\[
		C_* y(r) \leq r^n \eps^{-2} + r^{n/2-1/2}(y^\prime(r))^{1/2},
	\]
	for some constant $C_* > 0$, whence
	\begin{equation}\label{eq:y}
		y^\prime(r) \geq \frac{1}{r^{n-1}}\left( \left( C_* y(r) - r^n \eps^{-2} \right)^{+} \right)^2.
	\end{equation}
	Suppose now
	\begin{equation}\label{eq:y-contradiction}
		C_* y(r) - r^n \eps^{-2} \geq \frac{C_*}{2} y(r) \quad \mbox{for any } 
		r > 0 \mbox{ small enough}.
	\end{equation}
	Then, for~$r>0$ small enough we obtain
	\[
		y'(r) \geq \frac{C}{r^{n-1}} y(r)^2,
	\]
	which implies 
	\[
		\frac{1}{y(r)} \geq \frac{1}{y(R)} + C \left( \frac{1}{r^{n-2}} 
		- \frac{1}{R^{n-2}}\right) \geq \frac{C}{r^{n-2}}
	\]
	for any $0 < r < R$, with $R > 0$ small enough (and, perhaps, a different 
	constant $C > 0$). In particular,
	\begin{equation}\label{eq:y-ineq-1}
		y(r) \leq C_1 r^{n-2} \quad \mbox{for any } r > 0 \mbox{ small enough}
	\end{equation}
	for some constant $C_1 > 0$; but~\eqref{eq:y-contradiction} implies 
	\begin{equation}\label{eq:y-ineq-2}
		y(r) \geq C_2 r^n \eps^{-2} \quad \mbox{for any } r > 0 \mbox{ small enough}. 
	\end{equation}
	Inequality~\eqref{eq:y-ineq-2} contradicts~\eqref{eq:y-ineq-1} for $0 < r \ll \eps$. 
	Therefore, \eqref{eq:y-contradiction} does not hold; i.e., there is a 
	sequence $r_k \to 0$ such that 
	\[
		C_* y(r_k) - r^n_k \eps^{-2} < \frac{C_*}{2} y(r_k).
	\]
	Hence, by definition of $y(r)$, 
	\[
		\frac{C_*}{2} \int_{\mathcal{B}_{r_k}(x_0)} \abs{\D_{A_\eps}u}^2 
		\leq r^n_k \eps^{-2}, \quad \mbox{i.e.,} \quad 
		\frac{1}{r^n}\int_{\mathcal{B}_{r_k}(x_0)} \abs{\D_{A_\eps}u}^2 
		\lesssim \eps^{-2}.
	\]
	Taking the limit as $k \to +\infty$, we obtain
	\[
		\abs{\D_{A_\eps}u_\eps(x_0)}^2 \lesssim \eps^{-2},
	\]
	Since almost every $x_0 \in M$ is a Lebesgue point for 
	$\abs{\D_{A_\eps} u_\eps}^2$, the conclusion follows.
\end{proof}

Now, we provide a bound on the $L^\infty(M)$-norm of~$\{F_\eps\}$.

\begin{lemma} \label{lemma:Linfty-Feps}
 Let~$\eps > 0$ and let $(u_\eps,\,A_\eps) \in (W^{1,2}\cap L^\infty)(M,\,E) \times W^{1,2}(M,\, \T^*M)$ be a critical point of $G_\eps$. Then,
 \[
  \norm{F_\eps}_{L^\infty(M)} 
   \lesssim \eps^{-1}
 \]
\end{lemma}
\begin{proof}
 The Euler-Lagrange equation~\eqref{EL-A},
 together with Proposition~\ref{prop:L-infty-bound-u}
 and Lemma~\ref{lemma:L-infty-bound-D_Au}, implies
 \begin{equation} \label{Linfty-Feps}
  \norm{\d^* F_\eps}_{L^\infty(M)}
  = \norm{j(u_\eps, \, A_\eps)}_{L^\infty(M)}
  \leq \norm{u_\eps}_{L^\infty(M)} \norm{\D_{A_\eps} u_\eps}_{L^\infty(M)}
  \lesssim \eps^{-1}
 \end{equation}
 Let~$F_0$ be the curvature of the reference connection~$\D_0$.
 The difference~$F_\eps - F_0 = \d A_\eps$ is an exact
 form and, in particular, it is orthogonal to all harmonic~$2$-forms in~$M$.
 Then, for any~$p\in (n, \, +\infty)$, the Gaffney
 inequality can be written in the form
 \[
  \norm{F_\eps - F_0}_{W^{1,p}(M)} 
  \leq C_p \norm{\d^* F_\eps - \d^* F_0}_{L^p(M)}
 \]
 for some constant~$C_p$ that depend only on~$p$ and~$M$
 (see e.g.~\cite[Theorem~4.11]{IwaniecScottStroffolini}).
 Now, the form~$F_0$ is smooth and does not
 depend on~$\eps$, so
 \[
  \begin{split}
   \norm{F_\eps}_{W^{1,p}(M)} 
   \leq C_p \left(\norm{\d^* F_\eps}_{L^p(M)} + 1\right)
   &\leq C_p \, \eps^{-1}
  \end{split}
 \]
 The lemma follows by Sobolev embedding.
\end{proof}


\paragraph*{Compactness for gauge-invariant quantities.}

In this section, we apply the bounds obtained above,
as well as results from~\cite{CDO1},
to prove compactness results for gauge-invariant quantitites,
such as the norm~$\abs{u_\eps}$ and the pre-Jacobian~$j(u_\eps, \, A_\eps)$ 
of a sequence of critical points~$(u_\eps, \, A_\eps)$
that satisfies~\eqref{hp:logenergy}.

%

\begin{prop}\label{prop:conv-Lp-abs-u-eps}
	Let $\{(u_\eps,\,A_\eps)\} \subset (W^{1,2}\cap L^\infty)(M,\,E) \times W^{1,2}(M,\, \T^*M)$ 
	be a sequence of critical point of $G_\eps$ satisfying the logarithmic energy 
	bound~\eqref{hp:logenergy}. Then, for any $p$ with 
	$1 \leq p < 2$, there exists a constant $C_p$, depending only 
	on $\Lambda$, $p$, $M$, and $g$, such that
	\begin{equation}\label{eq:Lp-bound-abs-du-eps}
		\norm{\d(\abs{u_\eps})}_{L^p(M)} \leq C_p.
	\end{equation}
	Moreover, there holds
	\begin{equation}\label{eq:conv-Lp-rho-eps}
		\int_M \abs{\d (\abs{u_\eps})}^p \,\vol_g \to 0 
		\quad \mbox{as } \eps \to 0.
	\end{equation}
\end{prop}

\begin{proof}
	We follow the argument of~\cite[Lemma~X.13]{BBH}.
	 Set $\rho_\eps := \abs{u_\eps}$ and rewrite~\eqref{EL-abs} in the form
	 \begin{equation}\label{eq:EL-rho}
	 	\frac{1}{2} \Delta \rho_\eps^2 = 
	 	\frac{1}{\eps^2} \rho^2_\eps(\rho_\eps^2 -1 ) 
	 	+ \abs{\D_{A_\eps} u_\eps}^2.
	 \end{equation}
	 Next, for any $0 < \eps < 1$, we define:
	 \begin{equation}\label{eq:def-rho-M-N}
	 	\widetilde{\rho}_\eps := \max\left\{ \rho_\eps, 1 - \frac{1}{\abs{\log\eps}^2} \right\}, \quad 
	 	M_\eps := \left\{ x \in M : \widetilde{\rho_\eps}(x) = \rho_\eps(x)\right\}, \quad N_\eps := M \setminus M_\eps.
	 \end{equation}
	 Notice that on $M_\eps$ there holds 
	 $\rho_\eps = \abs{u_\eps} \geq 1 - \frac{1}{\abs{\log\eps}^2}$ 
	 for any $\eps \in (0,1)$. Moreover, as $\rho_\eps$ belongs to 
	 $(L^\infty \cap W^{1,2})(M)$, it follows that $\widetilde{\rho}_\eps$ 
	 belongs to $(L^\infty \cap W^{1,2})(M)$ as well, for any $\eps \in (0,1)$.  
	 Since $M$ is compact, also $1 - \widetilde{\rho}_\eps$ belongs to 
	 $(L^\infty \cap W^{1,2})(M)$ and,
	 by Remark~\ref{rk:test-functions},
	 we can use $1 - \widetilde{\rho}_\eps$ as a test function 
	 in~\eqref{eq:EL-rho}. Since $M$ is compact, 
	 $\widetilde{\rho}_\eps$ is bounded in $M$ and moreover constant 
	 in $N_\eps$, we obtain
	 \begin{equation}\label{eq:Lp-bound-comp1}
	 	\int_{M_\eps} \rho_\eps \abs{\d \rho_\eps}^2 \,\vol_g  \leq 
	 	2 \int_M (1-\widetilde{\rho}_\eps) \abs{\D_{A_\eps} u_\eps}^2\,\vol_g.
	 \end{equation}
	 Since $1 - \widetilde{\rho}_\eps \leq \frac{1}{\abs{\log\eps}^2}$, 
	 by~\eqref{eq:Lp-bound-comp1} and~\eqref{hp:logenergy} we infer
	 \begin{equation}\label{eq:Lp-bound-comp2}
	 	\int_{M_\eps} \rho_\eps \abs{\d \rho_\eps}^2 \,\vol_g 
	 	\lesssim \frac{1}{\abs{\log\eps}} \longrightarrow 0 
	 	\quad \mbox{as } \eps \to 0.
	 \end{equation}
	 Being $\rho_\eps = \abs{u_\eps} \geq 1 - \frac{1}{\abs{\log\eps}^2}$ 
	 on $M_\eps$, \eqref{eq:Lp-bound-comp2} gives
	 	\begin{equation}\label{eq:Lp-bound-comp2-bis}
	 		\norm{\d(\abs{u_\eps})}_{L^p(M_\eps)} \lesssim 1
	 	\end{equation}
	 	(where the implicit constant depends only on $\Lambda$, $p$, $M$, 
	 	and $g$) and
	 \begin{equation}\label{eq:Lp-bound-comp3}
	 	\int_{M_\eps} \abs{\d \rho_\eps}^2 \,\vol_g \longrightarrow 0 
	 	\quad \mbox{as } \eps \to 0.
	 \end{equation}
	
	Now, we show that the measure of the ``bad'' set $N_\eps$ decreases 
	 at a quantitative rate as $\eps \to 0$. Indeed, the logarithmic energy 
	 bound~\eqref{hp:logenergy} and \eqref{maxprinc} imply
	 \[
	 	\abs{\log \eps} \gtrsim \frac{1}{\eps^2} \int_{N_\eps} \left(1-\rho_\eps^2\right)^2 \,\vol_g \geq \frac{1}{\eps^2} \int_{N_\eps} (1-\rho_\eps)^2 \,\vol_g \geq \frac{\abs{N_\eps}}{\abs{\log \eps}^4 \eps^2},
	 \]
	 whence 
	 \begin{equation}\label{eq:Lp-bound-comp6}
	 	\abs{N_\eps} \lesssim \abs{\log \eps}^5 \eps^2.
	 \end{equation}
	By~\eqref{eq:dec-D_Au}, it is straightforward to see that, 
	for any $\eps > 0$,
	\begin{equation}\label{eq:kato}
		\abs{\d \rho_\eps} \leq \abs{\D_{A_\eps} u_\eps} 
		\quad \mbox{a.e. on } M.
	\end{equation}
	 Combining~\eqref{eq:Lp-bound-comp6} with \eqref{eq:L-infty-bound-D_Au} 
	 (which holds a.e. on $M$), we get
	 \begin{equation}\label{eq:Lp-bound-comp7}
	 	\norm{\D_{A_\eps} u_\eps}_{L^p(N_\eps)} \lesssim \eps^{\frac{2}{p}-1} \abs{\log\eps}^{\frac{5}{p}} \longrightarrow 0 \quad \mbox{as } \eps \to 0.
	 \end{equation}
	 Thus, by~\eqref{eq:Lp-bound-comp7} and~\eqref{eq:kato}, 
	 \begin{equation}\label{eq:Lp-bound-comp8}
	 	\int_{N_\eps} \abs{\d \rho_\eps}^p \,\vol_g \longrightarrow 0
	 	\quad \mbox{as } \eps \to 0
	 \end{equation}
	 for any $p$ with $1 \leq p < 2$. 
	 Moreover, along with~\eqref{eq:Lp-bound-comp2} and the definition of 
	 $M_\eps$, \eqref{eq:kato} and~\eqref{eq:Lp-bound-comp7} 
	 imply~\eqref{eq:Lp-bound-abs-du-eps}.
	 Finally, combining~\eqref{eq:Lp-bound-comp8} with~\eqref{eq:Lp-bound-comp3} 
	 yields~\eqref{eq:conv-Lp-rho-eps}, completing the proof of the 
	 proposition.
\end{proof}

By combining the compactness result in~\cite[Theorem~A]{CDO1}
with the properties of the London equation,
and reasoning along the lines of~\cite[Corollary B]{CDO1},
we can extract a subsequence and find limits~$J_*$, $F_*$
such that
\begin{equation} \label{compactnessF,J}
 J(u_\eps, \, A_\eps)\to \pi J_* \quad \textrm{strongly in } W^{-1,p}(M), \qquad
 F_\eps\to F_* \quad \textrm{strongly in } W^{1,p}(M)
\end{equation}
for any~$p$ with~$1\leq p < n/(n-1)$. In fact,
$F_*$ and~$J_*$ satisfy
\begin{equation} \label{limLondon}
 -\Delta F_* + F_* = 2\pi J_*
\end{equation}
that is the London equation in~\eqref{London*}.
We claim that 
the curvature energy grows less than
logaritmically in~$\eps$. This fact will be useful 
to our analysis of the properties of the energy concentration measure,
in Section~\ref{sec:varifold}.

\begin{lemma} \label{lemma:curvature}
 Let~$\{(u_\eps, \, A_\eps)\}_{\eps>0}
 \subset (W^{1,2}\cap L^\infty)(M, \, E)\times W^{1,2}(M, \, \T^*M)$
 be a sequence of finite-energy solutions of~\eqref{EL-u}--\eqref{EL-A}
 that satisfies~\eqref{hp:logenergy}. Then, 
 \[
  \lim_{\eps\to 0} \frac{1}{\abs{\log\eps}}\int_M \abs{F_\eps}^2 \vol_g = 0.
 \]
\end{lemma}
The proof of Lemma~\ref{lemma:curvature}
is very similar to arguments we used alredy
in~\cite[Section~4]{CDO1}.
\begin{proof}[Proof of Lemma~\ref{lemma:curvature}]
 The curvature form~$F_\eps$ is, by definition,
 closed. Therefore, by Hodge decomposition, we can write
 $F_\eps = \d\varphi_\eps + \xi_\eps$, where~$\xi_\eps$
 is a harmonic~$2$-form and~$\varphi_\eps$ is a co-exact $1$-form
 such that
 \begin{equation}
  \norm{\varphi_\eps}_{W^{1,2}(M)} 
   + \norm{\xi_\eps}_{L^2(M)} 
  \lesssim \norm{F_\eps}_{L^2(M)} 
   \stackrel{\eqref{hp:logenergy}}{\lesssim} \abs{\log\eps}^{1/2}
 \label{curvature1}
 \end{equation}
 (see all the discussion from~\cite[Proposition~A.14]{CDO1}).
 As the space of harmonic forms is finite-dimensional,
 the $W^{1,2}(M)$-norm of~$\xi_\eps$, too,
 is of order~$\abs{\log\eps}^{1/2}$ at most. 
 From~\eqref{EL-A}, we obtain
 \begin{equation*} 
  -\Delta\varphi_\eps = \d^*\d\varphi_\eps = \d^* F_\eps
  = j(u_\eps, \, A_\eps)
 \end{equation*}
 As~$\varphi_\eps$ is co-exact, hence
 orthogonal to all harmonic~$1$-forms, 
 elliptic regularity theory implies that
 the~$W^{2,2}(M)$-norm of~$\varphi_\eps$
 is bounded from above by the~$L^2(M)$-norm of~$-\Delta\varphi_\eps$.
 We deduce
 \begin{equation} \label{curvature3}
  \norm{\varphi_\eps}_{W^{2,2}(M)} 
  \lesssim \norm{j(u_\eps, \, A_\eps)}_{L^2(M)}
  \lesssim \norm{\D_{A_\eps} u_\eps}_{L^2(M)} \norm{u_\eps}_{L^\infty(M)}
  \lesssim \abs{\log\eps}^{1/2} 
 \end{equation}
 (the last inequality follows from~\eqref{hp:logenergy} and~\eqref{maxprinc}).
 Therefore, we have
 \begin{equation*} 
  \norm{F_\eps}_{W^{1,2}(M)} 
   \leq \norm{\varphi_\eps}_{W^{2,2}(M)}+ \norm{\xi_\eps}_{W^{1,2}(M)}
   \stackrel{\eqref{curvature1}, \, \eqref{curvature3}}{\lesssim}
    \abs{\log\eps}^{1/2}
 \end{equation*}
 and hence, by Sobolev embedding,
 \begin{equation} \label{curvature5}
  \norm{F_\eps}_{L^q(M)} \lesssim \abs{\log\eps}^{1/2}
  \qquad \textrm{with } q := 2^* = \frac{2n}{n-2}
 \end{equation}
 From~\eqref{compactnessF,J} and~\eqref{curvature5} we obtain,
 by interpolation,
 \begin{equation} \label{curvature6}
  \norm{F_\eps}_{L^2(M)} 
  \leq \norm{F_\eps}_{L^q(M)}^\alpha \norm{F_\eps}_{L^1(M)}^{1-\alpha} 
  \lesssim \abs{\log\eps}^{\alpha/2}
 \end{equation}
 with~$\alpha := n/(n+2) \in (0, \, 1)$. The lemma follows.
\end{proof}

Next, we notice that 
criticality and $L^p(M)$-compactness of $\{\d^* F_{A_\eps}\}$ for any $p$ 
with $1 \leq p < \frac{n}{n-1}$
(cf.~\cite[Proposition~3.11(i) and Remark~3.12]{CDO1}) imply compactness of 
the prejacobians $j(u_\eps,\,A_\eps)$ in $L^p(M,\,\T^*M)$, for any 
$1 \leq p < \frac{n}{n-1}$. 

\begin{lemma}\label{lemma:Lp-bound-prejac}
	Let 
	$\{(u_\eps,\,A_\eps)\} \subset 
	\mathcal{E}$ 
	be a sequence of critical points of $G_\eps$ 
	satisfying the logarithmic energy bound~\eqref{hp:logenergy}. 
	Then, there exists 
	$j_* \in L^p(M,\,\T^*M)$ for any $p$ with $1 \leq p < \frac{n}{n-1}$,  
	such that, 
	up to a (not relabelled) subsequence, there holds
	\begin{equation}\label{eq:Lp-conv-j_eps-j*}
		j(u_\eps,\,A_\eps) \longrightarrow j_* \quad \mbox{in } L^p(M, \,\T^*M)  
		\quad \mbox{as } \eps \to 0,
	\end{equation}
	for any $p$ with $1 \leq p < \frac{n}{n-1}$. In particular, there holds
	\begin{equation}\label{eq:Lp-bound-j_eps}
	  	\norm{j(u_\eps,\,A_\eps)}_{L^p(M)} \lesssim 1, 
	\end{equation}	 
	for any $p$ with $1 \leq p < \frac{n}{n-1}$, where the implicit constant in 
	front of the right-hand-side depends only on $\Lambda$, $p$, $M$ 
	and $g$.
\end{lemma}

\begin{proof}
	 Set, for short, $j_\eps := j(u_\eps,\,A_\eps)$ and 
	 $F_\eps := F_{A_\eps}$. 
	 To prove~\eqref{eq:Lp-conv-j_eps-j*}, 
	 it is enough to show that $\{ j_\eps\}$ contains
	 a Cauchy subsequence in (the Banach space) 
	 $L^p(M,\,\T^*M)$, for any $p$ with $1 \leq p < \frac{n}{n-1}$.
	 Due to \cite[Proposition~3.11(i) and 
	 Remark~3.12]{CDO1}, we can extract a subsequence~$\eps_k\to 0$
	 such that~$\{F_{\eps_k}\}$ is a Cauchy sequence in 
	 $W^{1,p}(M,\,\Lambda^2 \T^*M)$, for any $p$ with 
	 $1 \leq p < \frac{n}{n-1}$.
	 But then, by~\eqref{EL-A}, for any $k$, $m \in \N$ 
	 we have 
	 \[
	 	\norm{j_{\eps_k} - j_{\eps_m}}_{L^p(M)} 
	 	= \norm{\d^* (F_{\eps_k} - F_{\eps_m})}_{L^p(M)}.
	 \]
	 This proves~\eqref{eq:Lp-conv-j_eps-j*}.
	 If~\eqref{eq:Lp-bound-j_eps} were false,
	 there would be a (countable) sequence of critical
	 points~$\{(u_\eps, \, A_\eps)\}\subset\mathcal{E}$ 
	 that satisfy~\eqref{hp:logenergy} (for some~$\eps$-independent 
	 constant~$\Lambda$) and are such
	 that~$\norm{j(u_\eps, \, A_\eps)}_{L^p(M)}\to+\infty$ as~$\eps\to 0$.
	 This contradicts~\eqref{eq:Lp-conv-j_eps-j*}. 
	 Therefore, \eqref{eq:Lp-bound-j_eps} is proved.
\end{proof}

We apply Proposition~\ref{prop:conv-Lp-abs-u-eps} and 
Lemma~\ref{lemma:Lp-bound-prejac} to prove a uniform $L^p(M)$-bound 
on $\{\D_{A_\eps} u_\eps\}$. 

\begin{lemma}\label{lemma:bound-DA_eps-u_eps}
	Let $\{(u_\eps,\,A_\eps)\} \subset 
	\mathcal{E}$ 
	be a sequence of critical points of $G_\eps$ 
	satisfying the logarithmic energy bound~\eqref{hp:logenergy}. 
	Then, 
	for any $p$ with $1 \leq p < \frac{n}{n-1}$ 
	there exists a constant $C_p > 0$ 
	(depending on~$\Lambda$, $p$, $M$ and $g$) such that 
	\begin{equation}\label{eq:Lp-bound-D_Au}
		\norm{\D_{A_\eps} u_\eps}_{L^p(M)} \leq C_p	
	\end{equation}	
	for any $\eps > 0$. Moreover, there exists a (non-relabelled)
	subsequence such that
	\begin{equation}\label{eq:conv-D_A_eps-u_eps-j_*}
		\norm{\D_{A_\eps} u_\eps}_{L^p(M)} \to \norm{j_*}_{L^p(M)} 
		\quad \mbox{for any } 1 \leq p < \frac{n}{n-1},
		\quad \mbox{as } \to 0,
	\end{equation}
	where $j_*$ is given by Lemma~\ref{lemma:Lp-bound-prejac}.
\end{lemma}

\begin{proof}
	Define the sets $M_\eps$ and $N_\eps$ as in~\eqref{eq:def-rho-M-N}. 
	
	\medskip
	\noindent
	\emph{Proof of~\eqref{eq:Lp-bound-D_Au}.}
	 Let~$p$ be fixed, such that~$1 \leq p < \frac{n}{n - 1}$.
	 If~\eqref{eq:Lp-bound-D_Au} were false, we could
	 extract a (non-relabelled) subsequence)
	 in such a way that 
	 \begin{equation} \label{ciao1}
	  \norm{\D_{A_\eps} u_\eps}_{L^p(M)} \to +\infty
	  \qquad \textrm{as } \eps\to 0
	 \end{equation}
	 Let $\eps_2 > 0$ be such that 
	 $\frac{1}{\abs{\log \eps_2}^2} = \frac{1}{2}$ and notice that 
	 $\abs{u_\eps} > \frac{1}{2}$ on $M_\eps$ for any $\eps$ with 
	 $0 < \eps < \eps_2$. 
	 Then, by~\eqref{eq:dec-D_Au}, 
	 \[
	 	\abs{\D_{A_\eps} u_\eps} \leq 2 \left( \abs{\d(\abs{u_{\eps}})} + \abs{j(u_\eps,\,A_\eps)}) \right) \qquad \mbox{on } M_\eps,
	 \]
	 whence
	 \begin{equation}\label{eq:Lp-bound-D_Au-comp1}
	 	\abs{\D_{A_\eps} u_\eps}^p \leq 2^p \left( \abs{\d(\abs{u_{\eps}})}^p + \abs{j(u_\eps,\,A_\eps )}^p \right) \qquad \mbox{on } M_\eps,
	 \end{equation}
	 for any $\eps$ with $0 < \eps < \eps_2$. 
	 Combining~\eqref{eq:Lp-bound-D_Au-comp1} 
	 with~\eqref{eq:Lp-bound-abs-du-eps} and~\eqref{eq:Lp-bound-j_eps}, 
	 we obtain that for any $p$ with  
	 $1 \leq p < \frac{n}{n-1}$ there exists a constant $C_p^\prime > 0$ 
	 such that
	 \begin{equation}\label{eq:Lp-bound-comp9}
	 	\norm{\D_{A_\eps} u_\eps}_{L^p(M_\eps)} \leq C_p^\prime
	 \end{equation}
	 for any $\eps > 0$. 
	 By~\eqref{eq:Lp-bound-comp9} and \eqref{eq:Lp-bound-comp7}, 
	 we deduce that~$\D_{A_\eps}u_\eps$ is bounded in~$L^p(M)$.
	 This contradicts~\eqref{ciao1} and proves~\eqref{eq:Lp-bound-D_Au}.
	 
	 \medskip
	 \noindent
	 \emph{Proof of~\eqref{eq:conv-D_A_eps-u_eps-j_*}.}
	 By Lemma~\ref{lemma:Lp-bound-prejac}, \eqref{eq:Lp-bound-D_Au}, 
	 Lemma~\ref{lemma:conv-potential}, 
	 and the Dominated Convergence Theorem,
	 \begin{equation}\label{eq:j1}
	 \begin{split}
	 	\int_M \abs{j_*}^p \,\vol_g &\leq \liminf_{\eps \to 0} \int_{M_\eps} \frac{\abs{j(u_\eps,\,A_\eps)}^p}{\abs{u_\eps}^p}\,\vol_g \\
	 	&\leq\limsup_{\eps \to 0} \int_{M_\eps} \frac{\abs{j(u_\eps,\,A_\eps)}^p}{\abs{u_\eps}^p}\,\vol_g \\
	 	&\leq \limsup_{\eps \to 0} \left( 1 - \frac{1}{\abs{\log \eps}^2} \right) \int_{M_\eps} \abs{j(u_\eps,\,A_\eps)}^p\,\vol_g \\
	 	 &\leq \lim_{\eps \to 0} \int_M \abs{j(u_\eps,\,A_\eps)}^p\,\vol_g = \int_M \abs{j_*}^p\,\vol_g.
	 \end{split} 
	 \end{equation}
	 Thus, by \eqref{eq:dec-D_Au} and \eqref{eq:j1},
	 \begin{equation}
	 \begin{split}
	 	\norm{j_*}_{L^p(M)} &= \lim_{\eps \to 0} \left(\int_{M_\eps} \frac{\abs{j(u_\eps,\,A_\eps)}^p}{\abs{u_\eps}^p}\,\vol_g \right)^{1/p} \\
	 	&\leq \lim_{\eps \to 0} \left(\int_{M_\eps} \abs{\D_{A_\eps} u_\eps}^p\,\vol_g\right)^{1/p} \\
	 	&\leq \lim_{\eps \to 0} \left\{ \left(\int_{M_\eps} \abs{\d \rho_\eps}^p \,\vol_g\right)^{1/p} + \left( 1 - \frac{1}{\abs{\log\eps}^2} \right) \left(\int_{M_\eps} \abs{j(u_\eps,\,A_\eps)}^p\,\vol_g \right)^{1/p} \right\} \\
	 	&\stackrel{\eqref{eq:conv-Lp-rho-eps}}{=} \lim_{\eps \to 0} \left(\int_{M_\eps}  \abs{j(u_\eps,\,A_\eps)}^p\,\vol_g \right)^{1/p} \\
	 	&\stackrel{\eqref{eq:Lp-conv-j_eps-j*}}{=} \norm{ j_*}_{L^p(M)},
	 \end{split}
	\end{equation}
	for any $p$ with $1 \leq p < \frac{n}{n-1}$. 	 
	But, by \eqref{eq:Lp-bound-comp7} and the Dominated Convergence Theorem,
	\[
	\lim_{\eps \to 0} \left(\int_{M_\eps} \abs{\D_{A_\eps} u_\eps}^p\,\vol_g\right)^{1/p} = \lim_{\eps \to 0} \left(\int_{M} \abs{\D_{A_\eps} u_\eps}^p\,\vol_g\right)^{1/p},
	\]	 
	for any $p$ with $1 \leq p < \frac{n}{n-1}$, 
	hence \eqref{eq:conv-D_A_eps-u_eps-j_*} follows by~\eqref{eq:j1}.
\end{proof}


Finally, the following lemma is classical in Ginzburg-Landau theory.
\begin{lemma}\label{lemma:conv-potential}
	Let $\{ (u_\eps,\,A_\eps) \} \subset \mathcal{E}$
  be a sequence satisfying the logarithmic bound \eqref{hp:logenergy}. Then, 
  \begin{equation}\label{eq:conv-potential}
  	\lim_{\eps \to 0} \int_M \left(1-\abs{u_\eps}^2\right)^2\,\vol_g = 0.
  \end{equation}
  In particular, $\abs{u_\eps} \to 1$ in $L^p(M)$ as $\eps \to 0$, 
  for any finite $p \geq 1$.
\end{lemma}
\begin{proof}
 Equation~\eqref{eq:conv-potential} is an immediate consequence of~\eqref{hp:logenergy}. The $L^p(M)$-convergence $\abs{u_\eps}\to 1$
 follows by~\eqref{hp:logenergy} and~\eqref{maxprinc}, by interpolation.
\end{proof}


\section{Compactness up to gauge equivalence} \label{sec:compactness}

In this section, we prove Theorem~\ref{th:uA}. 
We assume that
$\{(u_\eps,\,A_\eps)\} \subset \mathcal{E}$
is a sequence of critical points of $G_\eps$ satisfying \eqref{hp:logenergy}  
and that each $A_\eps$ takes the form~\eqref{globalCoulomb}, i.e., 
that each $A_\eps$ writes as 
\[
	A_\eps = \d^* \psi_\eps + \zeta_\eps
\]
for an exact 2-form $\psi_\eps$ and a harmonic form $\zeta_\eps$ satisfying 
$\norm{\zeta_\eps}_{L^\infty(M)} \leq C_M$, where $C_M$ depends only on 
$M$.
As already mentioned in the Introduction, this is always possible. Indeed, a 
family of gauge transformations $\{\Phi_\eps\} \subset W^{2,2}(M,\,\S^1)$ 
with this property can be obtained by applying, for each $\eps > 0$, 
Lemma~2.10 in~\cite{CDO1}. As shown by (the proof of)
Proposition~3.1 in~\cite{CDO1}, if $A_\eps$ takes the 
form~\eqref{globalCoulomb}, we can find 
$A_* \in W^{2,p}(M,\,\T^*M)$ for any $p$ with 
$1 \leq p < \frac{n}{n-1}$ such that, possibly 
up to extraction of a (not relabeled) subsequence, 
\begin{equation}\label{eq:conv-gauged-A_eps-W^{2,p}}
	\Phi_\eps \cdot A_\eps \to A_* \quad \mbox{strongly in } W^{2,p}(M) \quad \mbox{as } \eps \to 0, \quad \mbox{for any } p \mbox{ with } 1 \leq p < \frac{n}{n-1}.
\end{equation}
Thus, to prove Theorem~\ref{th:uA}, 
it is enough to prove that there 
exists $u_*$ belonging to $W^{1,p}(M,\,E)$ for any $p$ with $1 \leq p < \frac{n}{n-1}$, 
such that, up to a subsequence, 
\begin{equation}\label{eq:conv-gauged-u_eps-W^{1,p}}
	\Phi_\eps u_\eps \to u_* \quad \mbox{strongly in } W^{1,p}(M) \quad \mbox{as } \eps \to 0, \quad \mbox{for any } p \mbox{ with } 1 \leq p < \frac{n}{n-1}.
\end{equation}
In~\eqref{eq:conv-gauged-A_eps-W^{2,p}}, \eqref{eq:conv-gauged-u_eps-W^{1,p}}, $\Phi_\eps u_\eps$, $\Phi_\eps \cdot A_\eps$ are defined as in~\eqref{eq:gauge-transf}; i.e., $\Phi_\eps u_\eps$  
is given by fibre-wise action of $\S^1$ on $E$ and 
$\Phi_\eps \cdot A_\eps = A_\eps -i \Phi_\eps^{-1} \d \Phi_\eps$.
Since in the rest of this section we will always work 
in the gauge~\eqref{globalCoulomb}, we will to  omit the gauge transformations 
$\Phi_\eps$ from the notation; i.e., while writing $A_\eps$ and $u_\eps$, we 
will actually mean $\Phi_\eps \cdot A_\eps$ and $\Phi_\eps u_\eps$, 
respectively.
 

The following corollary is a straightforward consequence of Lemma~\ref{lemma:bound-DA_eps-u_eps}.
\begin{corollary}\label{cor:conv-abs-u_eps}
	Let 
	$\{(u_\eps,\,A_\eps)\} \subset \mathcal{E}$ 
	be a sequence of critical points of $G_\eps$ 
	satisfying the logarithmic energy bound~\eqref{hp:logenergy}. 
	Assume, in addition, that~\eqref{globalCoulomb} holds.
	Then, there exists a section $u_* \in W^{1,p}(M,\,E)$ 
	for any $p$ with $1 \leq p < \frac{n}{n-1}$ satisfying 
	$\abs{u_*} = 1$ almost everywhere on $M$ such that, 
	up to extraction of a (not relabelled) subsequence, for any $p$ with
	$1 \leq p < \frac{n}{n-1}$ there holds
	\begin{align}
		& u_\eps \rightharpoonup u_*  \quad \mbox{weakly in } W^{1,p}(M), \label{eq:weak-conv-u_eps} \\
		& u_\eps \to u_* \quad \mbox{strongly in } L^p(M), \label{eq:strong-Lp-conv-u_eps} \\
		& \abs{u_\eps} \to 1 \quad \mbox{strongly in } W^{1,p}(M), \label{eq:strong-W1p-con-abs-u_eps}
	\end{align}
	as $\eps \to 0$.
\end{corollary}

\begin{proof}
	By Lemma~\ref{lemma:bound-DA_eps-u_eps}, 
	\eqref{eq:conv-gauged-A_eps-W^{2,p}}, and the uniform bound 
	$\norm{u_\eps}_{L^\infty(M)} \leq 1$ (given by 
	Proposition~\ref{prop:L-infty-bound-u}), it follows that
	the sequence $\{\D_0 u_\eps\}$ is uniformly bounded in 
	$L^p(M,\,\T^*M \otimes E)$, for any $p$ with 
	$1 \leq p < \frac{n}{n-1}$. Consequently, the sequence $\{u_\eps\}$ 
	is uniformly bounded in 
	$W^{1,p}(M,\,E)$, for any $p$ with 
	$1 \leq p < \frac{n}{n-1}$. Hence,
	there exists 
	a section $u_* \in W^{1,p}(M,\,E)$ for any $p$ with 
	$1 \leq p < \frac{n}{n-1}$ such that up to a 
	(not relabelled) subsequence 
	\[
		u_\eps \rightharpoonup u_*  \quad \mbox{weakly in } W^{1,p}(M), 
		\quad \mbox{for any } p \mbox{ with } 1 \leq p < \frac{n}{n-1}
		\quad \mbox{as } \eps \to 0,
	\]
	i.e., \eqref{eq:weak-conv-u_eps}. 
	Moreover, since $M$ is compact, by the Rellich-Kondrachov
	theorem it follows that, up to another (not relabelled) subsequence,
	\[
		u_\eps \to u_* \quad \mbox{strongly in } L^p(M), 
		\quad \mbox{for any } p \mbox{ with } 1 \leq p < \frac{n}{n-1}
		\quad \mbox{as } \eps \to 0,
	\]
	i.e., \eqref{eq:strong-Lp-conv-u_eps}, 
	and so, up to still another (not relabelled) subsequence, 
	$\abs{u_\eps} \to \abs{u_*}$ pointwise almost everywhere on $M$ as $\eps \to 0$. 
	Then, Lemma~\ref{lemma:conv-potential} 	implies
	\begin{equation}\label{eq:abs-u*}
		\abs{u_*} = 1 \quad \mbox{a.e. on } M.
	\end{equation}
	By~\eqref{eq:abs-u*}, $\d(\abs{u_*}) = 0$ almost everywhere on $M$. Therefore, 
	\[
		\norm{ \d(\abs{u_*}) - \d(\abs{u_\eps})}_{L^p(M)} 
		= \norm{\d(\abs{u_\eps})}_{L^p(M)} \stackrel{\eqref{eq:conv-Lp-rho-eps}}{\longrightarrow} 0 
		\quad \mbox{as } \eps \to 0,
	\]
	for any $p$ with $1 \leq p < \frac{n}{n-1}$. Along 
	with~\eqref{eq:strong-Lp-conv-u_eps}, this 
	proves~\eqref{eq:strong-W1p-con-abs-u_eps}, and the corollary follows.
\end{proof}

\begin{remark}
	By \cite[Proposition~3.11 and Remark~3.12]{CDO1}, 
	the uniform bound~\eqref{maxprinc}, 
	Corollary~\ref{cor:conv-abs-u_eps}, and the dominated 
	convergence theorem, it follows that 
	\begin{equation}\label{eq:strong-conv-A_eps-u_eps}
		A_\eps u_\eps \to A_* u_* \quad \mbox{strongly in } 
		L^p(M) \quad \mbox{as } \eps \to 0, 
		\quad \mbox{for any } p \mbox{ with }
		1 \leq p < \frac{n}{n-1}.
	\end{equation}
	Indeed, for any $p$ with $1 \leq p < \frac{n}{n-1}$, we have
	\[
	\begin{split}
		\norm{A_\eps u_\eps - A_* u_*}_{L^p(M)} &=  
		\norm{ (A_\eps - A_*) u_\eps - A_*( u_* - u_\eps)}_{L^p(M)} \\
		&\stackrel{\eqref{maxprinc}}{\leq} \norm{A_\eps - A_*}_{L^p(M)} + 
		\norm{\abs{A_*}\abs{u_* - u_\eps}}_{L^p(M)}.
	\end{split}
	\]
	The first term at right hand side vanishes as $\eps \to 0$, because 
	$A_\eps \to A_*$ in $W^{2,p}(M)$, for any $p$ with $
	1 \leq p < \frac{n}{n-1}$. On the other hand, since 
	$\abs{A_*}\abs{u_* - u_\eps} \leq 2 \abs{A_*}$ 
	(again by~\eqref{maxprinc}) and 
	$u_\eps \to u_*$ pointwise a.e.~in $M$ 
	(by~\eqref{eq:strong-Lp-conv-u_eps}), 
	the dominated convergence theorem yields that 
	$\norm{\abs{A_*}\abs{u_* - u_\eps}}_{L^p(M)} \to 0$ as $\eps \to 0$,
	proving~\eqref{eq:strong-conv-A_eps-u_eps}. 
\end{remark}

We are now in the position to prove Theorem~\ref{th:uA}.
\begin{proof}[Proof of Theorem~\ref{th:uA}]
	Since $A_\eps$ satisfies \eqref{globalCoulomb},
	Equations~\eqref{eq:weak-conv-u_eps} 
	and~\eqref{eq:strong-conv-A_eps-u_eps} imply
	\begin{equation}\label{eq:weak-conv-D_A_eps-u_eps}
		\D_{A_\eps} u_\eps \rightharpoonup \D_{A_*} u_* \quad \mbox{weakly in }
		L^p(M) \quad \mbox{for any } p \mbox{ with } 1 \leq p < \frac{n}{n-1}
		\quad \mbox{as } \eps \to 0.
	\end{equation}
	We claim that
	\begin{equation}\label{eq:strong-conv-D_A_eps-u_eps}
		\int_M \abs{\D_{A_\eps} u_\eps}^p \,\vol_g \to 
		\int_M \abs{\D_{A_*} u_*}^p \,\vol_g \quad \mbox{for any } p \mbox{ with }
		1 \leq p < \frac{n}{n-1}, \quad \mbox{as } \eps \to 0,
	\end{equation}
	The condition~\eqref{eq:strong-conv-D_A_eps-u_eps} 
	implies the strong convergence 
	$\D_{A_\eps} u_\eps \to \D_{A_*} u_*$ in $L^p(M)$ for any 
	$1 \leq p < \frac{n}{n-1}$,
	and, in view of~\eqref{eq:strong-conv-A_eps-u_eps}, that 
	$\D_0 u_\eps \to \D_0 u_*$ strongly in $L^p(M)$, for any 
	$1 \leq p < \frac{n}{n-1}$. 
	
	Towards the proof of~\eqref{eq:strong-conv-D_A_eps-u_eps}, 
	we observe that, 
	by Lemma~\ref{lemma:dec-D_Au} and since 
	$\abs{u_*} = 1$ a.e. on $M$ (by 
	Corollary~\ref{cor:conv-abs-u_eps}), there holds
	\[
		\D_{A_*} u_* = i j(u_*,\,A_*) u_* \quad \mbox{a.e. on } M,
	\]
	where $j(u_*,\,A_*) = \ip{\D_{A_*} u_*}{iu_*}$. 
	On the other hand, Equations~\eqref{eq:weak-conv-D_A_eps-u_eps}, 
	\eqref{eq:strong-Lp-conv-u_eps}, and the uniform 
	bound $\norm{\D_{A_\eps}u_\eps}_{L^p(M)} \leq C_p$ given, 
	for any $p$ with $1 \leq p < \frac{n}{n-1}$, by 
	Lemma~\ref{lemma:bound-DA_eps-u_eps}, yield
	\begin{equation} \label{conv-j}
		j(u_\eps,\,A_\eps) = \ip{\D_{A_\eps} u_\eps}{i u_\eps} 
		\rightharpoonup \ip{\D_{A_*} u_*}{i u_*} = j(u_*,\,A_*) 
	\end{equation}
	weakly in $L^p(M)$ for any $p$ with $1 \leq p < \frac{n}{n-1}$ as 
	$\eps \to 0$. By the uniqueness of the weak limit, we must have 
	$j(u_*,\,A_*) = j_*$, where $j_*$ is given by 
	Lemma~\ref{lemma:Lp-bound-prejac}. 
	Recalling~\eqref{eq:conv-D_A_eps-u_eps-j_*}, 
	we deduce~\eqref{eq:strong-conv-D_A_eps-u_eps}.
	
	By passing to the limit in the Euler-Lagrange equation~\eqref{EL-A},
	with the help of~\eqref{compactnessF,J} and~\eqref{conv-j},
	we see that 
	\begin{equation} \label{j*F*}
	 j(u_*, \, A_*) = \d^* F_*
	\end{equation}
	in the sense of distributions in~$M$ ---
	that is, Equation~\eqref{weak-A*harmonic} holds. 
	In order to complete the proof, it only remains to show
	that~$u_*$ and~$A_*$ are smooth in~$M\setminus\spt\mu_*$,
	where~$\mu_*$ is the limit of the rescaled energy densities,
	given by Theorem~\ref{th:energy_density}.
	By passing to the limit as~$\eps\to 0$ in~\eqref{globalCoulomb}
	(with the help of~\cite[Proposition~5.6]{Scott} 
	or, say, \cite[Proposition~A.13]{CDO1}),
	we see that~$A_*$ can be written in the form
	\begin{equation} \label{globalCOulomblimit}
	 A_* = \d^*\psi_* + \zeta_*,
	\end{equation}
	where~$\psi_*$ is an exact~$2$-form and~$\zeta_*$
	a harmonic $1$-form on~$M$. By taking the differential
	of both sides, we obtain
	\begin{equation*}
	-\Delta\psi_* = \d\d^*\psi_* = \d A_* = F_* - F_0,
	\end{equation*}
	where~$F_0$ is the curvature of the reference connection~$\D_0$.
	Now, $F_*$ is a solution of the London equation~\eqref{limLondon},
	whose right-hand side~$2\pi J_*$ is a measure
	concentrated on~$\spt\mu_*$, due to~\cite[Theorem~A]{CDO1}.
	By elliptic regularity, it follows that~$F_*$ and~$\psi_*$
	are smooth in~$M\setminus\spt\mu_*$
	and hence, $A_*$ is.
	Now, let~$B\csubset M\setminus\spt\mu_*$ be a ball.
	We claim that there holds
	\begin{equation} \label{localquadratic}
	 \int_B \abs{\D_{A_*} u_*}^2 \,  \vol_g < +\infty
	\end{equation}
	Equation~\eqref{localquadratic} is an almost immediate consequence
	of Proposition~\ref{prop:small-ball-small-energy},
	which we will prove later on in Section~\ref{sec:clearing-out},
	combined with a covering argument.
	As~$A_*$ is smooth in a neighbourhood of~$\bar{B}$,
	from~\eqref{localquadratic} we deduce that~$u_*\in W^{1,2}(B, \, E)$.
	Therefore, up to identifying~$u_{*|B}$ with a complex-valued map,
	we can write~$u_* = e^{i\theta_*}$ in~$B$ for some scalar 
	function~$\theta_*\in W^{1,2}(B, \, \R)$.
	(This claim follows by lifting results that were 
	originally proven in~\cite{BethuelZheng}; 
	see also~\cite[Theorem~1.1]{BrezisMironescu-book}.)
	The reference connection can locally be written 
	as~$\D_0 = \d - i\gamma_0$ on~$B$, for some (smooth)
	$1$-form~$\gamma_0$. Then, a direct computation shows that
	\[
	 j(u_*, \, A_*) = \d\theta_* - \gamma_0 - A_*
	\]
	However, $j(u_*, \, A_*)$ is smooth in~$B$, due to~\eqref{j*F*}.
	Therefore, $\d\theta_*$ is smooth in~$B$, 
	and hence~$u_* = e^{i\theta_*}$ is.
	This completes the proof of the theorem.
\end{proof}

\begin{remark} \label{rk:strongA*harmonic}
 Let~$\varphi\in C^\infty_{\mathrm{c}}(M\setminus\spt\mu_*)$
 be a test function. From~\eqref{j*F*}, we know that 
 $\d^*j(u_*, \, A_*) = 0$ and hence, by integrating
 by parts,
 \[
  \begin{split}
   0 = \int_M \ip{j(u_*, \, A_*)}{\d\varphi} \, \vol_g
   &= \int_M \ip{\D_{A_*} u_*}{\d\varphi\otimes iu_*} \, \vol_g \\
   &= \int_M \ip{\D_{A_*} u_*}{\D_{A_*}(i\varphi u_*)} \, \vol_g
   = \int_M \ip{\D_{A_*}^*\D_{A_*} u_*}{i u_*}\varphi \, \vol_g
  \end{split}
 \]
 where~$\D_{A_*}^*$ is the~$L^2$-adjoint of~$\D_{A_*}$.
 As~$\varphi$ is arbitrary, it follows
 that~$\D_{A_*}^*\D_{A_*} u_* = \lambda u_*$ in~$M\setminus\spt\mu_*$,
 for some (real-valued) function~$\lambda\colon M\setminus\spt\mu_*\to\R$.
 Since~$\abs{u_*} = 1$, we obtain that $\lambda = \ip{\D_{A_*}^*\D_{A_*} u_*}{u_*}$ and hence, by Lemma~\ref{lemma:maxprinc1},
 \[
  \D_{A_*}^*\D_{A_*} u_* = \abs{\D_{A_*} u_*}^2 u_*
  \qquad \textrm{in } M\setminus\spt\mu_*.
 \]
 In other words, $u_*$ is an $A_*$-harmonic
 unit section away from the singular set~$\spt\mu_*$.
\end{remark}


\section{Monotonicity formula}\label{sec:monotonicity}

This section is devoted to the proof of the monotonicity formula, 
Theorem~\ref{thm:monotonicity}. 
For the reader's convenience, we first recall some notation.
Let~$(u_\eps, \, A_\eps)\in\mathcal{E}$
be a critical point of~$G_\eps$, fixed once and
for all throughout this section. For~$x_0\in M$
and~$r\in (0, \, \inj(M))$, we define
\begin{equation}\label{eq:def-E-rho}
		E_\eps(x_0,\,\rho) := 
		\int_{\mathcal{B}_\rho(x_0)} e_\eps(u_\eps,\,A_\eps)\,\vol_g.
\end{equation} 
Moreover, if $\nu$ is the exterior unit normal field 
to $\partial\mathcal{B}_\rho(x_0)$, we set
\begin{equation}\label{eq:def-X-rho}
	X_\eps(x_0,\,\rho) := \int_{\partial \mathcal{B}_\rho(x_0)} 
	\left( \abs{\D_{A_\eps,\nu} u_\eps}^2 + \abs{{\rm i}_\nu F_\eps}^2 \right) \vol_{\hat{g}}
	+ \frac{1}{2\eps^2\, \rho} \int_{\mathcal{B}_\rho(x_0)} \left(1 - \abs{u_\eps}^2\right)^2 \vol_g
\end{equation}
where~$F_\eps := F_{A_\eps}$,
$\mathrm{i}_\nu$ denotes the contraction against~$\nu$
(as defined in~\eqref{innerprod}) 
and~$\hat{g}$ is the metric induced by $g$ on 
$\partial \mathcal{B}_\rho(x_0)$.

We will achieve the proof of Theorem~\ref{thm:monotonicity}
through a series of lemmas. The first one is a well-known 
consequence of the inner variation equation for critical points
(see e.g.~\cite[Section~4]{PigatiStern}).

\begin{lemma}\label{lemma:PS-ineq}
	Let $(u_\eps,\,A_\eps) \in \mathcal{E}$ be 
	any critical point of $G_\eps$. For any $x_0 \in M$ and any 
	$\rho \in (0,\,\inj(M))$, there holds 
	\begin{equation}\label{eq:PS-ineq}
		\rho \int_{\partial \mathcal{B}_\rho(x_0)} 
		e_\eps(u_\eps,\,A_\eps) \,\vol_{\hat{g}} 
		\geq \rho X_\eps(x_0,\,\rho) 
		+(n-2- C(M) \rho^2) E_\eps(x_0,\,\rho)\\ 
		-\int_{\mathcal{B}_\rho(x_0)} \abs{F_\eps}^2 \,\vol_g 
	\end{equation}	 
	where $C(M) > 0$ is a constant depending only on $M$ and $g$.
\end{lemma}
Lemma~\ref{lemma:PS-ineq} follows by the
very same computations contained in, e.g., \cite{PigatiStern}
(cf.~the computations starting at the beginning of Section~4
and ending right above Equation~(4.6) there).
Therefore, we skip its proof.

The second lemma provides an estimate on the term 
$\int_{\mathcal{B}_\rho(x_0)} \abs{F_\eps}^2 \,\vol_g$ in \eqref{eq:PS-ineq}.

\begin{lemma}\label{lemma:est-Feps}
	There exists $\bar{\rho} \in (0, \inj(M))$ and a constant $\lambda > 0$, 
	both depending only on $M$ and $g$, such that the following holds. If 
	$(u_\eps,\,A_\eps) \in \mathcal{E}$ is 
	any critical point of $G_\eps$, then for
	any $x_0 \in M$ and any $\rho \in (0,\,\bar{\rho})$ there holds
	\begin{equation}\label{eq:est-Feps}
		\norm{F_\eps}_{L^2(\mathcal{B}_\rho(x_0))} 
		\leq K \rho \norm{\D_{A_\eps} u_\eps}_{L^2(\mathcal{B}_\rho(x_0))}
		+ \lambda \sqrt{\rho} \norm{{\rm i}_\nu F_\eps}_{L^2(\partial \mathcal{B}_\rho(x_0))},
	\end{equation}
	where $K > 0$ is a constant depending only on $M$ and $g$.
\end{lemma}

\begin{proof}
	The idea is to exploit both the Euler-Lagrange equations and gauge 
	transformations to rewrite 
	$\int_{\mathcal{B}_\rho(x_0)} \abs{F_\eps}^2 \,\vol_g$ 
	in terms of quantities that can be estimated 
	more easily. To this purpose, we proceed step by step.
	
	\medskip
	\noindent
	\emph{Step~1.}
	Let $(u_\eps,\,A_\eps) \in (W^{1,2}\cap L^\infty)(M,\,E) \times W^{1,2}(M,\,\T^* M)$ 
	be any critical point of $G_\eps$. Then, the curvature 
	$F_\eps := F_{A_\eps}$ satisfies the Euler-Lagrange equation 
	\begin{equation}\label{eq:Feps-comp0}
		\d^* F_\eps = j(u_\eps, \, A_\eps)
	\end{equation}
	in the sense of distributions on $M$, and hence in particular on 
	$\mathcal{B}_\rho(x_0)$, for any fixed $x_0 \in M$. 
	Recall that, by definition, $\D_{A_\eps} = \D_0 -i A_\eps$, and that 
	locally we may write 
	$\D_0 = \d - i \gamma_0$, where $\gamma_0$ is a closed 1-form. Thus, on the 
	geodesic ball $\mathcal{B}_\rho(x_0)$ we can write 
	$\D_{A_\eps} = \d - i \widetilde{A}_\eps$, where 
	$\widetilde{A}_\eps := A_\eps + \gamma_0$, 
	and it holds $\d \widetilde{A}_\eps = F_\eps$. 
	\textbf{From now on, we shall work locally and, for convenience, 
	we will drop the tilde, writing $A_\eps$ to mean, actually, $\widetilde{A}_\eps$.}
	
	We fix Coulomb gauge in $\mathcal{B}_\rho(x_0)$ (without changing
	notation), so that $A_\eps$ satisfies 
	\begin{equation}\label{eq:Feps-comp1}
		\left\{
			\begin{aligned}
				& \d^* A_\eps = 0 & & \mbox{in } \mathcal{B}_\rho(x_0), \\
				& {\rm i}_\nu A_\eps = 0 & & \mbox{on } \partial \mathcal{B}_\rho(x_0), \\
				& \d A_\eps = F_\eps.
			\end{aligned}
			\right.
	\end{equation}
	(Such a gauge exists: see, e.g.~\cite[Proposition~2.1]{Orlandi}.)
	Moreover, $A_\eps$ is smooth up to the boundary of $\mathcal{B}_\rho(x_0)$, 
	and hence so is $F_\eps$ (recall that $F_\eps$ is gauge-invariant, so it 
	is actually globally smooth). 
	
	\medskip
	\noindent
	\emph{Step~2.}
	Using~\eqref{eq:Feps-comp1}, we write
	\[
		\abs{F_\eps}^2 \,\vol_g = \ip{F_\eps}{F_\eps} = \ip{F_\eps}{\d A_\eps}.
	\]
	Then, by the well-known integration-by-parts 
	formula for differential forms (see, for instance, \cite[(2.22) 
	and~(2.34)]{IwaniecScottStroffolini}), we obtain
	\begin{equation}\label{eq:Feps-comp2}
		\int_{\mathcal{B}_\rho(x_0)} \abs{F_\eps}^2 \,\vol_g = 
		\int_{\mathcal{B}_\rho(x_0)} \ip{j(u_\eps,\,A_\eps)}{A_\eps}
		+\int_{\partial \mathcal{B}_\rho(x_0)} A_\eps \wedge 
		(\star F_\eps)_{\T}
	\end{equation}
	where $(\star F_\eps)_{\T}$ denotes the part of $\star F_\eps$ tangent 
	to $\partial \mathcal{B}_{\rho}(x_0)$. We are going to estimate the 
	terms at the right hand side of~\eqref{eq:Feps-comp2}. 
	Before doing this, we notice that the Poincar\'e and Gaffney inequalities
	yield the estimate 
	\begin{equation}\label{eq:Feps-comp3}
		\norm{A_\eps}_{L^2(\mathcal{B}_\rho(x_0))} \lesssim 
		\rho \norm{F_\eps}_{L^2(\mathcal{B}_\rho(x_0))},
	\end{equation}
	up to constant depending only on $M$ and $g$
	(see Proposition~\ref{prop:gaffney-rho} in the appendix for more details).
	
	\medskip
	\noindent
	\emph{Step~3.}
	Let $\varphi = (x^1,\dots,x^n) : \mathcal{B}_{\rho}(x_0) \to \R^n$ be 
	coordinates so that $\d{x^1}, \dots, \d{x^{n-1}}$ are 1-forms tangent to 
	$\partial \mathcal{B}_\rho(x_0)$ and $\d{x^n}$ is normal to 
	$\mathcal{B}_\rho(x_0)$, in the sense of~\cite{IwaniecScottStroffolini}. 
	We may further assume that $\{ \d{x^1},\dots,\d{x^n} \}$ is an orthonormal 
	frame for $\T^*_p M$ at some \emph{prescribed} $p \in \partial \mathcal{B}_\rho(x_0)$.
	(The existence of such 	coordinate systems is well-known; see, for instance, 
	\cite[p.~49]{IwaniecScottStroffolini}.) Then, we can split
	\begin{equation}\label{eq:Feps-comp4}
		\star F_\eps = \star \left(\sum_{1 \leq i < j \leq n-1} F_{ij} 
		\d x^i \wedge \d x^j \right)
		+ \star \left(\sum_{j=1}^{n-1} F_{jn} \d x^j \wedge \d x^n \right)
	\end{equation}
	In particular, at $p$,
	\begin{equation}\label{eq:Feps-comp5}
	\begin{split}
	\star F_\eps(p) =  &\sum_{1 \leq i < j \leq n-1} F_{ij}(p) 
	\, \reallywidehat{\d x^i_p \wedge d x^j_p} 
	+ \sum_{j=1}^{n-1} F_{jn}(p) 
	\, \reallywidehat{\d x^j_p \wedge d x^n_p},
	\end{split}
	\end{equation}
	where the caret denotes omission
	(i.e., $\reallywidehat{\d x^i_p \wedge d x^j_p}$
	is the wedge product of~$\d x^k_p$ over all indices~$k$,
	in increasing order, except~$k=i$ and~$k=j$).
	Thus, the first terms at right hand side of~\eqref{eq:Feps-comp5} is the 
	normal part of $\star F_\eps(p)$ to $\partial \mathcal{B}_\rho(x_0)$ 
	while the second one is its tangent part, and we have 
	\begin{equation}\label{eq:Feps-comp6}
		\sum_{j=1}^{n-1} F_{jn}(p) 
	\, \reallywidehat{\d x^j_p \wedge \d x^n_p} 
	= \star \left( {\rm i}_\nu F_\eps \wedge \nu^\flat \right)_p,
	\end{equation}
	Since $p$ can be prescribed arbitrarily on 
	$\partial\mathcal{B}_\rho(x_0)$ and $\star F_\eps$ is invariant under 
	changes of coordinates, we deduce that
	\begin{equation}\label{eq:Feps-comp6b}
		( \star F_\eps )_{\rm T} = \star \left( {\rm i}_\nu F_\eps \wedge \nu^\flat \right)
	\end{equation}
	on $\partial \mathcal{B}_\rho(x_0)$.
	
	\medskip
	\noindent
	\emph{Step~4.}
	By~\eqref{eq:Feps-comp6b}, we obtain
	\begin{equation}\label{eq:Feps-comp7}
	\begin{split}
		\int_{\partial \mathcal{B}_\rho(x_0)} A_\eps \wedge (\star F_\eps)_{\T} 
		&= \int_{\partial \mathcal{B}_\rho(x_0)} \ip{A_\eps}{({\rm i}_\nu F_\eps) 
		\wedge \nu^\flat} \\
		&\leq \norm{A_\eps}_{L^2(\partial \mathcal{B}_\rho(x_0))} 
		\norm{{\rm i}_\nu F_\eps}_{L^2(\partial \mathcal{B}_\rho(x_0))}.
	\end{split}
	\end{equation}
	On the other hand, the trace inequality
	and Gaffney's inequality imply that	
	there exists a constant $\lambda > 0$, depending only on $(M,\,g)$,
	such that 
	\begin{equation}\label{eq:Feps-comp9}
		\norm{A_\eps}_{L^2(\partial \mathcal{B}_\rho(x_0))} \leq 
		\lambda \sqrt{\rho} \norm{F_\eps}_{L^2(\mathcal{B}_\rho(x_0))}.
	\end{equation}
	(see~\eqref{eq:trace} in the appendix for details).
	Recalling the easy estimate
	\begin{equation}\label{eq:Feps-comp10}
		\norm{j(u_\eps,\,A_\eps)}_{L^2(\mathcal{B}_\rho(x_0))} 
		\leq \norm{\D_{A_\eps} u_\eps}_{L^2(\mathcal{B}_\rho(x_0))},
	\end{equation}
	and combining~\eqref{eq:Feps-comp2}, \eqref{eq:Feps-comp7},
	\eqref{eq:Feps-comp9}, and~\eqref{eq:Feps-comp10} 
	with~\eqref{eq:Feps-comp2}, we obtain
	\begin{equation}\label{eq:Feps-comp11}
		\norm{F_\eps}^2_{L^2(\mathcal{B}_\rho(x_0))} \leq \left(K \rho 
		\norm{\D_{A_\eps} u_\eps}_{L^2(\mathcal{B}_\rho(x_0))} 
		+ \lambda \sqrt{\rho} \norm{{\rm i}_\nu F_\eps}_{L^2(\partial B_\rho(x_0))} \right)
		\norm{F_\eps}_{L^2(\mathcal{B}_\rho(x_0))} .
	\end{equation}
	As~\eqref{eq:est-Feps} is obvious if 
	$\norm{F_\eps}_{L^2(\mathcal{B}_\rho(x_0))} = 0$, we may assume 
	$\norm{F_\eps}_{L^2(\mathcal{B}_\rho(x_0))} > 0$. Then, dividing both sides 
	of~\eqref{eq:Feps-comp11} by 
	$\norm{F_\eps}_{L^2(\mathcal{B}_\rho(x_0))} $ yields the conclusion.
\end{proof}

\begin{prop}\label{prop:n-4-monotonicity}
	There exist a number $\alpha \in (0,\,2]$, a number 
	$\bar{\rho} \in (0,\,\inj(M))$ and a constant~$C>0$,
	depending only on $M$ and $g$, 
	such that the following holds. 
	If $(u_\eps,\,A_\eps) \in \mathcal{E}$ is  
	any critical point of $G_\eps$, then for any $x_0$ in $M$ and any 
	$\rho \in (0,\,\bar{\rho})$, there holds 
	\begin{equation}\label{eq:n-4-monotonicity}
		\frac{\d}{\d \rho}\left( \frac{e^{C\rho^2} 
		\, E_\eps(x_0, \, \rho)}{\rho^{n-4+\alpha}} \right) > 0
	\end{equation}
	where~$E_\eps(x_0, \, \rho)$ is defined by~\eqref{eq:def-E-rho}.
\end{prop}

\begin{proof}
	Assume, for the moment being, $\alpha$ is any (fixed) nonnegative number. 
	Using~\eqref{eq:PS-ineq}, we compute
	\[
	\begin{split}
		&\frac{\d}{\d \rho} \left( \frac{e^{C\rho^2} \, 
		E_\eps(x_0, \, \rho)}{\rho^{n-4+\alpha}} \right) \\
		&= \frac{e^{C\rho^2}}{\rho^{n-4+\alpha}} \left[ 
		\left( 2 C \rho - \frac{n-4+\alpha}{\rho} \right) E_\eps(x_0,\,\rho) 
		+ \int_{\partial \mathcal{B}_\rho (x_0)} e_\eps(u_\eps,\, A_\eps) \right] \\
		&\stackrel{\eqref{eq:PS-ineq}}{\geq} 
		\frac{e^{C\rho^2}}{\rho^{n-4+\alpha}} 
		\left\{ \vphantom{\int_{\mathcal{B}_\rho(x_0)} \frac{\left( 1 - \abs{u_\eps}^2 \right)^2}{4 \eps^2} \,\vol_g}
		 \left( 2 C - C(M) \right) \rho E_\eps(x_0,\,\rho) + X_\eps(x_0,\,\rho) \right. \\
		& \left. + \frac{2 -\alpha}{\rho} \int_{\mathcal{B}_\rho(x_0)} 
		\left( \frac{1}{2} \abs{\D_{A_\eps} u_\eps}^2
		+ \frac{\left( 1 - \abs{u_\eps}^2 \right)^2}{4 \eps^2} \right)\,\vol_g - \frac{\alpha}{2 \rho} \int_{\mathcal{B}_\rho(x_0)} \abs{F_\eps}^2\,\vol_g \right\},
	\end{split}
	\]
	for any $\rho \in (0,\,\inj(M))$. Then, by~\eqref{eq:est-Feps} and Young's 
	inequality we obtain
	\[
	\begin{split}
		& \frac{\d}{\d \rho} \left( \frac{e^{C\rho^2}}{\rho^{n-4+\alpha}} 
		\int_{B_\rho(x_0)} e_\eps(u_\eps,\,A_\eps)\,\vol_g \right) \\
		&\geq \frac{e^{C\rho^2}}{\rho^{n-4+\alpha}} 
		\left\{ \vphantom{\int_{\mathcal{\partial B}_\rho(x_0)} \abs{{\rm i}_\nu F_\eps}^2 \,\vol_{\hat{g}}}
		 \left( 2 C - C(M) \right) \rho E_\eps(x_0,\,\rho) + X_\eps(\rho,\, x_0) \right. \\
		 & \left. + \frac{2 -\alpha}{\rho} \int_{\mathcal{B}_\rho(x_0)} 
		\frac{1}{2} \abs{\D_{A_\eps} u_\eps}^2 + \frac{\left( 1 - \abs{u_\eps}^2 \right)^2}{4 \eps^2} \,\vol_g \right.\\
		&\left. - 2 \alpha K^2 \rho \int_{\mathcal{B}_\rho(x_0)} \abs{\D_{A_\eps} u_\eps}^2\, \vol_g 
		-2 \alpha \lambda^2 \int_{\mathcal{\partial B}_\rho(x_0)} \abs{{\rm i}_\nu F_\eps}^2 \,\vol_{\hat{g}} \right\},
	\end{split}
	\]
	for any $\rho \in (0,\,\bar{\rho})$, with $\bar{\rho}$ as in 
	Lemma~\ref{lemma:est-Feps}, and 
	where $K$ and $\lambda$ are the same constants as in 
	Lemma~\ref{lemma:est-Feps} (both of which depend only on $M$ and $g$). 
	Thus, recalling the definition~\eqref{eq:def-X-rho} of $X_\eps(x_0,\,\rho)$ 
	and taking $\alpha = \min\left\{2,\,\frac{1}{2\lambda^2},\,\frac{C(M)}{2K^2} \right\}$ 
	and (for instance) $C = 2 C(M)$, the proposition follows.
\end{proof}



With Proposition~\ref{prop:n-4-monotonicity} at hand, we can prove a 
decay estimate for the $L^2$-norm of the curvatures $F_\eps$ in sufficiently 
small geodesic balls. In turn, such decay property is the key point to obtain 
the $(n-2)$-monotonicity formula~\eqref{eq:monotonicity}.

\begin{lemma}\label{lemma:decay-Feps}
	There exist a number $\alpha \in (0,\,2)$ and a number $R_0 \in (0,\,\inj(M))$,
	both depending only on $M$ and $g$, such that the following holds. If 
	$(u_\eps,\,A_\eps) \in \mathcal{E}$ is any 
	critical point of $G_\eps$, 
	then for any $x_0 \in M$ and any $0 < r < R < R_0$, there holds 
	\begin{equation}\label{eq:decay-Feps}
		\int_{\mathcal{B}_r(x_0)} \abs{F_\eps}^2\,\vol_g \lesssim
		\frac{r^{n-2+\alpha}}{R^{n-2+\alpha}} 
		\int_{\mathcal{B}_{R}(x_0)} e_\eps(u_\eps,\,A_\eps)\,\vol_g,
	\end{equation}
	where the implicit constant in of the right-hand-side depends only on 
	$M$ and $g$. 
\end{lemma}

\begin{proof}
	We start by fixing $x_0 \in M$ and considering two geodesic balls 
	$\mathcal{B}_r := \mathcal{B}_r(x_0)$ and 
	$\mathcal{B}_s := \mathcal{B}_s(x_0)$, with $0 < r < s < R$ and 
	$R \in (0,\,\inj(M))$ small enough (to be further specified later). Up to 
	subtracting a (smooth, $\eps$-independent) 1-form to $A_\eps$, we can 
	write 
	\[
		\d A_\eps = F_\eps \qquad \mbox{in } \mathcal{B}_R.
	\]
	Moreover, we can fix Coulomb gauge (without changing notation), i.e., 
	we can assume that
	\[
		\d^* A_\eps = 0 \quad \mbox{in } \mathcal{B}_R,\quad 
		{\rm i}_\nu A_\eps = 0 \quad \mbox{on } \partial \mathcal{B}_R,
	\]
	where $\nu$ is the unit normal field to $\partial \mathcal{B}_R$. Then, 
	the Euler-Lagrange equation~\eqref{EL-A} becomes
	\[
		-\Delta A_\eps = j(u_\eps,\,A_\eps) \qquad \mbox{in } \mathcal{B}_R.
	\]
	Standard decay estimates for elliptic equations
	with a right-hand side (see e.g. Proposition~\ref{prop:decay-RHS}
	for more details) imply that there exists $R_* > 0$ (depending only on 
	$M$ and $g$) so that, for $0 < r < s < R < R_*$, there holds
	\[
		\int_{\mathcal{B}_r} 
		\left(\abs{\d A_\eps}^2 + \abs{\d^* A_\eps}^2\right)\vol_g \lesssim 
		\frac{r^n}{s^n} \int_{\mathcal{B}_s} 
		\left(\abs{\d A_\eps}^2 + \abs{\d^* A_\eps}^2\right)\vol_g 
		+ s^2 \int_{\mathcal{B}_s} \abs{j(u_\eps,\,\D_{A_\eps})}^2\,\vol_g, 
	\]
	up to a constant depending only on $M$ and $g$.
	Keeping in mind that~$\d^* A_\eps = 0$ and
	$\d A_\eps = F_\eps$ in $\mathcal{B}_R$,
	and recalling the pointwise estimate 
	$\abs{j(u_\eps,\,A_\eps)} \leq \abs{\D_{A_\eps} u_\eps}$, it follows 	
	\begin{equation}\label{eq:Feps-decay-1}
		\int_{\mathcal{B}_r} \abs{F_\eps}^2\,\vol_g \lesssim 
		\frac{r^n}{s^n} \int_{\mathcal{B}_s} \abs{F_\eps}^2\,\vol_g 
		+ s^2 \int_{\mathcal{B}_s} \abs{\D_{A_\eps} u_\eps}^2\,\vol_g, 
	\end{equation}
	up to a constant depending only on $M$ and $g$.
	
	Taking now $R_0 = \min\{R_*,\,\bar{\rho}\}$, 
	where $\bar{\rho}$ is the distinguished radius (depending only on $M$ 
	and $g$) given by Proposition~\ref{prop:n-4-monotonicity} and shrinking 
	$R$ so that $0 < r < s < R < R_0$, by~\eqref{eq:Feps-decay-1} 
	and~\eqref{eq:n-4-monotonicity} we obtain
	\begin{equation}\label{eq:Feps-decay-2}
	\begin{split}
		\int_{\mathcal{B}_r} \abs{F_\eps}^2\,\vol_g & \lesssim 
		\frac{r^n}{s^n} \int_{\mathcal{B}_s} \abs{F_\eps}^2\,\vol_g 
		+ s^2 \int_{\mathcal{B}_s} e_\eps(u_\eps,\,A_\eps)\,\vol_g \\
		& \lesssim \frac{r^n}{s^n} \int_{\mathcal{B}_s} \abs{F_\eps}^2\,\vol_g 
		+ \frac{s^{n-2+\alpha}}{R^{n-4+\alpha}} 
		\int_{\mathcal{B}_{R}} e_\eps(u_\eps,\,A_\eps)\,\vol_g,
	\end{split}
	\end{equation}
	for $\alpha \in (0,2]$ as given by Proposition~\ref{prop:n-4-monotonicity}. 
	
	Assume $\alpha \in (0,\,2)$. Then, by a classical iteration lemma 
	(e.g., \cite[Lemma~III.2.1]{Giaquinta-MultipleIntegrals} 
	or \cite[Lemma~B.3]{Beck}), it follows 
	\begin{equation}\label{eq:Feps-decay-3}
	\begin{split}
		\int_{\mathcal{B}_r} \abs{F_\eps}^2 \,\vol_g & \lesssim 
		\frac{r^{n-2+\alpha}}{s^{n-2+\alpha}} 
		\int_{\mathcal{B}_s} \abs{F_\eps}^2 \vol_g 
		+ \frac{r^{n-2+\alpha}}{R^{n-4+\alpha}} 
		\int_{\mathcal{B}_{R}} e_\eps(u_\eps,\,A_\eps)\,\vol_g \\
		& \lesssim \frac{r^{n-2+\alpha}}{s^{n-2+\alpha}} 
		\int_{\mathcal{B}_s} \abs{F_\eps}^2 \,\vol_g + \frac{r^{n-2+\alpha}}{R^{n-2+\alpha}}\int_{\mathcal{B}_{R}} e_\eps(u_\eps,\,A_\eps)\,\vol_g
	\end{split}
	\end{equation}
	for any $0 < r < s < R < R_0$. 
	(The second inequality holds because $R$ is 
	certainly bounded from above.) In particular, by taking 
	the limit $s \to R$, we get the desired decay 
	estimate~\eqref{eq:decay-Feps}. 
\end{proof}

\begin{remark}\label{rk:alpha-2}
	The restriction $\alpha < 2$ is necessary in the argument 
	from~\eqref{eq:Feps-decay-2} to~\eqref{eq:Feps-decay-3}, as otherwise we 
	could not use \cite[Lemma~III.2.1]{Giaquinta-MultipleIntegrals} or
	\cite[Lemma~B.3]{Beck}. This inconvenience will produce
	the ``error term'', i.e., the term depending on $\alpha$, 
	in~\eqref{eq:monotonicity}. 
	Nevertheless, the monotonicity formula~\eqref{eq:monotonicity} 
	is enough for our purposes. 
\end{remark}


We are now ready to prove Theorem~\ref{thm:monotonicity}.

\begin{proof}[Proof of Theorem~\ref{thm:monotonicity}]
	Thanks to Lemma~\ref{lemma:decay-Feps}, 
	the inequality~\eqref{eq:PS-ineq} yields
	\begin{equation}\label{eq:mon-comp-1}
		\rho E^\prime_\eps(x_0,\,\rho) \geq \rho X_\eps(x_0,\,\rho) + (n-2) E_\eps(x_0,\,\rho) 
		+ C \rho^2 E_\eps(x_0,\,\rho) 
		- \frac{C \rho^{n-2+\alpha}}{R^{n-2+\alpha}} E_\eps(x_0,\,R)
	\end{equation}
	for any $0 < \rho < R < R_0$, where $R_0$ is the distinguished radius 
	in $(0,\, \inj(M))$ given by Lemma~\ref{lemma:decay-Feps}, any 
	$x_0 \in M$, and some sufficiently large constant~$C$
	(depending on~$M$, $g$ only).
	
	Dividing both sides of~\eqref{eq:mon-comp-1} by $\rho^{n-1}$, we obtain 
	\begin{equation}\label{eq:mon-comp-2}
	\begin{split}
		\rho^{2-n} E^\prime_\eps(x_0,\,\rho) &- (n-2) \rho^{1-n} E_\eps(x_0,\,\rho) 
		+ C \rho^{3-n} E_\eps(x_0,\,\rho) \\
		& \geq \rho^{2-n} X_\eps(x_0,\,\rho) 
		- \frac{C \rho^{\alpha-1}}{R^{n-2+\alpha}} E_\eps(x_0,\,R)
	\end{split}
	\end{equation}
	Multiplying~\eqref{eq:mon-comp-2} by $\exp(C\rho^2/2)$, where $C$ is 
	exactly the same constant as in~\eqref{eq:mon-comp-2}, we get
	\begin{equation}\label{eq:mon-comp-3}
		\frac{\d }{\d \rho}\left\{ e^{C \rho^2/2} \rho^{2-n} E_\eps(x_0,\,\rho) \right\} \geq e^{C \rho^2 /2} \rho^{2-n} X_\eps(x_0,\,\rho) - \frac{C e^{C \rho^2/2} \rho^{\alpha-1}}{R^{n-2+\alpha}} E_\eps(x_0,\,R),
	\end{equation}
	and the conclusion follows by integrating both sides 
	of~\eqref{eq:mon-comp-3} over $\rho \in (r,\,R)$ and recalling that $\alpha$ depends only on $(M,\,g)$.
\end{proof}

\begin{remark}\label{rk:improved-monotonicity}
Multiplying~\eqref{eq:mon-comp-3} by $\rho$ and recalling the
definition~\eqref{eq:def-X-rho} of $X_\eps$, we obtain the following variant 
of~\eqref{eq:mon-comp-3}: for any~$0 < \rho < R < R_0$,
\begin{equation} \label{improvedmonotonicity}
 \rho \, \frac{\d}{\d\rho}\left(\frac{e^{C\rho^2} E(x_0,\,\rho)}
  {\rho^{n - 2}}\right)
 \geq \frac{1}{2\eps^2 \rho^{n-2}} 
  \int_{\mathcal{B}_\rho} \left(1 - \abs{u_\eps}^2\right)^2\vol_g
  - \frac{C^\prime \, \rho^\alpha \, E(x_0,\,R)}{R^{n-2+\alpha}}
\end{equation}
where~$\alpha\in (0, \, 2)$, $C$, $C^\prime$
are absolute constants that depend only on~$M$ and~$g$. We will 
use~\eqref{improvedmonotonicity} in the proof of the clearing-out, 
in Section~\ref{sec:clearing-out}.
\end{remark}

As a byproduct of our arguments, we obtain the following
bound on the curvature in terms of the rescaled energy:

\begin{prop} \label{prop:almostLinftycurvature}
	For any~$q\in (0, \, n)$ there exist constant~$C_q > 0$,
	depending only on~$(M, g)$ and~$q$,
	such that the follwing statement holds.
	If $(u_\eps,\,A_\eps) \in \mathcal{E}$ is any critical point of $G_\eps$, then
	\begin{equation}\label{almostLinfty-Feps}
		\int_{\mathcal{B}_r(x_0)} \abs{F_\eps}^2 \,\vol_g 
		\leq \frac{C_q \, r^{q}}{R^{n-2}} E_\eps(x_0, \, R)
	\end{equation} 
	holds for any~$x_0\in M$ and~$0 < r < R$.
\end{prop}
\begin{proof}
 By a covering argument, there is no loss
 of generality in assuming that~$R< R_0$, where~$R_0$
 is a uniform constant. For instance, we can take~$R_0$
 exactly equal to the constant given by Lemma~\ref{lemma:decay-Feps}.
 Let~$\alpha \in (0, \, 2]$ be given by Proposition~\ref{prop:n-4-monotonicity}, and let~$0 < r < s < R < R_*$.
 The estimate~\eqref{eq:Feps-decay-2},
 obtained in the proof of Lemma~\ref{lemma:decay-Feps},
 and the (almost) monotonicity formula~\eqref{eq:monotonicity} imply
 \[
	\begin{split}
		\int_{\mathcal{B}_r} \abs{F_\eps}^2\,\vol_g & \lesssim 
		\frac{r^n}{s^n} \int_{\mathcal{B}_s} \abs{F_\eps}^2\,\vol_g 
		+ s^2 \int_{\mathcal{B}_s} e_\eps(u_\eps,\,A_\eps)\,\vol_g \\
		& \lesssim \frac{r^n}{s^n} \int_{\mathcal{B}_s} \abs{F_\eps}^2\,\vol_g 
		+ \frac{s^{n}}{R^{n-2}} 
		\int_{\mathcal{B}_{R}} e_\eps(u_\eps,\,A_\eps)\,\vol_g,
	\end{split}
 \]
 The proposition follows by an iteration argument, 
 based on~\cite[Lemma~III.2.1]{Giaquinta-MultipleIntegrals}.
\end{proof}

Finally, we conclude this section by stating and proving a
corollary of the monotonicity formula, which will be useful
in Section~\ref{sec:clearing-out}.

\begin{corollary}\label{cor:monotonicity}
	Let $\eps > 0$ and let $(u_\eps,\,A_\eps) \subset \mathcal{E}$ be any critical point of $G_\eps$. Moreover, let 
	$\alpha \in (0,\,2)$, $R_0 \in (0,\,\inj(M))$ and $C > 0$ be the numbers 
	(depending only on $(M,\,g)$) given by Theorem~\ref{thm:monotonicity}.  
	Then, for any $x_0 \in M$ and any $R \in (0,\,R_0)$, there holds
	\begin{equation}\label{eq:cor-monotonicity}
		\frac{1}{\eps^2}\int_{B_R(x_0)} \frac{\left( 1 - \abs{u_\eps(y)}^2 \right)^2}{\dist^{n-2}(y,\,x_0)}
		\,\vol_g(y) \lesssim
		R^{2-n} E_\eps(x_0,\,R),
	\end{equation}
	where the implicit constant in front of the right hand side depends only on 
	$(M,\,g)$.
\end{corollary}
\begin{proof}
	By Proposition~\ref{prop:L-infty-bound-u}, 
	Lemma~\ref{lemma:L-infty-bound-D_Au} and 
	Lemma~\ref{lemma:Linfty-Feps}, the energy
	density is a bounded function (and its
	$L^\infty(M)$-norm is of order~$\mathrm{O}(\eps^{-2})$).
	Consequently, there holds
	\begin{equation} \label{eq:cor-mon-compu0}
		\lim_{r \to 0} r^{2-n} E_\eps(x_0,\, r) = 0.
	\end{equation}
	Let us set
	\[
		W_\eps(x_0,\,\rho) := \int_{\mathcal{B}_\rho(x_0)} \frac{\left(1-\abs{u_\eps(y)}^2\right)^2}{2 \eps^2} \vol_g(y).
	\]
	By~\eqref{eq:mon-comp-3} and~\eqref{eq:def-X-rho}, we clearly obtain
	\[
			\frac{\d }{\d \rho}\left\{ e^{C \rho^2/2} \rho^{2-n} E_\eps(x_0,\,\rho) \right\} \geq \rho^{1-n} \int_{\mathcal{B}_\rho(x_0)} \frac{\left(1-\abs{u_\eps(y)}^2\right)^2}{2 \eps^2} \vol_g(y)  - \frac{C e^{C \rho^2/2} \rho^{\alpha-1}}{R^{n-2+\alpha}} E_\eps(x_0,\,R).
	\]
	From now on, we work in polar coordinates in $\mathcal{B}_\rho$ 
	(centered at $x_0$). 
	Since the function 
	$(0,\,R) \ni \rho \mapsto W_\eps(x_0,\,\rho)$ is absolutely 
	continuous, for a.e. $\rho \in (0, R)$ we have 
	\[
		W^\prime_\eps(x_0,\,\rho) = \int_{\partial \mathcal{B}_\rho(x_0)} \frac{\left(1-\abs{u_\eps(y)}^2\right)^2}{2 \eps^2} \vol_{\hat{g}}(y)
	\]
	as well as 
	\[
		\rho^{1-n} W_\eps(x_0,\,\rho) = \frac{1}{n-2}\left\{-\frac{\d}{\d \rho}\left( \rho^{2-n} W_\eps(x_0,\,\rho) \right) + \rho^{2-n} W^\prime_\eps(x_0,\,\rho)\right\}.
	\]
	Hence,
	\begin{equation}\label{eq:cor-mon-compu2}
	\begin{split}
	\frac{\d }{\d \rho}\Big\{ e^{C \rho^2/2} \rho^{2-n} E_\eps(x_0,\,\rho)
	&+ \frac{1}{n-2} \rho^{2-n} W_\eps(x_0,\,\rho) \Big\}\geq\\
	 &\frac{\rho^{2-n}}{n-2} W^\prime_\eps(x_0,\,\rho) - \frac{C e^{C \rho^2/2} \rho^{\alpha-1}}{R^{n-2+\alpha}} E_\eps(x_0,\,R)
	 \end{split}
	\end{equation} 
	for a.e. $\rho \in (0,\,R)$. Integrating both sides 
	of~\eqref{eq:cor-mon-compu2} over $(0,\,R)$,
	and taking~\eqref{eq:cor-mon-compu0} into account, yields
	\[
	\begin{split}
		\frac{1}{n-2} &\int_{\mathcal{B}_R(x_0)} \frac{1}{2\eps^2} \frac{\left(1-\abs{u_\eps(y)}^2\right)^2}{\dist(x_0,y)^{n-2}} \,\vol_g(y) \\
		&\leq \left( 1 + \frac{C}{\alpha} \right) e^{C R^2/2} R^{2-n}E_\eps(x_0,\,R)	+ \frac{R^{2-n}}{n-2} \int_{\mathcal{B}_R(x_0)} \frac{\left( 1 - \abs{u_\eps(y)}^2 \right)^2}{2\eps^2} \,\vol_g(y)\\
		&\leq \left( \frac{n}{n-2} + \frac{C}{\alpha}\right) R^{2-n} E_\eps(x_0,\,R).
	\end{split}
	\]
	Recalling that $\alpha$ depends only on $(M,\,g)$, the conclusion follows.
\end{proof}
\section{Clearing-out and its consequences}\label{sec:clearing-out}

The aim of this section is to prove 
Proposition~\ref{prop:small-ball-small-energy} below,
which will be essential 
in characterising the support of the
energy-concentration measure~$\mu_*$, defined in~\eqref{eq:mu-eps}
(see Section~\ref{sec:varifold} and Lemma~\ref{lemma:out-of-support}
in particular).
As in the previous section, given a finite-energy
pair~$(u_\eps, \, A_\eps)$, a point~$x_0\in M$ and
a radius~$\rho > 0$, we write~$E_\eps(x_0, \, \rho) 
:= G_\eps(u_\eps, \, A_\eps; \, \mathcal{B}_\rho(x_0))$.
\begin{prop}\label{prop:small-ball-small-energy}
	There exists constants $\eta_0 > 0$, $R_* \in (0, \inj(M))$ and~$\eps_* > 0$, 
	depending on $(M,g)$ only, such that the following statement holds. 
	Let $x_0 \in M$, $R \in (0,R_*)$ and let $\{(u_\eps,\,A_\eps)\}
	\subset (W^{1,2}\cap L^\infty)(M, \, E)\times W^{1,2}(M, \, \T^*M)$ 
	be a sequence of critical points of $G_\eps$ 
	that satisfy~\eqref{hp:logenergy} and
	\begin{equation}\label{eq:small-resclaed-energy}
		E_\eps(x_0,R) \leq \eta_0 \, R^{n-2} \log\frac{R}{\eps}.
	\end{equation}
	Then, 
	for any $\eps < \eps_*$, there holds
	\begin{equation}\label{eq:clearingout-small-ball}
		\textrm{for any } x \in \mathcal{B}_{3R/4}(x_0),
		\quad \abs{u_\eps(x)} \geq \frac{1}{2}.
	\end{equation} 
	Moreover, 
	there holds
	\begin{equation}\label{eq:no-energy-small-balls}
		\sup_{\eps > 0} E_\eps(x_0, R/2) < +\infty.
	\end{equation}
\end{prop}

Proposition~\ref{prop:small-ball-small-energy} is the counterpart 
of~\cite[Proposition~VII.1]{BethuelBrezisOrlandi}. For its proof, 
we shall follow the same strategy as in~\cite{BethuelBrezisOrlandi},
involving essentially three ingredients: energy decay estimates, 
a {\em clearing-out} property, and elliptic estimates. 

The pivotal point of the argument is the ``clearing-out'' property, 
Proposition~\ref{prop:clearingout}. Indeed, this property implies 
that Equation~\eqref{EL-A} is uniformly elliptic in any geodesic ball in 
which~\eqref{eq:small-resclaed-energy} holds. 
In turn, uniform ellipticity implies several strong elliptic estimates that 
allow to achieve the conclusion of Proposition~\ref{prop:small-ball-small-energy}. 

According to \cite{BethuelBrezisOrlandi}, the main concern towards the proof 
of the clearing-out property lies in obtaining a suitable, 
{\em quantitative} decay of the rescaled energy in small balls with respect 
to the radius. This is done in Section~\ref{sec:energydecay} below.

Throughout Section~\ref{sec:clearing-out}, we 
will consider a sequence of critical points 
$\{(u_\eps,\,A_\eps)\}\subset (W^{1,2}\cap L^\infty)(M, \, E)
\times W^{1,2}(M, \, \T^*M)$, a point~$x_0\in M$
and a radius~$R > 0$ (smaller than the injectivity radius of~$M$).
We will write indifferently~$\mathcal{B}_r$
or~$\mathcal{B}_r(x_0)$ for a geodesic ball of center~$x_0$
and radius~$r > 0$.

\subsection{An energy decay estimate}
\label{sec:energydecay}

Towards the proof of Proposition~\ref{prop:small-ball-small-energy},
the first step is an energy decay estimate, 
analogous to~\cite[Theorem~3]{BethuelBrezisOrlandi}.

\begin{prop} \label{prop:decay}
 There exist numbers~$R_* > 0$ and~$\alpha\in (0, \, 2)$,
 depending on~$(M, \, g)$ only, such that the following
 statement holds. Let~$x_0\in M$, $0 < \eps < R < R_*$
 and~$0 < \delta < 1/10$.
 Let~$(u_\eps, \, A_\eps)$ be a 
 critical point of $G_\eps$.
 Define
 \begin{equation} \label{}
  p_\eps := \frac{1}{\eps^2  \, R^{n-2}}
    \int_{\mathcal{B}_R(x_0)} \left(1 - \abs{u_\eps}^2\right)^2 \vol_g
 \end{equation}
 Then, there holds 
 \[
  \begin{split}
   \frac{E_\eps(x_0, \, \delta R)}{R^{n-2}} \lesssim 
    \left(p_\eps^{1/3} + \delta^{n-2+\alpha}\right)
     \frac{E_\eps(x_0, \, R)}{R^{n-2}} + p_\eps^{2/3}
  \end{split}
 \]
 where
 
\end{prop}

The rest of the section is devoted to the proof of
Proposition~\ref{prop:decay}.
Let~$x_0\in M$, $0 < R < \inj(M)$, 
$\eps > 0$ and~$0 < \delta < 1/10$ be fixed,
and let~$(u_\eps, \, A_\eps)$
be a 
critical point of $G_\eps$.
First of all, we choose a radius~$r\in (R/10, \, R/5)$
such that
\begin{align}
 \int_{\partial \mathcal{B}_r} \abs{\D_{A_\eps} u_\eps}^2 \d\H^{n-1}
  &\leq \frac{40}{R} \int_{\mathcal{B}_R} \abs{\D_{A_\eps} u_\eps}^2 \vol_g 
  \label{average-r-1} \\
 \int_{\partial \mathcal{B}_r} \left(\abs{u_\eps}^2 - 1\right)^2 \d\H^{n-1}
  &\leq \frac{40}{R} \int_{\mathcal{B}_R} \left(\abs{u_\eps}^2 - 1\right)^2 \vol_g 
  \label{average-r-2}
\end{align}
Such a radius exist, because the set of radii~$r\in(R/10, \, R/5)$
that do not satisfy either of the inequalities~\eqref{average-r-1}
has length~$R/40$ at most. For such a choice of~$r$, we have~$r>\delta R$.
Let~$\eta_r$ be the indicator function of the ball~$\mathcal{B}_r$
(i.e, $\eta_r(x) := 1$ if~$x\in \mathcal{B}_r$ and~$\eta_r(x) := 0$
otherwise). By Hodge decomposition,
there exist~$\varphi_\eps \in W^{1,2}(M)$,
$\psi_\eps\in W^{1,2}(M, \, \Lambda^2\T^*M)$ and~$\xi_\eps\in \Harm^1(M)$
such that
\begin{equation} \label{Hodge-jac}
 \eta_r \, j(u_\eps, \, A_\eps) = \d\varphi_\eps + \d^*\psi_\eps + \xi_\eps 
\end{equation}
where the three terms in the right-hand side are orthogonal in~$L^2(M)$.
Moreover, there is no loss of generality in assuming that~$\psi_\eps$
is exact (if not, we replace~$\psi_\eps$ by its projection onto exact forms).
We prove decay estimates on~$\varphi_\eps$, $\psi_\eps$
and~$\xi_\eps$ separately.

\begin{lemma} \label{lemma:decay-phixi}
 We have
 \[
  \int_{\mathcal{B}_{\delta R}}\left(\abs{\d\varphi_\eps}^2 
   + \abs{\xi_\eps}^2\right)\vol_g \lesssim \delta^n E_\eps(x_0, \, R) 
 \]
\end{lemma}
\begin{proof}
 We consider~$\varphi_\eps$ first. In the interior of~$\mathcal{B}_r$,
 we have
 \[
  -\Delta\varphi_\eps = \d^*\d\varphi_\eps = \d^*j(u_\eps, \, A_\eps)
  = \d^*\d^* F_\eps = 0
 \]
 By differentiating both sides of~$-\Delta\varphi_\eps =0$,
 we deduce that~$-\Delta(\d\varphi_\eps) = 0$ in the 
 interior of~$\mathcal{B}_{R}$. Then, Theorem~\ref{thm:decay-harmonic}
 implies
 \[
  \int_{\mathcal{B}_{\delta R}}\abs{\d\varphi_\eps}^2 \,\vol_g 
  \lesssim \delta^n \int_{\mathcal{B}_{R}}\abs{\d\varphi_\eps}^2 \,\vol_g 
  \lesssim \delta^n \int_{M}\abs{\d\varphi_\eps}^2 \,\vol_g 
 \]
 As the decomposition~\eqref{Hodge-jac} is orthogonal in~$L^2(M)$,
 we deduce
 \[
  \int_{\mathcal{B}_{\delta R}}\abs{\d\varphi_\eps}^2 \,\vol_g 
  \lesssim \delta^n \int_{\mathcal{B}_R} \abs{j(u_\eps, \, A_\eps)}^2 \, \vol_g
  \lesssim \delta^n E_\eps(x_0, \, R) 
 \]
 where we have used that 
 $\abs{j(u_\eps, \, A_\eps)}\leq \abs{\D_{A_\eps}u_\eps}$
 for the last inequality.
 An analogous argument applies to the harmonic form~$\xi_\eps$. 
\end{proof}

The proof of the decay estimate for~$\psi_\eps$ is more technical.
It is convenient to further decompose~$\psi_\eps$ as 
a sum of several contributions, as 
in~\cite[Proof of Theorem~3, Step~2]{BethuelBrezisOrlandi}.
Let~$\beta\in (0, \, 1/4)$ be a parameter, to be chosen later.
Let~$f\colon [0, \, +\infty)\to [0, \, +\infty)$ be a smooth
function, such that
\begin{align} 
  f(t) = 1 \quad \textrm{if } 0 \leq t \leq 1 - 2\beta, \qquad
  &f(t) = \frac{1}{t} \quad \textrm{if } t \geq 1 - \beta, \label{f-BBO1} \\
  1 \leq f(t) \leq \frac{1}{t} \quad \textrm{if } 1 - 2\beta < t < 1 - \beta,
  \qquad &\abs{f^\prime(t)} \leq 4 \qquad \textrm{for any } t \geq 0 
  \label{f-BBo2}
\end{align}
We define
\begin{equation} \label{rho_BBO}
 \rho_\eps := f(\abs{u_\eps}), 
 \qquad v_\eps := \rho_\eps u_\eps
\end{equation}
By construction, the function~$\rho_\eps$ satisfies
\begin{equation} \label{rho_BBO_beta}
 0 \leq \rho_\eps^2 - 1 \leq 4\beta \qquad \textrm{in } M.
\end{equation}
Moreover, writing the reference connection
as~$\D_0 = \d - i\gamma_0$ for some some (locally defined)
$1$-form~$\gamma_0$, we have
\begin{equation} \label{v_BBO}
 j(v_\eps, \, A_\eps) 
 = \ip{\d \rho_\eps\otimes u_\eps + \rho_\eps \, \d u_\eps 
  - i(\gamma_0 + A_\eps)u_\eps}{i \rho_\eps u_\eps}
 = \rho_\eps^2 \, j(u_\eps, \, A_\eps) 
\end{equation}
Now, we consider the Hodge decomposition~\eqref{Hodge-jac}.
As we have seen above, we can assume with no loss of generality
that~$\psi_\eps$ is exact.
Then, by differentiating both sides of~\eqref{Hodge-jac}, we obtain
\[
 \begin{split}
  -\Delta\psi_\eps = \d\d^*\psi_\eps
  &= \d\left(\eta_r \, j(u_\eps, \, A_\eps)\right)\\
  &\hspace{-.17cm} \stackrel{\eqref{v_BBO}}{=} 
   \d\left(\eta_r \, j(v_\eps, \, A_\eps)\right)
   + \d\left(\eta_r \, (1 - \rho_\eps^2) \, j(u_\eps, \, A_\eps)\right) \\
  &= \eta_r \, \d j(v_\eps, \, A_\eps)
   - \sigma_r \, \nu^\flat\wedge j(v_\eps, \, A_\eps)
   + \d\left(\eta_r \, (1 - \rho_\eps^2) \, j(u_\eps, \, A_\eps)\right) 
 \end{split}
\]
where~$\sigma_r$ denotes the~$(n-1)$-dimensional Hausdorff measure
on~$\partial \mathcal{B}_r$ and~$\nu$ is the outward unit vector to~$\partial \mathcal{B}_r$.
Recalling the definition of Jacobian, Equation~\eqref{Jac}, we deduce
\[
 \begin{split}
  -\Delta\psi_\eps
  &= 2\eta_r \, J(v_\eps, \, A_\eps) - \eta_r \, F_\eps
   - \sigma_r \, \nu^\flat\wedge j(v_\eps, \, A_\eps)
   + \d\left(\eta_r \, (1 - \rho_\eps^2) \, j(u_\eps, \, A_\eps)\right) 
 \end{split}
\]
Therefore, we can write
\begin{equation} \label{London-psi}
 -\Delta\psi_\eps 
 = \omega_1 + \omega_2 + \omega_3 + \omega_4,
\end{equation}
where
\begin{align}
 \omega_1 := 2\eta_r \, J(v_\eps, \, A_\eps), \qquad
  &\omega_2 := - \eta_r \, F_\eps \label{omega-BBO1} \\
 \omega_3 := - \sigma_r \, \nu^\flat\wedge j(v_\eps, \, A_\eps), \qquad
  &\omega_4 := \d\left(\eta_r \, (1 - \rho_\eps^2) \, j(u_\eps, \, A_\eps)\right) 
\end{align}
We have~$\omega_k\in W^{-1, 2}(M, \, \Lambda^2\T^*M)$
for each~$k$ --- in fact, $\omega_1$, $\omega_2$ 
belong to~$L^2(M, \, \Lambda^2\T^*M)$, $\omega_4$
is the differential of an~$L^2$-form, 
and~$\omega_3\in W^{-1,2}(M, \, \Lambda^2\T^*M)$ 
by continuity of the trace operator~$W^{1,2}(M)\to L^2(\partial \mathcal{B}_r)$.
Let~$H\colon W^{-1,2}(M, \, \Lambda^2\T^*M)\to\Harm^2(M)$
be the orthogonal projection onto the space
of harmonic~$2$-forms, defined by
\begin{equation} \label{harmonicprojection}
 H(\omega) := \sum_{j=1}^{\ell}
 \ip{\omega}{\zeta_j}_{\mathcal{D}^\prime(M), \, \mathcal{D}(M)} \zeta_j
\end{equation}
where~$(\zeta_1, \, \ldots, \, \zeta_\ell)$
is an orthonormal basis of the finite-dimensional space~$\Harm^2(M)$.
By Lax-Milgram lemma, for each~$k$ there exists
a unique~$\psi_k\in W^{1,2}(M, \, \Lambda^2\T^*M)$
such that
\begin{equation} \label{psik}
 \begin{cases}
  -\Delta\tau_k = \omega_k - H(\omega_k) \qquad 
   &\textrm{in the sense of distributions in } M \\[5pt]
   \displaystyle\int_M \ip{\tau_k}{\zeta} \, \vol_g = 0 
   &\textrm{for all harmonic $2$-forms } \zeta.
 \end{cases}
\end{equation}
in the sense of distributions on~$M$.
We have~$\psi_\eps = \psi_1 + \psi_2 + \psi_3 + \psi_4$
because, by construction,
$\psi_\eps$ is exact and hence, orthogonal
to all harmonic~$2$-forms.
We shall prove a decay estimate for each~$\d^*\psi_k$ separately.

\paragraph{Estimate for~$\tau_4$.}
We claim that
\begin{equation} \label{tau4}
 \int_{M} \abs{\d^*\tau_4}^2 \, \vol_g \lesssim \beta^2 \, E_\eps(x_0, \, R)
\end{equation}
Indeed, as~$\omega_4$ is exact, it is orthogonal to
all harmonic forms, i.e.~$H(\omega_4) = 0$.
Then, by testing the equation for~$\tau_4$ against~$\tau_4$,
we obtain
\[
 \int_{M} \left(\abs{\d\tau_4}^2 + \abs{\d^*\tau_4}^2 \right)\vol_g
 = \int_{\mathcal{B}_r} (1 - \rho_\eps^2) \, \ip{j(u_\eps, \, A_\eps)}{\d^*\tau_4} \vol_g
\]
By applying the Young inequality at the right hand side,
we deduce
\[
 \begin{split}
  \int_{M} \left(\abs{\d\tau_4}^2 + \abs{\d^*\tau_4}^2\right)\vol_g
  \lesssim \int_{\mathcal{B}_r} (1 - \rho_\eps^2)^2 \, 
   \abs{j(u_\eps, \, A_\eps)}^2 \vol_g 
  \stackrel{\eqref{rho_BBO_beta}}{\lesssim} 
   \beta^2 \int_{\mathcal{B}_r} \abs{j(u_\eps, \, A_\eps)}^2 \vol_g
 \end{split}
\]
and~\eqref{tau4} follows.

\paragraph{Estimate for~$\tau_3$.}
We shall prove that
\begin{equation} \label{tau3}
 \int_{\mathcal{B}_{\delta R}} \abs{\d^*\tau_3}^2 \, \vol_g 
  \lesssim \delta^n E_\eps(x_0, \, R)
\end{equation}
To this end, we will need the following lemma.

\begin{lemma} \label{lemma:trace-R2}
 For any~$\tau\in W^{1,2}(M, \, \Lambda^k\T^*M)$, any~$x_0\in M$
 and any~$r>0$ small enough, there holds
 \begin{equation*} 
  \frac{1}{r^2}\int_{\partial \mathcal{B}_r(x_0)} \abs{\tau}^2 \, \d\H^{n-1}
  \lesssim \int_M\left(\abs{\d\tau}^2 
   + \abs{\d^*\tau}^2 + \abs{\tau}^2\right) \vol_g
 \end{equation*}
 where the implicit constant at the right-hand side
 does not depend on~$r$, $x_0$.
\end{lemma}
\begin{proof}
 For simplicity of notation, we write~$\mathcal{B}_r$
 instead of~$\mathcal{B}_r(x_0)$.
 We claim that, for any scalar function~$f\in W^{1,2}(\mathcal{B}_r)$,
 there holds
 \begin{equation} \label{trace-R2-1}
  \frac{1}{r^2}\int_{\partial \mathcal{B}_r} f^2 \, \d\H^{n-1}
  \lesssim \int_{\mathcal{B}_r} \abs{\d f}^2 \, \vol_g
  + \frac{1}{r^2}\int_{\mathcal{B}_r} f^2 \, \vol_g
 \end{equation}
 The inequality~\eqref{trace-R2-1} holds true 
 if~$\mathcal{B}_r$ is a ball in~$\R^n$, 
 equipped with the Euclidean metric. In case~$\mathcal{B}_r$
 is a geodesic ball in~$M$ (with~$r < \inj(M)/2$),
 the estimate~\eqref{trace-R2-1} follows
 by composition with 
 (normal geodesic) coordinate charts.
 The implicit constant in front of the right-hand side of~\eqref{trace-R2-1}
 depends only on the metric and is bounded uniformly
 with respect to~$x_0$, $r$, because~$M$ is compact and smooth.
 
 Now, take a form~$\tau\in W^{1,2}(M, \, \Lambda^k\T^*M)$.
 Let~$(x^1, \, \ldots, x^n)$ be normal geodesic
 charts in~$\mathcal{B}_r(x_0)$. Let us 
 write~$\tau = \sum_\alpha \tau_\alpha \d x^\alpha$,
 where the sum is taken over all multi-indices~$\alpha$ of order~$k$.
 By applying~\eqref{trace-R2-1} to each component~$\tau_\alpha$,
 we deduce
 \begin{equation*} 
  \frac{1}{r^2}\int_{\partial \mathcal{B}_r} \abs{\tau}^2 \, \d\H^{n-1}
  \lesssim \sum_\alpha \int_{\mathcal{B}_r} \abs{\d \tau_\alpha}^2 \, \vol_g
   + \frac{1}{r^2}\int_{\mathcal{B}_r} \abs{\tau}^2 \, \vol_g
  \lesssim \norm{\tau}^2_{W^{1,2}(M)} + \frac{1}{r^2}\norm{\tau}^2_{L^2(\mathcal{B}_r)}
 \end{equation*}
 The last term in the right-hand side can be estimated
 by applying the H\"older inequality and Sobolev embeddings.
 Indeed, if~$p := 2^* = 2n/(n - 2)$ 
 then $1/2 = 1/p + 1/n$ and hence,
 \begin{equation*} 
  \begin{split}
   \frac{1}{r^2}\norm{\tau}^2_{L^2(\mathcal{B}_r)}
   \lesssim \frac{1}{r^2} \norm{\tau}^2_{L^p(M)} \abs{\vol(\mathcal{B}_r)}^{2/n}
   \lesssim \norm{\tau}^2_{W^{1,2}(M)}
  \end{split}
 \end{equation*}
 (where the implicit constants are uniform with respect to~$r$, $x_0$).
 The lemma now follows by applying Gaffney's inequality 
 (see e.g.~\cite[Theorem~4.8]{IwaniecScottStroffolini}).
\end{proof}

\begin{remark} \label{rk:trace-R2}
 If~$\tau\in W^{1,2}(M, \, \Lambda^k\T^*M)$
 is orthogonal to all harmonic~$k$-forms,
 then the~$L^2(M)$-norm of~$\tau$ is bounded
 by the~$L^2(M)$-norms of~$\d\tau$ and~$\d^*\tau$
 (see e.g.~\cite[Theorem~4.11]{IwaniecScottStroffolini}),
 and we obtain
 \[
  \frac{1}{r^2}\int_{\partial \mathcal{B}_r(x_0)} \abs{\tau}^2 \, \d\H^{n-1}
  \lesssim \int_M\left(\abs{\d\tau}^2 + \abs{\d^*\tau}^2\right) \vol_g
 \]
\end{remark}

Now, we proceed to the proof of~\eqref{tau3}.
In the interior of~$\mathcal{B}_r$, the form~$\tau_3$
satisfies $-\Delta\tau_3 = 0$. By taking 
the codifferential of both sides of this equation,
we deduce $-\Delta(\d^*\tau_3) = 0$ in the interior of~$\mathcal{B}_r$.
Then, standard decay estimates for elliptic equations 
(see e.g. Theorem~\ref{thm:decay-harmonic} in the appendix) give
\begin{equation} \label{tau3,1}
 \int_{\mathcal{B}_{\delta R}}\abs{\d^*\tau_3}^2 \, \vol_g 
 \lesssim \delta^n \int_{\mathcal{B}_R} \abs{\d^*\tau_3}^2 \, \vol_g
\end{equation}
It remains to estimate the right-hand side of~\eqref{tau3,1}.
By testing the equation for~$\tau_3$ against~$\tau_3$,
we obtain
\begin{equation*} 
 \begin{split}
  \int_{M} \left(\abs{\d\tau_3}^2 + \abs{\d^*\tau_3}^2\right)\vol_g
  &= -\int_{\partial \mathcal{B}_r}  \ip{\nu^\flat\wedge
   j(v_\eps, \, A_\eps)}{\tau_3}   \d\H^{n-1} 
   + \int_M \ip{H(\omega_3)}{\tau_3} \vol_g \\
  &\leq \frac{r}{2\lambda} \int_{\partial \mathcal{B}_r} 
   \rho_\eps^4\abs{j(u_\eps, \, A_\eps)}^2 \, \d\H^{n-1}
   + \frac{\lambda}{2 r} \int_{\partial \mathcal{B}_r} \abs{\tau_3}^2 \d\H^{n-1} \\
  &\hspace{1.8cm}  + \frac{1}{2\lambda} \int_{M} \abs{H(\omega_3)}^2 \vol_g
   + \frac{\lambda}{2} \int_{M} \abs{\tau_3}^2 \vol_g
 \end{split}
\end{equation*}
for any~$\lambda > 0$. If~$\lambda$ is small enough,
the last term in the right-hand side can be absorbed into the
left-hand side, because~$\tau_3$ is orthogonal to all
harmonic~$2$-forms (see e.g.~\cite[Theorem~4.11]{IwaniecScottStroffolini}).
Moreover, due to Lemma~\ref{lemma:trace-R2}
and Remark~\ref{rk:trace-R2},
we can choose a constant~$\lambda$ small enough
(uniformly with respect to~$r$),
so as to obtain
\begin{equation} \label{tau3,3}
 \begin{split}
  \int_{M} \left(\abs{\d\tau_3}^2 + \abs{\d^*\tau_3}^2\right)\vol_g 
  &\lesssim R \int_{\partial \mathcal{B}_r} 
   \abs{j(u_\eps, \, A_\eps)}^2 \, \d\H^{n-1}
   + \int_M \abs{H(\omega_3)}^2 \vol_g 
 \end{split}
\end{equation}
Finally, we estimate the~$L^2(M)$-norm of~$H(\omega_3)$.
We observe that~$\omega_3$ is a bounded measure,
whose total variation is given by
\[
 \abs{\omega_3}\!(M) = \int_{\partial \mathcal{B}_r} 
   \abs{\nu^\flat\wedge j(u_\eps, \, A_\eps)} \, \d\H^{n-1}
   \leq \left(\H^{n-1}(\partial \mathcal{B}_r)\right)^{1/2} 
   \norm{j(u_\eps, \, A_\eps)}_{L^2(\partial \mathcal{B}_r)}
\]
From the definition of the projection operator~$H$,
Equation~\eqref{harmonicprojection}, we immediately obtain
\begin{equation} \label{tau3,4}
 \norm{H(\omega_3)}_{L^1(M)}
 \lesssim \abs{\omega_3}\!(M)
 \lesssim r^{n/2 - 1/2}  
   \norm{j(u_\eps, \, A_\eps)}_{L^2(\partial \mathcal{B}_r)}
\end{equation}
Since all norms on the finite-dimensional
space~$\Harm^2(M)$ are equivalent, Equations~\eqref{tau3,3}
and~\eqref{tau3,4} together imply
\begin{equation} \label{tau3,5}
 \begin{split}
  \int_{M} \left(\abs{\d\tau_3}^2 + \abs{\d^*\tau_3}^2\right)\vol_g 
  &\lesssim \left(R + R^{n-1}\right) \int_{\partial \mathcal{B}_r} 
   \abs{j(u_\eps, \, A_\eps)}^2 \, \d\H^{n-1}
  \stackrel{\eqref{average-r-1}}{\lesssim} E_\eps(x_0, \, R)
 \end{split}
\end{equation}
By combining~\eqref{tau3,1} with~\eqref{tau3,5},
the desired estimate~\eqref{tau3} follows.

\paragraph*{Estimate for~$\tau_2$.}

We claim that, for any~$\mu\in (0, \, 2)$, there exists a
constant~$C_\mu > 0$ (independent on~$\eps$, $R$, $\delta$
and the point~$x_0\in M$) such that
\begin{equation} \label{tau2}
 \int_{\mathcal{B}_{\delta R}} \abs{\d^*\tau_{2}}^2 \, \vol_g 
 \leq C_\mu \, \delta^{n - \mu} \, E_\eps(x_0, \, R)
\end{equation}

We test the equation for~$\tau_2$ against~$\tau_2$.
By applying Young's inequality at the right-hand side,
we obtain
\begin{equation*}
 \begin{split}
  \int_{M} \left(\abs{\d\tau_2}^2 + \abs{\d^*\tau_2}^2\right) \vol_g
  \lesssim \frac{1}{2\lambda} 
   \int_{\mathcal{B}_r} \left(\abs{F_\eps}^2 + \abs{H(\omega_2)}^2\right) \vol_g
   + \lambda \int_{M} \abs{\tau_2}^2 \vol_g
 \end{split}
\end{equation*}
Since~$H$ is the $L^2(M)$-orthogonal projection onto 
the space of harmonic forms, the~$L^2(M)$-norm of~$H(\omega_2)$
is not greater than the $L^2(M)$-norm of~$\omega_2 = - \eta_r F_\eps$.
Moreover, as~$\tau_2$ is orthogonal to all harmonic~$2$-forms,
the~$L^2(M)$-norm of~$\tau_2$ can be estimated
by the~$L^2(M)$-norms of~$\d\tau_2$ and~$\d^*\tau_2$
(see e.g.~\cite[Theorem~4.11]{IwaniecScottStroffolini}).
Therefore, choosing~$\lambda$ small enough,
we obtain
\begin{equation} \label{tau2,1}
 \int_{M} \left(\abs{\d\tau_2}^2 + \abs{\d^*\tau_2}^2\right) \vol_g
 \lesssim \int_{\mathcal{B}_r} \abs{F_\eps}^2 \vol_g 
 \lesssim E_\eps(x_0, \, R)
\end{equation}
Now, let~$s$, $t$ be numbers such that~$0 < s < t < r$.
The decay estimate in Proposition~\ref{prop:decay-RHS}
implies
\begin{equation*}
 \begin{split}
  \int_{\mathcal{B}_s} \left(\abs{\d\tau_2}^2 + \abs{\d^*\tau_2}^2\right) \vol_g
  \lesssim \frac{s^n}{t^n}\int_{\mathcal{B}_t} \left(\abs{\d\tau_2}^2 
   + \abs{\d^*\tau_2}^2\right) \vol_g
   + t^2 \int_{\mathcal{B}_t} \abs{F_\eps}^2 \vol_g
 \end{split}
\end{equation*}
(so long as we choose~$R < R_*$,
where~$R_*$ is a small number that depends on~$M$ only,
as given by Lemma~\ref{lemma:Poincare-R}).
The integral of~$\abs{F_\eps}^2$ at the right-hand side 
can be further bounded by applying
the decay estimate~\eqref{eq:Feps-decay-3}; we obtain
\begin{equation} \label{tau2,2}
 \begin{split}
  \int_{\mathcal{B}_s} \left(\abs{\d\tau_2}^2 + \abs{\d^*\tau_2}^2\right) \vol_g
  \leq \frac{C_1 \, s^n}{t^n}\int_{\mathcal{B}_t} \left(\abs{\d\tau_2}^2 
   + \abs{\d^*\tau_2}^2\right) \vol_g
   + \frac{C_2 \, t^n \, E_\eps(x_0, \, R)}{R^{n-2}}
 \end{split}
\end{equation}
for some positive constants~$C_1$, $C_2$
that depend only on~$M$ (in particular, 
not on~$\eps$, $R$, or~$x_0$). 
Let~$\mu \in (0, \, 2)$ be given.
We choose a number~$\theta\in (0, \, 1)$
small enough that~$C_1\theta^n \leq \theta^{n - \mu/2}$.
Then, the inequality~\eqref{tau2,2} implies
\begin{equation*} 
 \begin{split}
  \int_{\mathcal{B}_{\theta t}} \left(\abs{\d\tau_2}^2 + \abs{\d^*\tau_2}^2\right) \vol_g
  \leq \theta^{n - \mu/2} \int_{\mathcal{B}_t} \left(\abs{\d\tau_2}^2 
   + \abs{\d^*\tau_2}^2\right) \vol_g
   + \frac{C_3 \, t^{n - \mu} \, E_\eps(x_0, \, R)}{R^{n-2}}
 \end{split}
\end{equation*}
for any~$t\in (0, \, r)$ and some constant~$C_3$
that depends only on~$C_2$, $\mu$ and~$R_*$. 
By an iteration argument (see e.g. \cite[Lemma~III.2.1]{Giaquinta-MultipleIntegrals} or~\cite[Lemma~B.3]{Beck}),
we deduce
\begin{equation*}
 \begin{split}
  \int_{\mathcal{B}_t} \left(\abs{\d\tau_2}^2 + \abs{\d^*\tau_2}^2\right) \vol_g
  &\leq \frac{C_\mu \, t^{n - \mu}}{r^{n - \mu}} 
   \int_{\mathcal{B}_r} \left(\abs{\d\tau_2}^2 
   + \abs{\d^*\tau_2}^2\right) \vol_g
   + \frac{C_\mu \, t^{n - \mu} \, E_\eps(x_0, \, R)}{R^{n-2}} \\
  &\hspace{-.22cm} \stackrel{\eqref{tau2,1}}{\lesssim}
   \frac{C_\mu \, t^{n - \mu} \, E_\eps(x_0, \, R)}{R^{n-2}} 
 \end{split}
\end{equation*}
for all~$0 < t < r$, where~$C_\mu$ is a constant 
that depends only on~$\mu$, $n$, $\theta$ and~$C_3$.
The desired estimate~\eqref{tau2} now follows by
taking~$t=\delta R$.

\paragraph{Estimate for~$\tau_1$.}
We will prove that
\begin{equation} \label{tau1}
 \int_{\mathcal{B}_{\delta R}} \abs{\d^*\tau_1}^2
 \lesssim \beta^{-4} E_\eps(x_0, \, R)
 \left(\frac{1}{\eps^2  \, R^{n-2}}
  \int_{\mathcal{B}_R} \left(1 - \abs{u_\eps}^2\right)^2 \vol_g\right) 
\end{equation}
To this purpose, we need an estimate for the ``Green function''
of the Hodge Laplacian, which is provided by Proposition~\ref{prop:Green} 
below. Moreover, we will need the basic bound on curvatures in 
Lemma~\ref{lemma:Linfty-Feps}.

We proceed exactly as in~\cite[pp.~453--454]{BethuelBrezisOrlandi}.
First of all, we observe that
\begin{equation} \label{tau1,1}
 \abs{\omega_1} \lesssim 
  \frac{1}{\beta^2 \, \eps^2} \left(1 - \abs{u_\eps}^2\right)^2
  \qquad \textrm{pointwise in } \mathcal{B}_r.
\end{equation}
Indeed, in the open set~$\{x\in \mathcal{B}_r\colon \abs{u_\eps(x)} > 1 - \beta\}$,
we have~$\abs{v_\eps} = 1$ (by construction)
and hence, $\omega_1 = 2J(v_\eps, \, A_\eps) = 0$.
In the complement, $\{x\in \mathcal{B}_r\colon \abs{u_\eps(x)} \leq 1 - \beta\}$,
we have 
\[
 \abs{\D_{A_\eps} v_\eps} \lesssim \eps^{-1}
\]
because of~\eqref{f-BBO1}--\eqref{f-BBo2} and
the $L^\infty$ estimate~\eqref{eq:L-infty-bound-D_Au} 
for~$\D_{A_\eps} u_\eps$.
Taking Lemma~\ref{lemma:Linfty-Feps} into account,
we deduce
\begin{equation*}
 \abs{\omega_1} 
 \lesssim \abs{J(u_\eps, \, A_\eps)} 
 \lesssim \abs{\D_{A_\eps} u_\eps}^2 + \abs{F_\eps}
 \lesssim \frac{1}{\eps^2}
 = \frac{1}{\beta^2 \, \eps^2} \cdot \beta^2 
 \lesssim \frac{1}{\beta^2 \, \eps^2} \left(1 - \abs{u_\eps}^2\right)^2 
\end{equation*}
at each point of~$\{x\in \mathcal{B}_r\colon \abs{u_\eps(x)} \leq 1 - \beta\}$.
Therefore, \eqref{tau1,1} is proved.

Next, we show that
\begin{equation} \label{tau1,2}
 \norm{\tau_1}_{L^\infty(M)} \lesssim \frac{E_\eps(x_0, \, R)}{\beta^2 \, R^{n-2}}  
\end{equation}
Indeed, for any~$x\in M$, Proposition~\ref{prop:Green}
implies
\begin{equation} \label{tau1,3}
 \abs{\tau_1(x)} 
 \lesssim \int_{\mathcal{B}_r} \frac{\abs{\omega_1(y)}}{\dist(x, \, y)^{n-2}} \vol_g(y)
\end{equation}
(we have used the fact that~$\omega_1 = 0$ out of~$\mathcal{B}_r$).
If~$x\in M\setminus \mathcal{B}_{2r}$, then
\[
 \abs{\tau_1(x)} 
 \lesssim \frac{1}{r^{n-2}} \int_{\mathcal{B}_r} 
  \abs{\omega_1(y)} \vol_g(y)
 \stackrel{\eqref{tau1,1}}{\lesssim}
  \frac{1}{\beta^2 \, \eps^2 \, R^{n-2}} \int_{\mathcal{B}_r} 
  \left(1 - \abs{u_\eps}^2\right)^2 \vol_g
 \lesssim \frac{E_\eps(x_0, \, R)}{\beta^2 \, R^{n-2}}
\]
On the other hand, if~$x\in \mathcal{B}_{2r}$,
then we have $\mathcal{B}_{r} \subseteq \mathcal{B}_{3r}(x)$ and hence,
\[
 \begin{split}
  \abs{\tau_1(x)} 
 \lesssim \int_{\mathcal{B}_{3r}} 
  \frac{\abs{\omega_1(y)}}{\dist(x, \, y)^{n-2}} \vol_g(y)
 \stackrel{\eqref{tau1,1}}{\lesssim}
  \frac{1}{\beta^2 \, \eps^2} \int_{\mathcal{B}_{3r}(x)}
   \frac{(1 - \abs{u_\eps(y)}^2)^2}{\dist(x, \, y)^{n-2}} \, \vol_g(y)
 \end{split}
\]
By applying Corollary~\ref{cor:monotonicity}, we obtain
\[
 \begin{split}
  \abs{\tau_1(x)} 
   \lesssim \frac{E_\eps(u_\eps, \, A_\eps; \, \mathcal{B}_{3r}(x))}{\beta^2 \, r^{n-2}}
   \lesssim \frac{E_\eps(x_0, \, R)}{\beta^2 \, R^{n-2}}
 \end{split}
\]
(for the last inequality,
we have used that~$\mathcal{B}_{3r}(x)\subseteq \mathcal{B}_R$, 
because~$x\in \mathcal{B}_{2r}$ and~$r < R/5$). Therefore,
\eqref{tau1,2} is proved.

Finally, testing the equation for~$\tau_1$
against~$\tau_1$, we obtain
\[
 \int_M\left(\abs{\d\tau_1}^2 + \abs{\d^*\tau_1}^2 \right) \vol_g
 \leq \norm{\tau_1}_{L^\infty(M)} \norm{\omega_1 - H(\omega_1)}_{L^1(M)}
\]
From the definition of the harmonic projector~$H$,
Equation~\eqref{harmonicprojection},
we immediately deduce that $\norm{H(\omega_1)}_{L^1(M)} 
\lesssim \norm{\omega_1}_{L^1(M)}$.
Therefore, taking~\eqref{tau1,1} and~\eqref{tau1,2}
into account, we obtain
\[
 \int_M\left(\abs{\d\tau_1}^2 + \abs{\d^*\tau_1}^2 \right) \vol_g
 \lesssim \beta^{-4} E_\eps(x_0, \, R)
 \left(\frac{1}{\eps^2  \, R^{n-2}}
  \int_{\mathcal{B}_R} \left(1 - \abs{u_\eps}^2\right)^2 \vol_g\right) 
\]
and~\eqref{tau1} follows.

Combining Lemma~\ref{lemma:decay-phixi}
with~\eqref{tau1}, \eqref{tau2}, \eqref{tau3} and~\eqref{tau4},
we obtain a decay estimate for the prejacobian:
for any~$\delta\in (0, \, 1/10)$, $\mu\in (0, \, 2)$
and~$\beta\in (0, \, 1/4)$, there holds
\begin{equation} \label{decay-prejacobian}
  \int_{\delta R} \abs{j(u_\eps, \, A_\eps)}^2 \vol_g
   \leq C_\mu \left(\delta^{n-\mu} + \beta^2 + \beta^{-4}
   \left(\frac{1}{\eps^2  \, R^{n-2}}
  \int_{\mathcal{B}_R} \left(1 - \abs{u_\eps}^2\right)^2 \vol_g\right)
  \right) E_\eps(x_0, \, R)
\end{equation}
where~$C_\mu$ is a constant that depends only on~$\mu$
and the ambient manifold~$M$ (not on~$\eps$, $\delta$, $\beta$, $R$).

\paragraph*{Completing the proof of Proposition~\ref{prop:decay}.}

With~\eqref{decay-prejacobian} at our disposal, 
we can complete the proof of Proposition~\ref{prop:decay}.
The arguments are largely similar to those
in~\cite{BethuelBrezisOrlandi}, so we omit some details. Let
\[
 p_\eps := \frac{1}{\eps^2 R^{n-2}} 
 \int_{\mathcal{B}_R} \left(1 - \abs{u_\eps}^2\right)^2 \vol_g
\]
By reasoning exactly as
in~\cite[Proof of Theorem~3, Step~3, p.~455]{BethuelBrezisOrlandi}, 
we can show that
\begin{equation} \label{decay-absu}
 \begin{split}
  \int_{\mathcal{B}_r} \abs{\d(\abs{u_\eps}^2)}^2 \vol_g
  &\lesssim \beta^2 \int_{\mathcal{B}_R} \abs{\D_{A_\eps} u}^2 \, \vol_g
   + \frac{1}{\beta^{2}\,\eps^2} 
   \int_{\mathcal{B}_R} \left(1 - \abs{u_\eps}^2\right) \vol_p\\
  &\lesssim \beta^2 E_\eps(x_0, \, R) + \beta^{-2} p_\eps \, R^{n-2}
 \end{split}
\end{equation}
(the second inequality follows because, 
by assumption,~$R$ is uniformly bounded from above). 
Moreover, we have
\begin{equation} \label{decay-absDu}
 \int_{\mathcal{B}_r} \left(1 - \abs{u_\eps}^2\right) 
  \abs{\D_{A_\eps} u}^2 \vol_g
 \lesssim \beta^2 E_\eps(x_0, \, R) + \beta^{-2} p_\eps \, R^{2-n}
\end{equation}
(as in Eq.~(III.32) of~\cite{BethuelBrezisOrlandi}).
From Lemma~\ref{lemma:dec-D_Au}, we have
\[
 \begin{split}
  \abs{u_\eps}^2 \, \abs{\D_{A_\eps} u_\eps}^2 
  \lesssim \abs{u_\eps}^2 \, \abs{\d(\abs{u_\eps})}^2 
   + \abs{j(u_\eps, \, A_\eps)}^2
  \lesssim \abs{\d(\abs{u_\eps}^2)}^2 
   + \abs{j(u_\eps, \, A_\eps)}^2
 \end{split}
\]
and hence,
\begin{equation} \label{decay-Du1}
 \begin{split}
  \abs{\D_{A_\eps} u_\eps}^2 
  \lesssim \abs{\d(\abs{u_\eps}^2)}^2 
   + \abs{j(u_\eps, \, A_\eps)}^2
   + \left(1 - \abs{u_\eps}^2\right)^2 \abs{\D_{A_\eps} u_\eps}^2
 \end{split}
\end{equation}
Combining~\eqref{decay-Du1} with~\eqref{decay-prejacobian},
\eqref{decay-absu} and~\eqref{decay-absDu}, we deduce
\begin{equation} \label{decay-Du2}
 \int_{\mathcal{B}_{\delta R}} \abs{\D_{A_\eps} u_\eps}^2 \vol_g
 \leq C_\mu \left(\delta^{n-\mu} + \beta^2 
   + \beta^{-4} \, p_\eps \right) E_\eps(x_0, \, R)
   + \beta^{-2} \, p_\eps \, R^{n - 2}
\end{equation}
for all~$\mu\in (0, \, 2)$. On the other hand,
by~\eqref{eq:decay-Feps},  
there holds
\begin{equation} \label{decay-F}
 \int_{\mathcal{B}_{\delta R}} \abs{F_\eps}^2 \vol_g
 \lesssim \delta^{n-2 + \alpha} \,  E_\eps(x_0, \, R) 
\end{equation}
for some~$\alpha > 0$ that depends only on~$M$
(but is independent of~$R$, $\delta$, $\eps$).
Choosing~$\mu = 2 - \alpha$ in~\eqref{decay-Du2},
and taking~\eqref{decay-F} into account, 
we immediately deduce
\begin{equation*}
 E(\delta R) \lesssim \left(\delta^{n-2+\alpha} + \beta^2 
   + \beta^{-4} \, p_\eps \right) E_\eps(x_0, \, R) + \beta^{-2} \, p_\eps\, R^{n- 2}
\end{equation*}
Finally, choosing~$\beta = \min\{(p_\eps)^{1/6}, \, 1/8\}$,
the proposition follows.
\qed

\subsection{Proof of Proposition~\ref{prop:small-ball-small-energy}}

The energy decay estimate given by Proposition~\ref{prop:decay},
and the monotonicity formula (Theorem~\ref{thm:monotonicity}),
together, imply a ``clearing-out'' or ``$\eta$-ellipticity''
result (cf.~\cite[Theorem~2]{BethuelBrezisOrlandi}).

\begin{prop} \label{prop:clearingout}
 There exist positive numbers~$R_*$, $\eps_*$,
 $C$ and~$\gamma$, depending on~$M$ only, such that the following
 statement holds: for any~$x_0\in M$, $0 < R < R_*$ ,
 $0 < \eps < \eps_* R$, and any critical point
 $(u_\eps, \, A_\eps)\in (W^{1,2}\cap L^\infty)(M, \, E)\times W^{1,2}(M, \, T^*M)$ of~$G_\eps$, there holds
 \begin{equation} \label{clearingout}
  \begin{split}
   \abs{u_\eps(x_0)}
    \geq 1 - C \left(\frac{E_\eps(x_0, \, R)}
     {\log(R/\eps) \, R^{n-2}}\right)^{\gamma}
  \end{split}  
 \end{equation}
\end{prop}

Proposition~\ref{prop:clearingout} is only 
significant when the energy of~$(u_\eps, \, A_\eps)$
on the ball~$\mathcal{B}_R(x_0)$ is small compared 
to the logarithm of~$\eps$. Indeed, the estimate~\eqref{clearingout}
can be rephrased as follows: if
\[
 E_\eps(x_0\, \, R) \leq \eta R^{n-2} \log\frac{R}{\eps}
\]
for some~$\eta > 0$, then
$\abs{u(x_0)} \geq 1 - C \eta^\gamma$. In particular,
if~$\eta$ is small enough, then~$\abs{u_\eps(x_0)} \geq 1/2$.

We proceed to the proof of Proposition~\ref{prop:clearingout}.
Let~$x_0$, $R$, $\eps$ be fixed, and 
let~$(u_\eps, \, A_\eps)(u_\eps, \, A_\eps)\in (W^{1,2}\cap L^\infty)(M, \, E)\times W^{1,2}(M, \, T^*M)$ be a critical point of~$G_\eps$. Let
\begin{equation} \label{eta-clearingout}
 \eta := \frac{E_\eps(x_0, \, R)}{\log(R/\eps) \, R^{n-2}},
\end{equation}
As above, we omit~$x_0$ from the notation when convenient.

\begin{lemma} \label{lemma:goodradius}
 There exists a small number~$\eps_*$,
 depending on~$M$ only, such that the following
 statement holds: if~$0 < \eps < \eps_* R$,
 then for any~$0 < \delta < 1/10$ there exists
 a radius~$R_* \in (\eps^{3/4} R^{1/4}, \, \eps^{1/2} R^{1/2})$ 
 such that
 \begin{gather}
  \frac{E_\eps(x_0, \, R_*)}{R_*^{n-2}} 
   \lesssim \frac{E_\eps(x_0, \, \delta R_*)}{(\delta R_*)^{n-2}}
   + \eta \abs{\log\delta} \label{goodradius-energy} \\
  \frac{1}{\eps^2 R_*^{n-2}} \int_{\mathcal{B}_{R_*}}
   \left(1 - \abs{u_\eps}^2\right)^2 \vol_g
   \lesssim \eta\abs{\log\delta} \label{goodradius-potential}
 \end{gather}
\end{lemma}
\begin{proof}
 The proof follows along the lines of~\cite[Lemma~III.1]{BethuelBrezisOrlandi}.
 For any integer~$j\geq 0$, let 
 \[
  R_j := \eps^{1/2} \, R^{1/2} \left(\frac{\delta}{4}\right)^j
 \]
 Let~$k$ be the unique integer such that
 \begin{equation} \label{goodradius1}
  R_k > \eps^{3/4} \, R^{1/4},
  \qquad R_{k+1} \leq \eps^{3/4} \, R^{1/4}
 \end{equation}
 Let~$C > 0$ be a constant. By the monotonicity formula 
 (Theorem~\ref{thm:monotonicity}) and~\eqref{eta-clearingout}, 
 for a suitable choice of~$C$ (depending on~$M$ only)
 we have
 \[
  \begin{split}
   \sum_{j=0}^{k-1} \left(\frac{e^{CR_j^2} \, E_\eps(x_0, \, R_j)}{R_j^{n-2}} 
    - \frac{e^{CR_{j+1}^2} \, E_\eps(x_0, \, R_{j+1})}{R_{j+1}^{n-2}}\right)
    \leq \frac{e^{CR_0^2} \, E_\eps(x_0, \, R_0)}{R_0^{n-2}} 
    \lesssim \frac{E_\eps(x_0, \, R)}{R^{n-2}} = \eta\log\left(\frac{R}{\eps}\right)
  \end{split}
 \]
 On the other hand, the second inequality
 in~\eqref{goodradius1} implies, via some 
 algebraic manipulation,
 \[
  k+1 \geq \frac{\log(R/\eps)}{4 \log(4/\delta)}
 \] 
 As a consequence, there must be
 an integer~$J\in\{0, \, \ldots, \, k-1\}$
 such that
 \begin{equation} \label{goodradius2}
  \begin{split}
   \frac{e^{CR_J^2} \, E_\eps(x_0, \, R_J)}{R_J^{n-2}} 
    - \frac{e^{CR_{J+1}^2} \, E_\eps(x_0, \, R_{J+1})}{R_{J+1}^{n-2}}
    \lesssim \frac{\eta}{k}\log\left(\frac{R}{\eps}\right)
    \lesssim \eta \log\left(\frac{4}{\delta}\right) 
    \lesssim \eta \abs{\log\delta}
  \end{split}
 \end{equation}
 Due to the mean value theorem, there 
 exists~$R_*\in (R_{J}/2, \, R_J)$ such that
 \begin{equation} \label{goodradius3}
  \left.\frac{\d}{\d\rho}\right|_{\rho = R_*}
  \left(\frac{e^{C\rho^2} E_\eps(x_0, \, \rho)}{\rho^{n-2}}\right) 
  \lesssim \frac{\eta\abs{\log\delta}}{R_*}
 \end{equation}
 Then, the monotonicity formula
 (see Equation~\eqref{improvedmonotonicity})
 and~\eqref{eta-clearingout} imply
 \[
  \begin{split}
   \frac{1}{\eps^2 R_*^{n-2}}
    \int_{\mathcal{B}_{R_*}} \left(1 - \abs{u_\eps}^2\right)^2 \vol_g
   &\lesssim R_* \left.\frac{\d}{\d\rho}\right|_{\rho = R_*}
    \left(\frac{e^{C\rho^2} E_\eps(x_0, \, \rho)}{\rho^{n-2}}\right) 
    + \eta \left(\frac{R_*}{R}\right)^\alpha 
    \frac{E_\eps(x_0, \, R)}{R^{n-2}} \\
   &\lesssim \eta \abs{\log\delta} 
    + \eta \left(\frac{\eps}{R}\right)^{\alpha/2}
    \log\left(\frac{R}{\eps}\right)
  \end{split}
 \]
 If the ratio~$\eps/R$ is small enough,
 then~$(\eps/R)^{\alpha/2}\log(R/\eps) \leq 1
 \leq \abs{\log\delta}$
 for any~$\delta\in (0, \, 1/10)$
 and~\eqref{goodradius-potential} follows.
 As for the proof of~\eqref{goodradius-energy},
 we observe that~$R_J/2 \leq R_* \leq R_J$
 and hence, $\delta R_* \geq R_{J+1}$.
 Then, the same monotonicity formula~\eqref{improvedmonotonicity}
 implies
 \[
  \begin{split}
   &\frac{e^{CR_*^2} \, E_\eps(x_0, \, R_*)}{R_*^{n-2}} 
    - \frac{e^{C(\delta R_*)^2} \, E_\eps(x_0, \, \delta R_*)}{(\delta R_*)^{n-2}}
    + \frac{C^\prime \, E_\eps(x_0, \, R)}{R^{n-2+\alpha}} 
    \int_{\delta R_*}^{R_*} \rho^{\alpha - 1} \, \d\rho \\
   &\hspace{2cm}\leq \frac{e^{CR_J^2} \, E_\eps(x_0, \, R_J)}{R_J^{n-2}} 
    - \frac{e^{CR_{J+1}^2} \, E_\eps(x_0, \, R_{J+1})}{R_{J+1}^{n-2}}
    + \frac{C^\prime \, E_\eps(x_0, \, R)}{R^{n-2+\alpha}} 
    \int_{R_{j+1}}^{R_j} \rho^{\alpha - 1} \, \d\rho 
  \end{split}
 \]
 for some constant~$C^\prime$ that 
 depends only on~$M$. Therefore, taking~\eqref{eta-clearingout}
 and~\eqref{goodradius2} into account, we deduce
 \[
  \begin{split}
   \frac{e^{CR_*^2} \, E_\eps(x_0, \, R_*)}{R_*^{n-2}} 
    - \frac{e^{C(\delta R_*)^2} \, E_\eps(x_0, \, \delta R_*)}{(\delta R_*)^{n-2}}
   &\lesssim \eta \abs{\log\delta} + \eta \left(\frac{R_J}{R}\right)^\alpha 
    \frac{E_\eps(x_0, \, R)}{R^{n-2}} \\
   &\lesssim \eta \abs{\log\delta} 
    + \eta \left(\frac{\eps}{R}\right)^{\alpha/2}
    \log\left(\frac{R}{\eps}\right)
  \end{split}
 \]
 If~$\eps/R$ is small enough, the
 desired estimate~\eqref{goodradius-energy} follows. 
\end{proof}

\begin{proof}[Proof of Proposition~\ref{prop:clearingout}]
 We proceed as in~\cite[Theorem~2, Part~C, p.~456]{BethuelBrezisOrlandi}.
 Let~$\eta_0 > 0$ be a small parameter, to be
 chosen later on (depending on~$M$ only).
 Let~$\eta$ be defined as in~\eqref{eta-clearingout}.
 If~$\eta > \eta_0$, then Proposition~\ref{prop:clearingout}
 holds true for any value of~$\gamma > 0$,
 so long as we choose a constant~$C$ large enough
 (depending on~$\gamma$, $\eta_0$ only).
 Therefore, it suffices to consider the case~$\eta \leq \eta_0$.
 Let~$\delta\in (0, \, 1/10)$ be fixed, and 
 let~$R_*\in (\eps^{3/4} R^{1/4}, \, \eps^{1/2} R^{1/2})$
 be a radius that satisfies~\eqref{goodradius-energy} 
 and~\eqref{goodradius-potential}, as given 
 by Lemma~\ref{lemma:goodradius}. For such a
 choice of~$R_*$, Proposition~\ref{prop:decay}
 implies
 \[
  \begin{split}
   \frac{E_\eps(x_0, \, R_*)}{R_*^{n-2}} 
   &\lesssim \frac{E_\eps(x_0, \, \delta R_*)}{(\delta R_*)^{n-2}} + \eta\abs{\log\delta}\\
   &\lesssim \left(\frac{(\eta\abs{\log\delta})^{1/3}}{\delta^{n-2}}
    + \delta^{\alpha}\right)
    \frac{E_\eps(x_0, \, R_*)}{R_*^{n-2}} + \frac{(\eta\abs{\log\delta})^{2/3}}{\delta^{n-2}}
    + \eta\abs{\log\delta}
  \end{split}
 \]
 Assume~$\delta$ is such that~$\eta\abs{\log\delta}\leq 1$.
 Then, we obtain
 \[
  \begin{split}
   \left(1 - \frac{C(\eta\abs{\log\delta})^{1/3}}{\delta^{n-2}}
    - C \delta^{\alpha}\right)\frac{E_\eps(x_0, \, R_*)}{R_*^{n-2}} 
   \lesssim \frac{(\eta\abs{\log\delta})^{2/3}}{\delta^{n-2}}
  \end{split}
 \]
 for some uniform constant~$C$. We choose~$\delta = \eta^{1/(3n)}$ 
 and we take~$\eta_0$ so small that $\eta_0\abs{\log\eta_0} \leq 1$ and
 $C \eta_0^{2/(3n)}\abs{\log\eta_0} + C \eta_0^{\alpha/(3n)}\leq 1/2$.
 Then, we deduce
 \[
  \begin{split}
   \frac{E_\eps(x_0, \, R_*)}{R_*^{n-2}} 
   \lesssim \eta^{(n+2)/(3n)}\abs{\log\eta}^{2/3}
  \end{split}
 \]
 for any~$\eta\leq\eta_0$.
 Moreover, assuming~$\eps < R$,
 we have~$R_* > \eps^{3/4} R^{1/4} > \eps$
 and hence, the monotonicity formula 
 (Theorem~\ref{thm:monotonicity}) implies
 \[
  \frac{1}{\eps^n}\int_{\mathcal{B}_\eps}\left(1 - \abs{u_\eps}^2\right)^2\vol_g
  \lesssim \frac{E_\eps(x_0, \, \eps)}{\eps^{n-2}} 
  \lesssim \frac{E_\eps(x_0, \, R_*)}{R_*^{n-2}} 
  \lesssim \eta^{(n+2)/(3n)}\abs{\log\eta}^{2/3}
 \]
 Now the proposition follows by repeating
 the arguments in~\cite[Lemma~III.3]{BethuelBrezisOrlandi}.
\end{proof}

Finally, we are in a position to prove the main result of this section,
Proposition~\ref{prop:small-ball-small-energy}.

\begin{proof}[{Proof of Proposition~\ref{prop:small-ball-small-energy}}]
	In view of Proposition~\ref{prop:clearingout}, 
	the property~\eqref{eq:clearingout-small-ball} 
	follows by the very same argument as
	in~\cite[Proposition~VII.1]{BethuelBrezisOrlandi}.
	The proof of~\eqref{eq:no-energy-small-balls} is also
	ispired by~\cite[Proposition~VII.1]{BethuelBrezisOrlandi}.
	Suppose, towards a contradiction, that~\eqref{eq:no-energy-small-balls}
	fails. Then, there exists a (non-relabelled) subsequence such that
	\begin{equation} \label{eq:no-energy-small-balls-0}
	 E_\eps(x_0, \, R/2) \to +\infty \qquad \textrm{as } \eps\to 0.
	\end{equation}
	We split the rest of the proof in steps.
	
   \setcounter{step}{0}
   \begin{step}
	Locally, on the ball~$\mathcal{B}_{R} := \mathcal{B}_R(x_0)$,
	we may write the reference connection as~$\D_0 = \d - i \gamma_0$
	for some (smooth, $\eps$-independent) $1$-form~$\gamma_0$.
	Upon replacing~$A_\eps$ with~$A_\eps + \gamma_0$, from now on 
	we assume that $\D_{A_\eps} = \d - i A_\eps$ on~$\mathcal{B}_R$.
	Moreover, we assume that~$A_\eps$ is in Coulomb gauge on~$\mathcal{B}_R$
	--- that is, $A_\eps$ satisfies~\eqref{eq:Feps-comp1}.
	Then, the Gaffney inequality (see 
	e.g.~Proposition~\ref{prop:gaffney-rho} in
	the appendix for details) implies
	\begin{equation*}
	 \norm{A_\eps}_{L^2(\mathcal{B}_R)}
	 \lesssim R\norm{F_\eps}_{L^2(\mathcal{B}_R)}
	 \stackrel{\eqref{eq:small-resclaed-energy}}{\lesssim}
	 R^{n/2} \left(\log\frac{R}{\eps}\right)^{1/2}
	\end{equation*}
	Moreover, Sobolev embeddings and the Gaffney inequality again
	imply
	\begin{equation} \label{eq:no-energy-small-balls-1}
	 \norm{A_\eps}_{L^{2^*}(\mathcal{B}_R)}
	 \leq C_R \norm{A_\eps}_{W^{1,2}(\mathcal{B}_R)}
	 \leq C_R\left(\norm{A_\eps}_{L^2(\mathcal{B}_R)}
	  + \norm{F_\eps}_{L^2(\mathcal{B}_R)}\right)
	 \leq C_R \abs{\log\eps}^{1/2}
	\end{equation}
	for~$2^*:= 2n/(n-2)$ and some constant~$C_{R}$
	depending on~$R$ (and, possibly, on~$x_0$).
	On the other hand, the energy bound~\eqref{eq:small-resclaed-energy}
	and the $L^\infty$-estimate given by 
	Proposition~\ref{prop:L-infty-bound-u} imply, by interpolation,
	\begin{equation} \label{eq:no-energy-small-balls-2}
	 \norm{1 - \abs{u_\eps}^2}_{L^n(\mathcal{B}_R)}
	 \leq \norm{1 - \abs{u_\eps}^2}_{L^2(\mathcal{B}_R)}^{2/n}
	  \norm{1 - \abs{u_\eps}^2}_{L^\infty(\mathcal{B}_R)}^{1 - 2/n}
	  \leq C_R \, \eps^{2/n} \abs{\log\eps}^{1/n}
	\end{equation}
	From~\eqref{eq:no-energy-small-balls-1} and~\eqref{eq:no-energy-small-balls-2}, we conclude that
	\begin{equation} \label{eq:no-energy-small-balls-3}
	 \norm{\left(1 - \abs{u_\eps}^2\right) A_\eps}_{L^2(\mathcal{B}_R)}
	 \to 0 \qquad \textrm{as } \eps\to 0
	\end{equation}
   \end{step}
   
   \begin{step}
	On the smaller ball~$\mathcal{B}_{3R/4} := \mathcal{B}_{3R/4}(x_0)$,
	the function~$\rho_\eps := \abs{u_\eps}$ is bounded away from zero,
	thanks to~\eqref{eq:clearingout-small-ball}. Therefore,
	by identifying~$u_\eps$ with a map~$\mathcal{B}_{3R/4}\to\C$,
	we find a (Lipschitz-continuous)
	function~$\theta_\eps\colon\mathcal{B}_{3R/4}\to\R$ such that
	$u_\eps = \rho_\eps\exp(i\theta_\eps)$.
	By direct computation, we see that
	\begin{equation} \label{eq:no-energy-small-balls-4}
	 j(u_\eps, \, A_\eps) = \rho_\eps^2 \left(\d\theta_\eps - A_\eps\right)
	\end{equation}
	Assume that~$(u_\eps, \, A_\eps)$ satisfies~\eqref{hp:logenergy}.
	Then, up to extraction of a subsequence, we know that
	$j(u_\eps, \, A_\eps)$ converges in~$L^p(M)$
	and~$F_\eps$ converges in~$W^{1,p}(M)$
	for any~$p < n/(n-1)$
	(by Lemma~\ref{lemma:Lp-bound-prejac}
	and~\eqref{compactnessF,J}, respectively).
	Since~$A_\eps$ is in Coulomb gauge
	(Equation~\eqref{eq:Feps-comp1}), from the Gaffney inequality
	(Proposition~\ref{prop:gaffney-rho}) we deduce
	that~$A_\eps$ is bounded in~$L^p(\mathcal{B}_R)$
	for~$p < n/(n-1)$. As a consequence, from~\eqref{eq:clearingout-small-ball}
	and~\eqref{eq:no-energy-small-balls-4} we obtain that~$\d\theta_\eps$
	is bounded in~$L^p(\mathcal{B}_{3R/4})$. Up to
	subtracting a constant multiple of~$2\pi$,
	we can moreover assume that the average of~$\theta_\eps$
	on~$B_{3R/4}$ belongs to the interval~$[0, \, 2\pi)$;
	then, it follows
	\begin{equation} \label{eq:no-energy-small-balls-5}
	 \norm{\theta_\eps}_{W^{1,p}(\mathcal{B}_{3R/4})} \leq C_{R,p}
	\end{equation}
	for any~$p < n/(n-1)$ and some constant~$C_{R,p}$
	depending on~$R$ and~$p$ (and possibly, on~$x_0$), but not on~$\eps$.
   \end{step}
   
   \begin{step}
    The Euler-Lagrange equation~\eqref{EL-A} implies that
    the form~$j(u_\eps, \, A_\eps)$ is co-exact --- in particular,
    co-closed. The form $A_\eps$ is co-closed in~$\mathcal{B}_{3R/4}$,
    too, because of our choice of gauge~\eqref{eq:Feps-comp1}.
    Therefore, we have
    \[
     0 = \d^* j(u_\eps, \, A_\eps) + \d^* A_\eps
     \stackrel{\eqref{eq:no-energy-small-balls-4}}{=}
     \d^*\left(\rho_\eps^2 \d\theta_\eps - \rho_\eps^2 A_\eps + A_\eps\right)
    \]
    and hence,
    \begin{equation} \label{eq:no-energy-small-balls-6}
     \d^*\left(\rho_\eps^2 \d\theta_\eps\right) =  
     \d^*\left((\rho_\eps^2 - 1)A_\eps\right)
    \end{equation}
    By a suitable Caccioppoli inequality
    (see e.g. Lemma~\ref{lemma:CaccioppoliL1} in the appendix),
    we deduce
    \begin{equation*}
     \norm{\d\theta_\eps}_{L^2(\mathcal{B}_{2R/3})} 
	 \lesssim R^{-n/2 - 1} \norm{\theta_\eps}_{L^1(\mathcal{B}_{3R/4})}
	 + \norm{(1 - \rho_\eps^2)A_\eps}_{L^2(\mathcal{B}_{3R/4})}
    \end{equation*}
    and hence, thanks to~\eqref{eq:no-energy-small-balls-3}
    and~\eqref{eq:no-energy-small-balls-5}, 
    \begin{equation} \label{eq:no-energy-small-balls-7}
     \norm{\d\theta_\eps}_{L^2(\mathcal{B}_{2R/3})} 
	 \leq C_R
    \end{equation}
    for some constant~$C_R$ that depends on~$R$
    (and possibly, on~$x_0$), but not on~$\eps$.
   \end{step}
   
   \begin{step}
    Due to the choice of gauge~\eqref{eq:Feps-comp0},
    the Euler-Lagrange equation~\eqref{EL-A} rewrites as
    \begin{equation} \label{eq:no-energy-small-balls-8}
     -\Delta A_\eps = j(u_\eps, \, A_\eps)
     \qquad \textrm{on } \mathcal{B}_R.
    \end{equation}
    Combining~\eqref{eq:no-energy-small-balls-8} 
    with~\eqref{eq:no-energy-small-balls-4}, we deduce
    \begin{equation} \label{eq:no-energy-small-balls-9}
     -\Delta A_\eps + \rho^2_\eps A_\eps = \rho_\eps^2 \,\d\theta_\eps
     \qquad \textrm{on } \mathcal{B}_R.
    \end{equation}
    The right-hand side of~\eqref{eq:no-energy-small-balls-9}
    is uniformly bounded in~$L^2(\mathcal{B}_{3R/4})$,
    because of~\eqref{eq:no-energy-small-balls-7}.
    A Cac\-ciop\-po\-li-type inequality (see Lemma~\ref{lemma:CaccioppoliL1bis}
    in the appendix for details) now implies
    \begin{equation} \label{eq:no-energy-small-balls-10}
     \norm{A_\eps}_{L^2(\mathcal{B}_{7R/12})} 
      + \norm{\d A_\eps}_{L^2(\mathcal{B}_{7R/12})}\leq C_R
    \end{equation}
    for some constant~$C_R$ that depends on~$R$, but ot on~$\eps$.
   \end{step}
   
   \begin{step}
    Finally, it remains to estimate the 
    $W^{1,2}$-norm of~$\rho_\eps := \abs{u_\eps}$.
    By explicit computation, recalling that~$u_\eps = \rho_\eps e^{i\theta_\eps}$ and~$\D_{A_\eps} = \d - i A_\eps$, we have
    \[
     -\frac{1}{2}\Delta\left(\rho_\eps^2\right) 
      + \abs{\D_{A_\eps} u_\eps}^2
     = - \rho_\eps \Delta \rho_\eps 
      + \rho_\eps^2 \abs{\d\theta_\eps - A_\eps}^2
    \]
    By injecting this identity into Equation~\eqref{EL-abs},
    and recalling that~$\rho_\eps \geq 1/2$ in~$\mathcal{B}_{3R/4}$,
    we obtain 
    \begin{equation} \label{eq:no-energy-small-balls-11}
     -\Delta\rho_\eps + \frac{1}{\eps^2} 
      \rho_\eps \left(\rho_\eps^2 - 1\right) + \rho_\eps \abs{\d\theta_\eps - A_\eps}^2 = 0 \qquad \textrm{in } \mathcal{B}_{3R/4}.
    \end{equation}
    We test Equation~\eqref{eq:no-energy-small-balls-11}
    against~$(1 - \rho_\eps)\zeta$, where~$\zeta\in C^\infty_{\mathrm{c}}(\mathcal{B}_{7R/12})$ is a cut-off function, such that
    $\zeta = 1$ in~$\mathcal{B}_{R/2}$ and~$\abs{\nabla\zeta}\lesssim R^{-1}$.
    We obtain
    \[
     \begin{split}
      \int_{\mathcal{B}_{7R/12}} \left(\abs{\nabla\rho_\eps}^2 
       + \frac{1}{\eps^2}\rho_\eps^2(1 + \rho_\eps)(1 - \rho_\eps)^2 \right)\zeta \, \vol_g
       &= \int_{\mathcal{B}_{7R/12}} (\rho_\eps - 1) \ip{\nabla \rho_\eps}{\nabla \zeta} \vol_g  \\
       &\qquad + \int_{\mathcal{B}_{7R/12}} 
        \rho_\eps (1 - \rho_\eps)\,\zeta 
        \abs{\d\theta_\eps - A_\eps}^2 \vol_g
     \end{split}
    \]
    The second term in the right-hand side is bounded
    uniformly with respect to~$\eps$,
    due to~\eqref{eq:no-energy-small-balls-7},
    \eqref{eq:no-energy-small-balls-10} and
    the bound~$\abs{\rho_\eps} \leq 1$. Therefore,
    upon applying the Young inequality at the right-hand side,
    we deduce that
    \begin{equation} \label{eq:no-energy-small-balls-12}
     \int_{\mathcal{B}_{R/2}} \left(\abs{\nabla\rho_\eps}^2 
       + \frac{1}{\eps^2}(1 - \rho_\eps^2)^2 \right)\vol_g \leq C_R
    \end{equation}
    for some~$\eps$-independent constant~$C_R$.
    From~\eqref{eq:no-energy-small-balls-7}, \eqref{eq:no-energy-small-balls-10} and~\eqref{eq:no-energy-small-balls-12},
    we obtain the energy bound~$E_\eps(x_0, \, R)\leq C_R$,
    which contradicts~\eqref{eq:no-energy-small-balls-0}.
    This completes the proof.
    \qedhere
   \end{step}
\end{proof}

\section{Convergence to the limiting varifold}
\label{sec:varifold}
In this section we consider the \emph{rescaled} energy measures, defined as
\begin{equation}\label{eq:mu-eps}
	\mu_\eps := \frac{e_\eps(u_\eps,\,A_\eps)}{\pi \abs{\log\eps}} \,\vol_g
\end{equation}
where we identify~$\vol_g$ with a Radon measure on~$M$, in a canonical way.
We prove that the measures~$\mu_\eps$ converge 
weakly$^*$ in the sense of Radon measures in $M$ to a limiting Radon measure 
$\mu_*$ and that $\mu_*$ is the weight measure of a stationary, rectifiable 
$(n-2)$-varifold in $M$. For a precise formulation of this result, 
cf.~Theorem~\ref{thm:rectifiability} below.

A \emph{$k$-varifold}~$V$ in a Riemannian $n$-manifold~$M$
is defined as a nonnegative
Radon measure on $G_k(\T M)$, where $k \leq n$ is an integer and $G_k(\T M)$ is 
the (total space of the) Grassmanian bundle on~$M$, a bundle whose typical 
fibre is the Grassmanian manifold of $k$-planes in the tangent space to~$M$. 
(Details on the construction of $G_k(\T M)$
can be found, for instance, in \cite[Section~6.5.7]{Monclair2021}.)
We will denote the set of $k$-varifolds in~$M$ by $\mathcal{V}_k(M)$.

\begin{remark} \label{rk:varifolds}
Varifolds lack a boundary operator and
there is no natural notion of orientation on them. 
Thus, varifolds are less regular objects than currents (indeed, any current 
induces a varifold but the converse does not hold), although they are both 
generalisations of the concept of smooth submanifold.
\end{remark}

As in~\cite{PigatiStern}, one wishes to prove the rectifiability of the 
limiting varifold by applying (a suitable variant of) the 
Ambrosio-Soner rectifiability criterion for generalised varifolds,
i.e.~Theorem~3.8 
of~\cite{AmbrosioSoner}. Since in \cite{AmbrosioSoner} the 
authors work in the Euclidean setting, an extension of the Ambrosio-Soner 
theorem to the present setting is needed but, as explained in 
\cite[p.~1058]{PigatiStern}, easy to obtain. For the reader's convenience, we 
recall the main points of the argument below.

Following \cite[Section~6]{PigatiStern}, we use the metric $g$ to canonically 
identify tensors of rank $(2,\,0)$, $(1,\,1)$ and $(0,\,2)$. We denote by 
$\EndSym(\T M)$ the space of symmetric endomorphisms of $\T M$ and we define a 
subbundle $\mathcal{A}_{k,n}(M)$ of $\EndSym(\T M)$ as follows: 
\[
	\mathcal{A}_{k,n}(M) := \left\{ S \in \EndSym(\T M) : -n g \leq S \leq g \mbox{ and } {\rm Tr}(S) \geq k  \right\}.
\]
The typical fibre of $\mathcal{A}_{k,n}(M)$ at $x \in M$ is the set 
$\mathcal{A}_{k,n}(\T_x M)$ of symmetric endomorphisms 
$S_x\colon \T_x M \to \T_x M$ with trace $\geq k$ and such that 
$-n g_x \leq S_x \leq g_x$ as bilinear forms on $\T_x M$.
 
We define a \emph{generalised $k$-varifold in $M$} as a nonnegative Radon 
measure on $\mathcal{A}_{k,n}(M)$, thus extending in the most natural way the 
definition in \cite[Section~3]{AmbrosioSoner} 
(cf.~\cite[Section~6]{PigatiStern}). We denote $\mathcal{V}^*_k(M)$ the set of 
generalised $k$-varifolds on $\mathcal{A}_{k,n}(M)$

Let $V \in \mathcal{V}^*_{k}(M)$
and let~$\pi$ be the canonical projection
$\mathcal{A}_{k,n}(M) \to M$. We call $\|V\| := \pi_* V$ the 
\emph{weight measure} of $V$, where 
$\pi$ denotes the canonical projection $\mathcal{A}_{k,n}(M) \to M$ and 
$\pi_* V$ denotes the pushforward measure of~$V$ through~$\pi$. 
We say that $V$ \emph{has first variation} if and only if there exists a 
vector Radon measure $\delta V \in C(M, \, \T M)^\prime$ 
such that, for any $X \in C^1(M,\,\T M)$,
\begin{equation} \label{firstvariation}
	\int_{\mathcal{A}_{k,n}(M)} \ip{ S_x}{\nabla X(x)} \, \d V(x, S_x) 
	= -\int_M \ip{X}{\nu}\,\d{\norm{\delta V}},
\end{equation}
where $\nabla$ denotes the Levi Civita connection of 
$M$, $S_x \in \mathcal{A}_{k,n}(\T_x M)$, $\ip{S_x}{\nabla X(x)}$ 
denotes the scalar product of the matrices representing $S_x$ and 
$\nabla X(x)$, regarded as endomorphisms of $\T_x M$, and $\nu$ is a 
$\norm{\delta V}$-measurable vector field on $M$ with $\abs{\nu}=1$ 
$\norm{\delta V}$-a.e. in $M$.
When~\eqref{firstvariation} holds, we say that $\delta V$ is the \emph{first variation} of~$V$. 
We call~$V$ \emph{stationary} if $\delta V(X) = 0$ for all 
$X \in C^1(M,\,\T M)$.
We say that a sequence $\{V_h\} \subset \mathcal{V}^*_k(M)$ \emph{converges weakly$^*$} to $V \in \mathcal{V}^*_k(M)$ if and only if $\{V_h\}$ 
converges weakly$^*$ in the sense of measures on~$\mathcal{A}_{k,n}(M)$
to~$V$. Next, we define the \emph{upper $k$-dimensional density} of $V$ 
at $x \in M$ as 
\[
	\Theta^*_k(\|V\|,\,x) := 
	\limsup_{r \downarrow 0} \frac{\|V\|(\mathcal{B}_r(x))}{\mu(\mathcal{B}_r(x))}.
\]
Keeping in mind the construction of $G_k(\T M)$, a standard localisation 
argument shows that the 
following version of \cite[Theorem~3.8(c)]{AmbrosioSoner} holds 
(cf.~\cite[Proposition~5.1]{Stern2020})
\begin{theorem}[Rectifiability criterion]\label{thm:rect-gen-varifolds}
	Let $V \in \mathcal{V}^*_k(M)$ be a generalised $k$-varifold with first 
	variation and assume that 
	\[
		\Theta^*_k(\|V\|,\,x) > 0 \qquad \mbox{for } \|V\|{\mbox-a.e.}\, x \in M.
	\]
	Then, there exists a (classical) $k$-rectifiable varifold $\widetilde{V}$ 
	such that 
	\[
		\norm{V} = \|\widetilde{V}\| \quad \mbox{and}\quad 
		\delta V = \delta \widetilde{V}.
	\]
\end{theorem}

We follow the strategy of \cite[Section~6]{PigatiStern}, applying the 
modifications made necessary by the logarithmic energy regime. 

For any given 
$u \in (W^{1,2}\cap L^\infty)(M,E)$ and $A \in W^{1,2}(M,\,\T^*M)$, 
we define the 
$(0,2)$ tensors $\D_A u^* \D_A u$ and $F_A^* F_A$ by
\[
	\D_A u^* \D_A u(E_i,E_j) := \sum_{i,j = 1}^n \ip{\D_{A,E_i} u}{\D_{A,E_j} u},
\]
and, respectively,
\[
	F_A^* F_A(E_i,E_j) := \sum_{i,j=1}^n \ip{F_A(E_i,\,E_j)}{F_A(E_i,\,E_j)},
\]
where $\{E_1, \dots, E_n\}$ is any local orthonormal frame for $E \to M$. 
It is easily checked that $\D_A u^* \D_A u$ and $F_A^* F_A$ are 
gauge-invariant and smooth.

Next, for any $\eps > 0$, we consider the
\emph{rescaled} stress-energy tensor field
\begin{equation} \label{stresstensor}
	T_\eps \equiv T_\eps(u_\eps,\,A_\eps) 
	:= \frac{1}{\abs{\log\eps}} \left( e_\eps(u_\eps,\,A_\eps) g - \D_A u^* \D_A u - F_A^* F_A \right), 
\end{equation}
where $(u_\eps,\,A_\eps)$ is any critical point of $G_\eps$. 
By the same 
computations as in \cite[Section~4]{PigatiStern}, we see that $T_\eps$ is 
divergence-free, for any $\eps > 0$. 
By means of the metric, we can identify $T_\eps$ with the induced 
$\EndSym(\T M)$-valued Radon measure in $M$. If we can show that 
the sequence $\{T_\eps\}$ 
converges to a limiting $\EndSym(\T M)$-valued measure $T_*$ and that 
$\abs{T_*}$ is absolutely continuous with respect to $\mu_*$, then we can 
represent $T_*$ by a generalised $(n-2)$-varifold $P$. The necessary estimates 
to this purpose are provided by Lemma~\ref{lemma:T_*}. These estimates ensure, 
in addition, that we can apply Theorem~\ref{thm:rect-gen-varifolds} to obtain 
the rectifiability of $P$.

The main result of this section is 
Theorem~\ref{thm:rectifiability} below (cf. \cite[Proposition~6.4]{PigatiStern}). 

\begin{theorem}\label{thm:rectifiability}
	Let $\{\mu_\eps\}$, where $\mu_\eps$ is defined
	by~\eqref{eq:mu-eps}, be the rescaled energy measures of a 
	sequence $\{(u_\eps,\,A_\eps)\}\subset \mathcal{E}$
	of critical points of $G_\eps$ satisfying the 
	logarithmic energy bound~\eqref{hp:logenergy}.
	Then, there exist a bounded Radon measure $\mu_*$ on $M$ and a (not 
	relabelled) subsequence 
	such that $\mu_\eps \rightharpoonup^* \mu_*$ weakly$^*$
	in the sense of measures. Moreover, the $(n-2)$-density 
	$\Theta(\mu_*,x_0) := \lim_{r \downarrow 0} r^{2-n} \mu_*(B_r(x_0))$ of $\mu_*$ 
	at $x_0$ is defined at every $x_0 \in M$ and
	$\mu_*$ is the weight measure of an associated stationary, rectifiable 
	$(n-2)$-varifold $V = v( \Sigma, \theta)$, which in turn is such that
	\[
		\lim_{\eps \to 0} 
		\int_M \ip{T_\eps(u_\eps,\,A_\eps)}{S}\,\vol_g 
		= \int_\Sigma \theta(x) \ip{\T_x \Sigma}{S(x)}\,\d \mathcal{H}^{n-2}(x)
	\]
	for all $S \in C^0(M,\, \EndSym(\T M))$, 
	where $\Sigma := \spt \mu_*$, $\theta(x) := \Theta_{n-2}(\mu_*, x)$, and 
	$\mu_* = \theta \mathcal{H}^{n-2} \mres \Sigma$.
\end{theorem}

\begin{remark} \label{rk:nodiffuse}
The limiting measure~$\mu_*$ given by Theorem~\ref{thm:rectifiability}
is purely concentrated on a co\-di\-men\-sion-two varifold
and it does not 
contain any ``diffuse part'' (i.e., a contibution absolutely continuous with respect to~$\vol_g$). 
The latter is instead typical of the ``non-magnetic'' Ginzburg-Landau 
functional 
\begin{equation} \label{nonmagnetic_GL}
	I_\eps(v) = \int_M \frac{1}{2} \abs{\d v}^2 + 
	\frac{1}{4\eps^2}\left( 1 - \abs{v}^2 \right)^2\,\vol_g
	\qquad \textrm{for } v\in W^{1,2}(M, \, \C)
\end{equation}
without further constraints. For instance, it is proven in \cite{Stern2021} 
that the rescaled energy measures of critical points 
$v_\eps$ of $I_\eps$ (satisfying 
$I_\eps(v_\eps) \sim \On(\abs{\log\eps})$ as $\eps \to 0$) 
concentrate towards a measure of the form 
$\widetilde{\mu} := \norm{V} + \abs{\psi}^2\,\vol_g$, 
where $\psi$ is a possibly non-trivial harmonic one-form on~$M$. 
A similar phenomenon occurs also in the evolutionary case, even in 
the Euclidean setting (i.e., in  $\R^n \times \R^+$). 
Indeed, it is proven in 
\cite[Theorem~A]{BethuelOrlandiSmets-Annals}
that for solutions to the Ginzburg-Landau heat flow,
in the logarithmic energy regime,
the rescaled parabolic energy measures converge to measures of the form
$\widehat{\mu}(x,t) = \|\widehat{V}\| + |\nabla \widehat{\Phi}|^2\,\d x$, 
where $\widehat{\Phi}$ is a solution of the heat equation in 
$\R^n \times \R^+$.
In rough terms, the diffuse part in the limiting measure arises because of wild 
oscillations in the phase made possible by the high energy at disposal. 
In the situation considered in \cite{BethuelBrezisOrlandi}, this phenomenon is 
ruled out by topological arguments. 
The same would occur 
if the manifold~$M$ is simply connected.
In Theorem~\ref{thm:rectifiability}, there 
is no diffuse part in~$\mu_*$ because of 
the structure of the gauge-invariant functional,
which allows us to remove the non-topological contributions 
to the energy, to leading order, by a suitable choice of 
the connection~$\D_{A}$.
This is independent of the topology of $M$; it is 
a feature of the energy functional~$G_\eps$.
\end{remark}

The first step in the proof of Theorem~\ref{thm:rectifiability} is the 
characterisation of the limiting measure $\mu_*$.
\begin{lemma}\label{lemma:existence+density-mu_*}
	Let $\{\mu_\eps\}$, where $\mu_\eps$ is defined 
	by~\eqref{eq:mu-eps},
	be the rescaled energy measures of a sequence 
	$\{(u_\eps,\,A_\eps)\}\subset \mathcal{E}$ of critical points of 
	$G_\eps$ satisfying the logarithmic energy bound~\eqref{hp:logenergy}. 
	Then, there exist
	a bounded Radon measure $\mu_*$ on $M$ and a 
	(not relabeled) subsequence such that 
	$\mu_\eps \rightharpoonup \mu_*$ weakly$^*$ in the sense of measures. 
	Moreover, there exist constants 
	$C > 0$ and $R_0 \in (0,\,\inj(M))$, depending on $(M,g)$ only, 
	such that the measure $\mu_*$ satisfies
	\begin{equation}\label{eq:monotonicity-measures}
		r^{2-n}\mu_*(\mathcal{B}_r(x_0)) + \frac{C e^{CR^2/2}}{R^{n-2+\alpha}}r^\alpha \mu_*(\mathcal{B}_R(x_0)) 
		\leq R^{2-n}\mu_*(\mathcal{B}_R(x_0))(1 + C e^{C R^2 /2}), 
	\end{equation}
	for every~$x_0\in M$, every~$0 < r < R \leq R_0$
	and some~$\alpha > 0$ depending on~$(M, \, g)$ only. As a consequence, 
	the \emph{$(n-2)$-density} 
	\begin{equation}\label{eq:density}
		\Theta_{n-2}(\mu_*,x_0) := \lim_{r \downarrow 0} r^{2-n} \mu_*(\mathcal{B}_r(x_0))
	\end{equation}
	is defined at every $x_0 \in M$.
\end{lemma}

\begin{proof}
	Up to small modifications, the proof is as in \cite[Section~6]{PigatiStern}. 
	The existence of $\mu_*$ is standard, because the uniform bound 
	$\frac{G_\eps(u_\eps,\,A_\eps)}{\abs{\log\eps}} \leq \Lambda$ immediately 
	yields a uniform bound on $\norm{\mu_\eps}_{(C^0(M))^*}$. The existence of the limit in the right 
	hand side of \eqref{eq:density} is a straightforward consequence 
	of~\eqref{eq:monotonicity-measures}. Hence, it suffices to prove~\eqref{eq:monotonicity-measures}.
	To the last purpose, fix arbitrarily $x_0 \in M$. Let $C > 0$ and 
	let $R_0 \in (0, \inj(M))$ be the same constants (depending on $(M,g)$ only) 
	as in Theorem~\ref{thm:monotonicity}. 
	Then, by the definition on $\mu_\eps$ and Theorem~\ref{thm:monotonicity},
	\[
	\begin{split}
		R^{2-n}\mu_*(\overline{\mathcal{B}_R(x_0)})(1 + C e^{C R^2 /2}) &\geq 
		\limsup_{\eps \downarrow 0} R^{2-n}\mu_\eps(\overline{\mathcal{B}_R(x_0)})(1 + C e^{C R^2 /2}) \\
		&= \limsup_{\eps \downarrow 0} R^{2-n}\mu_\eps(\mathcal{B}_R(x_0))(1 + C e^{C R^2 /2}) \\
		&\geq \liminf_{\eps \downarrow 0} R^{2-n}\mu_\eps(\mathcal{B}_R(x_0))(1 + C e^{C R^2 /2}) \\
		&\geq \liminf_{\eps \downarrow 0} \left\{ r^{2-n} \mu_\eps(\mathcal{B}_r(x_0)) + \frac{C e^{CR^2/2}}{R^{n-2+\alpha}} r^\alpha \mu_\eps(\mathcal{B}_R(x_0))\right\} \\
		&\geq r^{2-n} \mu_*(\mathcal{B}_r(x_0)) + \frac{C e^{CR^2/2}}{R^{n-2+\alpha}} r^\alpha \mu_*(\mathcal{B}_R(x_0)),
	\end{split}
	\]
	for any $0 < r < R < R_0$. Then, \eqref{eq:monotonicity-measures} follows by 
	approximating $R$ from below. Approximating $R_0$ with an increasing sequence 
	$\{R_k\}$, we see that \eqref{eq:monotonicity-measures} holds also 
	for $R = R_0$. This concludes the proof of the lemma.
\end{proof}

Lemma~\ref{lemma:out-of-support} below is the counterpart 
of~\cite[Proposition~VIII.1]{BethuelBrezisOrlandi} and it follows from 
Proposition~\ref{prop:small-ball-small-energy} exactly 
as~\cite[Proposition~VIII.1]{BethuelBrezisOrlandi} 
follows from \cite[Proposition~VII.1]{BethuelBrezisOrlandi}.
\begin{lemma}\label{lemma:out-of-support}
	There exists a constant $\eta_0 > 0$, depending on $\Lambda$ and
	$(M,g)$ only, such that, if 
	\[
		\mu_*(\mathcal{B}_R(x_0)) < \eta_0 R^{n-2},
	\] 
	then
	\[
		\mu_*(\mathcal{B}_{R/2}(x_0)) = 0,
	\]
	i.e., $\mathcal{B}_{R/2}(x_0) \subset M \setminus \spt \mu_*$.
\end{lemma}

\begin{lemma}\label{lemma:mu_*}
	There exist a constant $C_* := C_*(M,\,\Lambda) > 0$ and a number 
	$R_* \in (0, \inj(M))$, depending on $(M,g)$ only, such that 
	\begin{equation}\label{eq:control-rescaled-mu}
		\textrm{for any } x_0 \in \spt \mu_* 
		\textrm{ and } r \in (0,\,R_*], \quad C_*^{-1} \leq 
		r^{2-n} \mu_*(\mathcal{B}_r(x_0)) \leq C_*,
	\end{equation}
	whence
	\begin{equation}\label{eq:control-density}
		\textrm{for any } x_0 \in \spt \mu_*,
		\quad C_*^{-1} \leq \Theta_{n-2}(\mu,x_0) \leq C_*.
	\end{equation}
\end{lemma}

\begin{proof}
	The upper bound $r^{2-n} \mu_*(B_r(x_0)) \leq \widetilde{C}$, for 
	some $\widetilde{C} > 0$ depending on $(M,g)$ only, follows 
	from~\eqref{eq:monotonicity-measures} and the obvious inequality 
	$\mu_*(\mathcal{B}_r(x_0)) \leq \mu_*(M)$. 
	
	The lower bound is a straightforward 
	consequence of Lemma~\ref{lemma:out-of-support}. Indeed, were false, 
	for every choice of $c_* > 0$, we could find $x_0 \in \spt \mu_*$ and 
	$r_0 \in (0,R_*]$ such that 
	$\mu_*(\mathcal{B}_{r_0}(x_0)) < c_* r_0^{n-2}$. 
	By choosing $c_* < \eta_0$, where $\eta_0$ is the constant (depending 
	on 
	$(M,g)$ only) of Lemma~\ref{lemma:out-of-support}, 
	we would deduce 
	$x_0 \in M \setminus \spt \mu_*$, a contradiction. Then, 
	\eqref{eq:control-rescaled-mu} follows 
	by letting $C_* := \max\left\{ \widetilde{C}, \eta_0^{-1} \right\}$. 
	In turn, \eqref{eq:control-density} follows immediately 
	from~\eqref{eq:control-rescaled-mu} and~\eqref{eq:density}.
\end{proof}

We are now in a position to prove the following lemma, which is the counterpart 
of \cite[Lemma~6.3]{PigatiStern} and, as in \cite{PigatiStern}, is the key 
missing tool in the proof of Theorem~\ref{thm:rectifiability}.
\begin{lemma}\label{lemma:T_*}
	Let $\{T_\eps\}$ be the rescaled stress-energy tensor fields,
	defined by~\eqref{stresstensor}, of a sequence 
	$\{(u_\eps,\,A_\eps)\}\subset \mathcal{E}$ of critical points of $G_\eps$
	satisfying the logarithmic energy bound~\eqref{hp:logenergy}. Then, 
	there exists a bounded $\EndSym(M)$-valued Radon measure $T_*$ such that,  
	up to extraction of a subsequence, $T_\eps \rightharpoonup^* T_*$, 
	weakly$^*$ as $\EndSym(M)$-valued measures. Moreover, the following 
	statements hold.
	\begin{enumerate}[($i$)]
		\item $T_*$ is divergence-free, i.e., $\ip{T_*}{\nabla X} = 0$ for 
		all vector fields $X \in C^1(M,\,\T M)$ (here, $\nabla$ denotes the 
		Levi Civita connection of $M$).
		\item For every nonnegative $\varphi \in C^0(M)$, there holds 
		$\ip{T_*}{\varphi g} \geq (n-2) \ip{\mu_*}{\varphi}$
		\item For all $X \in C^0(M,\,\T M)$, there holds
		\[
			-\int_M \abs{X}^2 \,\d\mu_* \leq \ip{T_*}{X \otimes X} 
			\leq \int_M \abs{X}^2 \,\d\mu_*.
		\]
	\end{enumerate}	 
\end{lemma}

\begin{proof}
	The existence of $T_*$ is standard, as the uniform bound 
	$\frac{G_\eps(u_\eps,A_\eps)}{\abs{\log\eps}} \leq \Lambda$ immediately 
	yields a uniform bound on $\norm{T_\eps}_{(C^0(M))^*}$.
	Next, since each $T_\eps$ is divergence-free (because $(u_\eps,\,A_\eps)$ is a critical 
	point of $G_\eps$) and this property clearly passes to weak$^*$-limits in the sense of 
	measures, $T_*$ is divergence-free as well. This proves ($i$). 
	The proof of ($iii$) proceeds 
	exactly as in 
	\cite[Lemma~6.3]{PigatiStern}. 
	Concerning ($ii$), take any $\varphi \in C^0(M)$ such that $\varphi \geq 0$. Then,
	\[
	\begin{split}
		\int_M \ip{T_\eps}{\varphi g} \,\vol_g &= 
		\frac{1}{\abs{\log\eps}}\int_M \varphi(n e_\eps(u_\eps,A_\eps) - \abs{\D_{A_\eps}}^2 - 2 \abs{F_{A_\eps}}^2) \, \vol_g \\
		&= \frac{1}{\abs{\log\eps}}\int_M (n-2) \varphi \, e_\eps(u_\eps,A_\eps) \,\vol_g \\
		&\quad+ \frac{1}{\abs{\log\eps}} \int_M \varphi \left\{ \frac{1}{2\eps^2}\left(1-\abs{u_\eps}^2 \right)^2 - \abs{F_{A_\eps}}^2 \right\}\,\vol_g \\
		&\geq (n-2)\ip{\mu_\eps}{\varphi} - \frac{1}{\abs{\log\eps}}\int_M \abs{F_{A_\eps}}^2 \, \vol_g.
	\end{split}	
	\]
	By Lemma~\ref{lemma:curvature}, 
	$\frac{1}{\abs{\log\eps}} \int_M \abs{F_{A_\eps}}^2 \, \vol_g \to 0$ 
	as $\eps \to 0$, and (2) follows.
\end{proof}

\begin{proof}[Proof of Theorem~\ref{thm:rectifiability}]
	The existence of $\mu_*$ and of its $(n-2)$-density 
	$\Theta(\mu_*,x_0)$ 
	at every $x_0 \in M$ follow by Lemma~\ref{lemma:existence+density-mu_*}.
	Statement~($iii$) in Lemma~\ref{lemma:T_*}
	implies that the measure~$\abs{T_*}$ is absolutely continuous
	with respect to~$\mu_*$. 
	Hence, by the Radon-Nikodym theorem, we can represent $T_*$ by means of a 
	$L^1(\mu_*)$-section $P\colon M \to \EndSym(\T M)$; i.e., 
	for all $S \in C^0(M,\,\EndSym(\T M))$ we can write
	\[ 
		\ip{T_*}{S} = \int_M \ip{P(x)}{S(x)} \,\d \mu_*(x).
	\]
	Still by Lemma~\ref{lemma:T_*}, we infer that $-g \leq P(x) \leq g$ 
	(as bilinear forms) and ${\rm Tr}(P(x)) \geq n-2$ for $\mu_*$-a.e. $x \in M$. 
	This means that $T_*$ defines a generalised $(n-2)$-varifold 
	in $M$,
	with weight measure $\mu_*$. 
	Stationarity of such a varifold follows from ($i$) of Lemma~\ref{lemma:T_*} 
	and rectifiability follows from Theorem~\ref{thm:rect-gen-varifolds},
	which we can apply thanks to~\eqref{eq:control-density}. 
	In particular, $\spt \mu_*$ is $(n-2)$-rectifiable, $\mu_*$ coincides with $\theta \mathcal{H}^{n-2} \mres \spt\mu_*$, where 
	$\theta(x) := \Theta_{n-2}(\mu_*,x)$,
	and $P(x)$ is given, at $\mu_*$-a.e. $x \in M$, by the 
	orthogonal projection onto the $(n-2)$-subspace 
	$\T_x \spt \mu_* \subset \T_x M$. This proves the theorem.
\end{proof}

Finally, we record here a further 
property of the limit measure~$\mu_*$.
Lemma~\ref{lemma:curvature} implies that~$\mu_*$
can be characterised in terms of~$u_\eps$ only, 
\emph{so long as} we choose a convenient gauge.

\begin{prop} \label{prop:samelimit}
 Let~$\{(u_\eps, \, A_\eps)\}_{\eps>0}$ be a
 sequence of solutions of~\eqref{EL-u}--\eqref{EL-A}
 that satisfies~\eqref{hp:logenergy}. 
 Assume that~$A_\eps$ is in the form~\eqref{globalCoulomb}.
 Then,
 \[
  \mu_* = \lim_{\eps\to 0} \frac{1}{\abs{\log\eps}}
  \left( \frac{1}{2}\abs{\D_0 u_\eps}^2
   + \frac{1}{4\eps^2}\left(1 - \abs{u_\eps}^2\right)^2\right)
 \]
 (the limit being taken weakly$^*$ in the sense of measures).
\end{prop}
\begin{proof}
 By taking the differential in both sides of~\eqref{globalCoulomb},
 we obtain
 \[
  -\Delta \psi_\eps = \d\d^*\psi_\eps = \d A_\eps = F_\eps - F_0 
 \]
 where~$F_0$ is the curvature of the reference connection.
 Moreover, the form~$\psi_\eps$ is exact, hence orthogonal to
 all harmonic~$2$-forms. By elliptic regularity theory,
 it follows that the~$L^2(M)$-norm of~$-\Delta\psi_\eps$
 bounds the~$W^{2,2}(M)$-norm of~$\psi_\eps$ from above 
 (up to a constant factor). We deduce that 
 \[
  \norm{A_\eps}_{W^{1,2}(M)} 
  \leq \norm{\psi_\eps}_{W^{2,2}(M)} + \norm{\xi_\eps}_{W^{1,2}(M)} 
  \lesssim \norm{F_\eps}_{L^2(M)} + \norm{F_0}_{L^2(M)} + C_M
 \]
 (as~$\zeta_\eps$ is a harmonic form, the $W^{1,2}(M)$-norm of~$\zeta_\eps$
 is bounded by the~$L^\infty(M)$-norm, up to a constant factor
 that depends on~$M$ only).
 Then, Lemma~\ref{lemma:curvature} implies that the~$W^{1,2}(M)$-norm
 of~$A_\eps$ is of order lower than~$\abs{\log\eps}^{1/2}$,
 as~$\eps\to 0$. The proposition follows.
\end{proof}


\appendix

\section{Poincar\'e- and trace-type inequalities for differential forms}
\label{app:Poincare}

In the appendix, we collect a few technical results
on differential forms that are certainly part of the folklore.
However, since these results are crucial to our analysis,
we provide detailed proofs, for the convenience of the reader.
The following proposition provides a Poincar\'e-type inequality and a 
trace-type inequality for co-closed $k$-forms normal to the boundary of 
geodesic balls.  
A crucial point is that in both inequalities we make explicit the dependence
on the radius of the ball.
\begin{prop}\label{prop:gaffney-rho}
	Let $M$ be closed Riemannian manifold, $x_0 \in M$, 
	$\mathcal{B}_\rho(x_0)$ the geodesic ball of radius $\rho$ centered at 
	$x_0$, and let $\omega \in W^{1,p}(M,\,\T^*M)$, 
	with $p \in (1,\,+\infty)$ be a $1$-form such that
	\begin{equation}\label{eq:coulomb-omega}
	\begin{cases}
		\d^* \omega = 0 & \mbox{in } \mathcal{B}_\rho(x_0), \\
		{\rm i}_\nu \omega = 0 & \mbox{on } \partial \mathcal{B}_\rho(x_0).
	\end{cases}
	\end{equation}
	Then, 
	\begin{equation}\label{eq:gaffney-rho}
		\norm{\omega}_{L^p(\mathcal{B}_\rho(x_0))} \lesssim 
		\rho \norm{\d \omega}_{L^p(\mathcal{B}_\rho(x_0))},
	\end{equation}
	up to a constant depending only on $M$, $p$, and $g$. Moreover, 
	\begin{equation}\label{eq:trace}
		\norm{\omega}_{L^p(\partial \mathcal{B}_\rho(x_0))} \lesssim 
		\rho^{1-1/p}\norm{\d \omega}_{L^p(\mathcal{B}_\rho(x_0))},
	\end{equation}
	up to a constant depending only on $M$, $p$, and $g$.
\end{prop}

\begin{proof}
	To prove~\eqref{eq:gaffney-rho}, we notice that, if we were to work in a 
	Euclidean ball~$B_\rho(x_0)\subseteq\R^n$ with the Euclidean metric, 
	from the Gaffney-type inequality 
	\cite[Theorem~4.11]{IwaniecScottStroffolini} for any $1$-form 
	$\omega$ satisfying ${\rm i}_\nu \omega = 0$ we would get
	\begin{equation}\label{eq:gaffney-euc}
		\norm{\omega}_{L^p(B_\rho(x_0))} \leq 
		C_p(B_\rho(x_0)) 
		\left(\norm{\d \omega}_{L^p(B_\rho(x_0))}
		+ \norm{\d^*_0 \omega}_{L^p(B_\rho(x_0))} \right),
	\end{equation}
	where $\d^*_0$ denotes the codifferential operator with respect to the 
	Euclidean metric and $C_p(B_\rho(x_0))$ is a constant depending only on 
	$p$ and  
	$B_\rho(x_0)$ (in a way we are now going to make precise).
	By~\eqref{eq:gaffney-euc} with $\rho = 1$ and $x_0 = 0$ 
	and a classical scaling argument, we easily obtain
	\begin{equation}\label{eq:gaffney-euc-bis}
		\norm{\omega}_{L^p(B_\rho(x_0))} \lesssim 
		\rho\left(\norm{\d \omega}_{L^p(B_\rho(x_0))}
		+ \norm{\d^*_0 \omega}_{L^p(B_\rho(x_0))} \right),
	\end{equation}
	up to constant depending only on $n$, which in particular 
	establishes~\eqref{eq:gaffney-euc} in the Euclidean case. We remark 
	explicitly that the absolute values and the integrals involved in the 
	$L^p$-norms in~\eqref{eq:gaffney-euc-bis} are computed with respect to the 
	Euclidean metric and the Euclidean volume form.

	For the general case, we can work in normal coordinates centered at $x_0$. 
	Recall that the image of the geodesic ball 
	$\mathcal{B}_\rho(x_0)$ through the normal coordinates chart is the 
	Euclidean ball $B_\rho$ (centered at the origin). Thus, it is enough 
	to consider the case in which $\mathcal{B}_\rho(x_0) = B_\rho$ is a 
	Euclidean ball centered at the origin endowed with the metric $g$ of
	$M$.
	
	Assume $\omega \in W^{1,p}(M,\,\T^*M)$ is a $1$-form on $M$ 
	satisfying~\eqref{eq:coulomb-omega}.
	Since Gauss' Lemma implies that the unit normal field is the same both 
	when evaluated with respect to $g$ and with respect to the Euclidean 
	metric, \eqref{eq:gaffney-euc} holds for $\omega$. By hypothesis, we have 
	$\d^* \omega = 0$. Writing $\d^* \omega$ in coordinates (employing Einstein 
	notation), we see that
	\[
	\begin{split}
		\d^* \omega &= \frac{1}{\sqrt{\det(g^{ij})}} 
		\partial_i\left( \sqrt{\det g} g^{ij} \omega_j \right) \\
		&= g^{ij}\partial_i \omega_j + 
		\frac{\partial_i (\sqrt{\det g} g^{ij})}{\sqrt{\det g}} \omega_j \\
		&= \d^*_0 \omega + (g^{ij} - \delta^{ij}) \partial_i \omega_j 
		+ \frac{\partial_i(\sqrt{\det g} g^{ij})}{\sqrt{\det g}} \omega_j.
	\end{split}
	\]
	By the standard properties of normal coordinates, we have 
	\begin{equation}\label{eq:gaffney-comp1}
		(g^{ij} - \delta^{ij}) = \mathrm{O}(\rho^2) \qquad \mbox{and} \qquad 
		 \frac{\partial_i(\sqrt{\det g} g^{ij})}{\sqrt{\det g}} 
		 = \mathrm{O}(\rho)
	\end{equation}
	as $\rho \to 0$. Thus, the assumption $\d^* \omega = 0$ along 
	with~\eqref{eq:gaffney-comp1} leads to 
	\begin{equation}\label{eq:gaffney-comp2}
		\norm{\d^*_0 \omega}_{L^p(B_\rho)} \lesssim 
		\rho \left( \norm{\omega}_{L^p(B_\rho)} 
		+ \norm{\d \omega}_{L^p(B_\rho)} \right),
	\end{equation}
	for any $\rho$ sufficiently small and up to constant depending only on $n$, $p$ and $g$.
	In~\eqref{eq:gaffney-comp2}, the norms are again computed with respect to 
	the Euclidean metric and volume form. Choosing $\rho$ sufficiently small, 
	by~\eqref{eq:gaffney-euc} and~\eqref{eq:gaffney-comp2}, we obtain
	\begin{equation}\label{eq:gaffney-comp3}
		\norm{\omega}_{L^p(B_\rho)} \lesssim 
		\rho \norm{\d \omega}_{L^p(B_\rho)},
	\end{equation}
	for any $\rho$ sufficiently small and up to constant depending only on $n$, $p$ and $g$.
	Once more, the norms in~\eqref{eq:gaffney-comp3} are computed with 
	respect to 
	the Euclidean metric and volume form. However, since $M$ is compact, 
	$g$ is controlled from below and from above by the Euclidean metric. 
	Similarly, $\vol_g$ is controlled from 	below and from 	above by the 
	Euclidean volume form, and hence~\eqref{eq:gaffney-rho} follows from 
	\eqref{eq:gaffney-comp3} (up to replacing the implicit constant 
	in~\eqref{eq:gaffney-rho} with a larger one, which however depends still 
	only on $p$, $M$ and $g$. 
	
	The proof of~\eqref{eq:trace} follows a similar pattern. We begin again 
	with a Euclidean ball endowed with the Euclidean metric. Then the 
	inequality  
	\begin{equation}\label{eq:trace-comp1}
		\norm{\omega}_{L^p(\partial B_\rho)} \leq C_p(B_\rho) \left(  
		\norm{\omega}_{L^p(B_\rho)} + \norm{\d \omega}_{L^p(B_\rho)} + 
		\norm{\d^*_0 \omega}_{L^p(B_\rho)} \right)
	\end{equation}
	is proven by following the usual argument for the trace inequality for 
	Sobolev functions and using the Gaffney-type inequality 
	\cite[Theorem~4.8]{IwaniecScottStroffolini}. Then, by scaling, we derive 
	\begin{equation}\label{eq:trace-comp2}
		\norm{\omega}_{L^p(\partial B_\rho)} \lesssim \frac{1}{\rho^{1/p}}  
		\norm{\omega}_{L^p(B_\rho)} + 
		\rho^{1-1/p} \left( \norm{\d \omega}_{L^p(B_\rho)} + 
		\norm{\d^*_0 \omega}_{L^p(B_\rho)} \right),
	\end{equation}
	up to a constant depending only on $n$ and $p$. 
	In view of~\eqref{eq:coulomb-omega} and~\eqref{eq:gaffney-euc-bis}, this establishes~\eqref{eq:trace} 
	in the Euclidean case. For the general case, we obtain~\eqref{eq:trace} 
	from~\eqref{eq:trace-comp2} arguing as in the proof 
	of~\eqref{eq:gaffney-rho}. 
	Since the argument is very similar, we leave the details to the reader.
\end{proof}

\begin{lemma} \label{lemma:Poincare-R}
 There exists a number~$R_* > 0$
 such that, for any~$R\in (0, \, R_*)$,
 any~$x_0\in M$ and any~$k$-form~$\omega\in W^{1,2}(\mathcal{B}_R(x_0), \, \Lambda^k\T^*M)$ with~$\omega = 0$ on~$\partial \mathcal{B}_R(x_0)$,
 there holds
 \begin{equation}\label{eq:Poincare-R}
  \int_{\mathcal{B}_R(x_0)} \abs{\omega}^2 \vol_g
  \lesssim R^2 \int_{\mathcal{B}_R(x_0)} \left(\abs{\d\omega}^2
   + \abs{\d^*\omega}^2\right)\vol_g
 \end{equation}
 where the implicit constant in front of the right-hand side
 is independent of~$R$, $x_0$.
\end{lemma}
\begin{proof}
 Let~$\bar{\omega}$ be the $k$-form on~$M$
 defined by~$\bar{\omega}(x) := \omega(x)$
 if~$x\in \mathcal{B}_R(x_0)$ and~$\bar{\omega}(x) := 0$
 otherwise. As~$\omega = 0$ on~$\partial \mathcal{B}_R(x_0)$,
 we have~$\bar{\omega}\in W^{1,2}(M, \, \Lambda^k\T^*M)$.
 Let~$p := 2^* = 2n/(n-2)$. By applying Sobolev embeddings
 and the Gaffney inequality (in~$M$, \cite[Proposition~4.10]{Scott}), we have
 \[
  \norm{\bar{\omega}}_{L^p(M)}
  \lesssim \norm{\bar{\omega}}_{W^{1,2}(M)}
  \lesssim \norm{\d\bar{\omega}}_{L^2(M)} 
   + \norm{\d^*\bar{\omega}}_{L^2(M)}
   + \norm{\bar{\omega}}_{L^2(M)}
 \]
 where the implicit constants depend only on~$M$
 (not on~$R$ or~$x_0$). On the other hand,
 the H\"older inequality gives
 \[
  \begin{split}
   \norm{\omega}_{L^2(\mathcal{B}_R(x_0))}
   &\lesssim \left(\vol(\mathcal{B}_R(x_0))\right)^{1/2 - 1/p}
    \norm{\bar{\omega}}_{L^p(M)} \\
   &\lesssim R \norm{\d\bar{\omega}}_{L^2(M)} 
   + R \norm{\d^*\bar{\omega}}_{L^2(M)}
   + R \norm{\bar{\omega}}_{L^2(M)}
  \end{split}
 \]
 If~$R < R_*$ and~$R_*$ is chosen small enough,
 then the last term in the right-hand side
 can be absorbed in the left-hand side,
 and the lemma follows.
\end{proof}

\section{Elliptic estimates for the Hodge Laplacian}
\label{app:elliptic}

In this section, we gather a few regularity estimates
for elliptic problems involving differential forms.
These results, as the ones contained in Appendix~\ref{app:Poincare},
are classical, but we include detailed proofs for the reader's convenience.
We denote by~$\Delta$ the Hodge-Laplace operator
on differential forms, defined by~$-\Delta := \d\d^* + \d^*\d$,
where~$\d$ is the exterior differential and~$\d^*$ the codifferential.

\begin{prop} \label{prop:Green}
 Let~$p > n/2$, let $f\in L^p(M, \, \Lambda^k\T^*M)$
 and let~$\tau\in W^{1,2}(M, \, \Lambda^k\T^*M)$ be such that
 \begin{equation} \label{Green-equation}
  \begin{cases}
   -\Delta\tau = f - H(f)
    & \textrm{in the sense of distributions in } M\\[5pt]
   \displaystyle\int_{M}\ip{\tau}{\xi} \vol_g = 0
    & \textrm{for any harmonic~$k$-form } \xi.
  \end{cases}
 \end{equation}
 Then, $\tau$ is continuous and, for any~$x\in M$,
 there holds
 \begin{equation} \label{Green}
  \abs{\tau(x)} 
  \lesssim \int_M \frac{\abs{f(y)}}{\dist(x, \, y)^{n-2}} \vol_g(y)
 \end{equation}
 where~$\dist$ denotes the geodesic distance on~$M$.
\end{prop}

The proof of Proposition~\eqref{prop:Green}
is based on the existence of a Green form
for the Hodge-Laplace operator~$-\Delta$ on~$k$-forms.
The Green form is a differential form~$\mathfrak{g}$ on~$M\times M$
of degree~$(k, \, k)$ --- that is, at each point of~$M\times M$,
$\mathfrak{g}$ can locally be written as
\[
 \mathfrak{g}(x, \, y) = \sum_{I, J} \mathfrak{g}_{IJ}(x, \, y) \d x^I \wedge \d y^J,
\]
where~$(x^1, \, \ldots, \, x^n)$, $(y^1, \, \ldots, \, y^n)$
are local coordinate systems on~$M$, the~$\mathfrak{g}_{IJ}(x, \, y)$'s
are scalar coefficients, and the sum is taken over all
multi-indices~$I$, $J$ of order~$k$.
The Green form~$\mathfrak{g}$ is uniquely characterised by the 
following property: for any smooth $k$-form~$f$ on~$M$,
the unique solution~$\tau$ of~\eqref{Green-equation}
can be written as
\begin{equation} \label{Green-def}
 \tau(x) = \int_{M} \mathfrak{g}(x, \, y)\wedge\star f(y)
 \qquad \textrm{for any } x\in M.
\end{equation}
The existence of a Green form~$g$ was proved by de~Rham
(see~\cite[Chapter~III, Section~21]{deRhamKodaira}),
based on previous work by Bidal and de Rham
who constructed a parametrix (i.e., an ``approximate inverse'')
for the Hodge-Laplace operator. The Green form is smooth in
$\{(x, \, y)\in M\times M \colon x\neq y\}$
and satisfies 
\begin{equation} \label{Green-bound}
 \abs{\mathfrak{g}(x, \, y)} \lesssim \dist(x, \, y)^{n-2}
\end{equation}
for any~$x\in M$, $y\in M$ with~$x\neq y$
(see~\cite{deRhamKodaira}).

\begin{proof}[Proof of Proposition~\ref{prop:Green}]
 If~$f$ is a smooth~$k$-form, the estimate~\eqref{Green}
 is an immediate consequence of~\eqref{Green-def}
 and~\eqref{Green-bound}. The estimate~\eqref{Green}
 implies, via the H\"older inequality,
 an estimate for the~$L^\infty(M)$-norm of~$\tau$
 in terms of the~$L^p(M)$-norm of~$f$, for any~$p > n/2$.
 Indeed, if~$p > n/2$ and~$q := p^\prime = p/(p-1)$,
 then~$q < n/(n-2)$ and the function~$\dist(x, \, \cdot)^{n-2}$
 belongs to~$L^q(M)$ (in fact, the $L^q(M)$-norm
 of~$\dist(x, \, \cdot)^{n-2}$ is uniformly bounded
 with respect to~$x$). Then, the proposition
 follows by a density argument.
\end{proof}


\begin{lemma} \label{lemma:Caccioppoli}
 Let~$x_0\in M$, $0 < R < \inj(M)$
 and let~$\omega\in W^{1,2}(\mathcal{B}_R(x_0), \, \Lambda^k\T^*M)$ 
 be such that~$-\Delta\omega = 0$ in~$\mathcal{B}_R(x_0)$. Then,
 \begin{equation} \label{sketch-decay1}
  \norm{\omega}_{W^{1,2}(\mathcal{B}_{R/2})}
  \lesssim R^{-1} \norm{\omega}_{L^2(\mathcal{B}_R)}
 \end{equation}
 (where the implicit constant in front of the
 right-hand side is independent of~$x_0$, $R$).
\end{lemma}
\begin{proof}
 Let~$\zeta\in C^\infty_{\mathrm{c}}(\mathcal{B}_R)$ be a cut-off
 function such that~$\zeta = 1$ in~$\mathcal{B}_{R/2}$
 and~$\abs{\d\zeta} \lesssim R^{-1}$ in~$\mathcal{B}_R$.
 There holds
 \begin{equation} \label{Cacc1}
  \begin{split}
   \ip{\d\omega}{\d(\zeta^2\omega)}
   = \ip{\d\omega}{\d\zeta\wedge\zeta\omega + \zeta\d(\zeta\omega)}
   &= \ip{\zeta\d\omega}{\d\zeta\wedge\omega + \d(\zeta\omega)}\\
   &= \ip{\d(\zeta\omega) - \d\zeta\wedge\omega}
    {\d\zeta\wedge\omega + \d(\zeta\omega)}\\
   &= \abs{\d(\zeta\omega)}^2 - \abs{\d\zeta\wedge\omega}^2
  \end{split}
 \end{equation}
 and similarly,
 \begin{equation} \label{Cacc2}
  \begin{split}
   \ip{\d^*\omega}{\d^*(\zeta^2\omega)}
   = \ip{\d(\star\omega)}{\d(\zeta^2\star\omega)}
   &= \abs{\d(\zeta\star\omega)}^2 - \abs{\d\zeta\wedge\star\omega}^2\\
   &= \abs{\d^*(\zeta\omega)}^2 - \abs{\d\zeta\wedge\star\omega}^2
  \end{split}
 \end{equation}
 Therefore, by testing the equation~$-\Delta\omega = 0$
 against~$\zeta^2\omega$, we obtain
 \[
  \begin{split}
   \int_{\mathcal{B}_R(x_0)} \left(\abs{\d(\zeta\omega)}^2 
    + \abs{\d^*(\zeta\omega)}^2\right) \vol_g
   &= \int_{\mathcal{B}_R(x_0)} \left(\abs{\d\zeta\wedge\omega}^2 
    + \abs{\d\zeta\wedge\star\omega}^2\right) \vol_g \\
   &\lesssim R^{-2} \int_{\mathcal{B}_R(x_0)} \abs{\omega}^2 \vol_g
  \end{split}  
 \]
 By applying the Gaffney inequality (in~$M$, \cite[Proposition~4.10]{Scott}),
 we obtain
 \begin{equation*}
  \norm{\zeta\omega}_{W^{1,2}(M)}
  \lesssim R^{-1} \norm{\omega}_{L^2(\mathcal{B}_R)} + \norm{\zeta\omega}_{L^2(M)}
  \lesssim R^{-1} \norm{\omega}_{L^2(\mathcal{B}_R)}
 \end{equation*}
 and the lemma follows.
\end{proof}

Thus, harmonic forms in geodesic balls satisfy a Caccioppoli-type inequality, 
exactly as in the Euclidean case. Moreover, exactly as in the case of harmonic 
functions on balls of $\R^n$  (see, e.g., \cite[Lemma~4.11]{Beck}), 
Caccioppoli inequality allows to prove a decay property of 
$L^2$-norm of harmonic forms in geodesic balls.
\begin{theorem}\label{thm:decay-harmonic}
 Let~$x_0\in M$, $0 < R < \inj(M)$
 and let~$\omega\in W^{1,2}(\mathcal{B}_R(x_0), \, \Lambda^k\T^*M)$ 
 be such that~$-\Delta\omega = 0$ in~$\mathcal{B}_R(x_0)$. Then,
 \begin{equation}\label{eq:decay-harmonic}
 	\int_{\mathcal{B}_r(x_0)} \abs{\omega}^2 \, \vol_g
 	\lesssim \frac{r^n}{R^n} \int_{\mathcal{B}_R(x_0)} \abs{\omega}^2 \, \vol_g,
 \end{equation}
 where the implicit constant in front of the
 right-hand side is independent of~$x_0$, $R$.
\end{theorem}

\begin{proof}
Let $\zeta\in C^\infty_{\mathrm{c}}(\mathcal{B}_{R/2})$ be a cut-off function,
such that~$\zeta = 1$ in~$\mathcal{B}_{R/4}$
and~$\abs{\nabla\zeta} \lesssim R^{-1}$, $\abs{\nabla^2\zeta}\lesssim R^{-2}$.
Let~$\tau$ be the unique solution of
\[
 -\Delta\tau = -\Delta(\zeta\omega) \quad \textrm{in } M,
 \qquad H(\tau) = 0.
\]
($\tau$~exists, because~$\Delta(\zeta\omega)$ is orthogonal 
to all harmonic forms.) Then, $\zeta\omega = \tau + H(\zeta\omega)$.
Let~$x\in \mathcal{B}_{R/8}$. By Proposition~\ref{prop:Green},
we have
\begin{equation} \label{sketch-decay2}
 \abs{\tau(x)} 
 \lesssim \int_M \frac{\abs{\Delta(\zeta\omega)(y)}}
  {\dist(x, \, y)^{n-2}} \vol_g(y)
 \lesssim R^{2 - n} \norm{\Delta(\zeta\omega)}_{L^1(\mathcal{B}_{R/2}\setminus \mathcal{B}_{R/4})}
\end{equation}
because~$\Delta(\zeta\omega) = \Delta\omega = 0$ in~$\mathcal{B}_{R/4}$
and~$\zeta\omega = 0$ out of~$\mathcal{B}_{R/2}$
(and~$\dist(x, \, y) \geq R/8$ for any~$x\in \mathcal{B}_{R/8}$, 
$y\in \mathcal{B}_{R/2}\setminus \mathcal{B}_{R/4}$).
By the H\"older inequality,
\begin{equation} \label{sketch-decay3}
 \abs{\tau(x)}
 \lesssim R^{2 - n/2} \norm{\Delta(\zeta\omega)}_{L^2(\mathcal{B}_{R/2}\setminus \mathcal{B}_{R/4})}
\end{equation}
Now, 
we have
\[
 \Delta(\zeta\omega) = \d^*(\d\zeta\wedge\omega + \zeta\d\omega)
  + \d\left(\pm \star(\d\zeta\wedge\star\omega) + \zeta\d^*\omega)\right)
 = \ldots = \zeta\Delta\omega + \mathrm{l.o.t.},
\]
where~l.o.t. is a sum of terms that do not 
contain second derivatives of~$\omega$
(but do contain derivatives of~$\zeta$). In particular, 
as~$\Delta\omega = 0$,
\begin{equation} \label{sketch-decay4}
 \norm{\Delta(\zeta\omega)}_{L^2(\mathcal{B}_{R/2})}
 \lesssim R^{-1} \norm{\omega}_{W^{1,2}(\mathcal{B}_{R/2})}
  + R^{-2} \norm{\omega}_{L^2(\mathcal{B}_{R/2})}
 \stackrel{\eqref{sketch-decay1}}{\lesssim}
  R^{-2} \norm{\omega}_{L^2(\mathcal{B}_R)}
\end{equation}
(the factors of~$R^{-1}$, $R^{-2}$ come from the derivatives of~$\zeta$).
Combining~\eqref{sketch-decay3} with~\eqref{sketch-decay4},
we deduce
\begin{equation} \label{sketch-decay5}
 \norm{\tau}_{L^\infty(\mathcal{B}_{R/8})}
 \lesssim R^{- n/2} \norm{\omega}_{L^2(\mathcal{B}_R)}
\end{equation}
The estimate of the $L^\infty$-norm
of~$H(\zeta\omega)$ is immediate
because~$H(\zeta\omega)$ is harmonic ---
and all norms are equivalent on~$\Harm^k(M)$.
In particular,
\[
 \norm{H(\zeta\omega)}_{L^\infty(M)}
  \lesssim \norm{H(\zeta\omega)}_{L^2(M)} 
  \lesssim \norm{\omega}_{L^2(M)}
  \lesssim R^{-n/2} \norm{\omega}_{L^2(M)}
\]
(because~$R$ is bounded from above).
Overall, we have proved
\begin{equation} \label{sketch-decay6}
 \norm{\omega}_{L^\infty(\mathcal{B}_{R/8})}
 \lesssim R^{- n/2} \norm{\omega}_{L^2(\mathcal{B}_R)}
\end{equation}
which immediately implies decay (cf., for instance, 
\cite[Lemma~4.11]{Beck}.)
\end{proof}

In Proposition~\ref{prop:decay-RHS} below, we extend the decay estimate in 
Theorem~\ref{thm:decay-harmonic} to forms solving Poisson-type equations with 
$L^2$-integrable sources.
\begin{prop} \label{prop:decay-RHS}
 There exists a number~$R_*> 0$, depending only on~$M$,
 such that the following statement holds.
 If~$x_0\in M$, $0 < R < R_*$ and~$\omega$, $f$ 
 are $k$-form on~$\mathcal{B}_R(x_0)$ such that
 $-\Delta\omega = f$ in~$\mathcal{B}_R(x_0)$, then
 \begin{equation}\label{eq:decay-RHS}
  \int_{\mathcal{B}_r(x_0)} \left(\abs{\d\omega}^2 + \abs{\d^*\omega}^2\right)\vol_g 
  \lesssim \frac{r^n}{R^n} \int_{\mathcal{B}_R(x_0)}
   \left(\abs{\d\omega}^2 + \abs{\d^*\omega}^2\right)\vol_g 
   + R^2 \int_{\mathcal{B}_R(x_0)} \abs{f}^2 \vol_g
 \end{equation}
 for every~$r\in (0, \, R)$.
\end{prop}

\begin{proof}
	The proof follows a very well-known pattern, typical for analogous 
	estimates for elliptic systems in the Euclidean setting (see, e.g., 
	\cite[Lemma~4.13]{Beck}). 
	We provide details for completeness, as for differential forms on manifolds 
	they are apparently lacking in the literature (to the best to our 
	knowledge, at least).
	
	\medskip
	\noindent
	\emph{Step~1.} Fix arbitrarily $x_0 \in M$. For the time being, let 
	$R_* = \inj(M)$, and pick any $R \in (0,\,R_*)$. During the whole proof, 
	we omit to indicate the center of the balls, as it will be always 
	$x_0$.

	As a preliminary step, we write 
	$\omega = \omega_1 + \omega_2$, where $\omega_1$ satisfies 
	\begin{equation}\label{eq:decay-RHS-compu1}
	\begin{cases}
		\Delta \omega_1 = 0 & \mbox{in } \mathcal{B}_R,\\
		\omega_1 = \omega & \mbox{on } \partial \mathcal{B}_R,
	\end{cases} 
	\end{equation}
	while $\omega_2$ is the solution to
	\begin{equation}\label{eq:decay-RHS-compu2}
	\begin{cases}
		-\Delta \omega_2 = f & \mbox{in } \mathcal{B}_R,\\
		\omega_2 = 0 & \mbox{on } \partial \mathcal{B}_R.
	\end{cases} 
	\end{equation}
	In the next steps, we first obtain suitable estimates for $\omega_1$ and 
	$\omega_2$, and then we combine them to yield~\eqref{eq:decay-RHS}. 
	
	\medskip
	\noindent
	\emph{Step~2: Estimate for $\omega_1$.} Since $\Delta \omega_1 = 0$ in 
	$\mathcal{B}_R$, we have $\Delta(\d \omega_1) = 0$ and 
	$\Delta(\d^* \omega_1) = 0$ in $\mathcal{B}_R$ as well. Thus, the decay 
	estimate~\eqref{eq:decay-harmonic} holds for both $\d \omega_1$ and 
	$\d^* \omega_1$ and yields
	\begin{equation}\label{eq:decay-RHS-compu3}
		\int_{\mathcal{B}_r} \left( \abs{\d\omega_1}^2 + \abs{\d^*\omega_1}^2 \right)\, \vol_g \lesssim 
		\frac{r^n}{R^n} \int_{\mathcal{B}_R} \left( \abs{\d\omega_1}^2 + \abs{\d^*\omega_1}^2 \right) \,\vol_g 
	\end{equation}
	for any $r \in (0,\,R)$, where the implicit constant at right-hand-side
	depends only on $M$ and $g$. On the other hand, $\omega_1$ minimises the 
	functional
	\begin{equation}\label{eq:dirichlet-energy}
		\omega \mapsto \int_{\mathcal{B}_R} \left( \abs{\d \omega}^2 + \abs{\d^*\omega}^2 \right) \,\vol_g
	\end{equation}
	among all $k$-forms $\omega$ such that $\omega = \omega_1$ on 
	$\partial\mathcal{B}_R$. This fact, along 
	with~\eqref{eq:decay-RHS-compu3}, yields
	\begin{equation}\label{eq:decay-RHS-compu4}
		\int_{\mathcal{B}_r} \left( \abs{\d\omega_1}^2 + \abs{\d^* \omega_1}^2 \right)\,\vol_g 
		\lesssim \frac{r^n}{R^n} \int_{\mathcal{B}_R} \left( \abs{\d\omega_1}^2 + \abs{\d^* \omega_1}^2 \right)\,\vol_g  .
	\end{equation}
	
	\medskip
	\noindent
	\emph{Step~3: Estimate for $\omega_2$.} The estimate for $\omega_2$ follows 
	by testing~\eqref{eq:decay-RHS-compu2} against $\omega_2$. We obtain 
	\begin{equation}\label{eq:decay-RHS-compu5}
		\int_{\mathcal{B}_R} \left( \abs{\d \omega_2}^2 + \abs{\d^* \omega_2}^2 \right)\,\vol_g 
		\leq \norm{f}_{L^2(\mathcal{B}_R)} \norm{\omega_2}_{L^2(\mathcal{B}_R)}.
	\end{equation}
	Notice that there are no boundary terms in~\eqref{eq:decay-RHS-compu5} 
	because $\omega_2 = 0$ on $\partial \mathcal{B}_R$. By 
	Lemma~\ref{lemma:Poincare-R}, up to shrink $R_*$ if necessary (but still 
	according only to $M$ and $g$), we get
	\begin{equation}\label{eq:decay-RHS-compu6}
		\norm{\omega_2}_{L^2(\mathcal{B}_R)} \lesssim 
		R \left( \norm{\d \omega_2}_{L^2(\mathcal{B}_R)} + \norm{\d^* \omega_2}_{L^2(\mathcal{B}_R)}\right),
	\end{equation}
	and therefore
	\begin{equation}\label{eq:decay-RHS-compu7}
		\norm{\d \omega_2}_{L^2(\mathcal{B}_R)} + \norm{\d^* \omega_2}_{L^2(\mathcal{B}_R)} \lesssim R \norm{f}_{L^2(\mathcal{B}_R)},
	\end{equation}
	up to a constant depending only on $M$ and $g$.
	
	\medskip
	\noindent
	\emph{Step~4: Conclusion.} We can now estimate
	\begin{multline*}
		\int_{\mathcal{B}_r} \left( \abs{\d \omega}^2 + \abs{\d^* \omega}^2 \right)\,\vol_g \leq \int_{\mathcal{B}_r} \left( \abs{\d \omega_1}^2 + \abs{\d^* \omega_1}^2 +\abs{\d \omega_2}^2 + \abs{\d^* \omega_2}^2 \right)\,\vol_g \\
	\stackrel{{\eqref{eq:decay-RHS-compu4}, \eqref{eq:decay-RHS-compu7}}}{ \lesssim} \frac{r^n}{R^n} \int_{\mathcal{B}_R} \left( \abs{\d \omega_1}^2 + \abs{\d^* \omega_1}^2 \right)\,\vol_g + R^2 \int_{\mathcal{B}_R} \abs{f}^2 \,\vol_g \\
	 \lesssim  \frac{r^n}{R^n} \int_{\mathcal{B}_R} \left( \abs{\d \omega}^2 + \abs{\d^* \omega}^2 \right)\,\vol_g + R^2 \int_{\mathcal{B}_R} \abs{f}^2 \,\vol_g,
	\end{multline*}
	where the last line (which provides the desired inequality) follows because 
	$\omega_1$ minimises the functional~\eqref{eq:dirichlet-energy} among 
	$k$-forms with the same value on the boundary of $\mathcal{B}_R$. 
	Then, the conclusion follows by the arbitrariness of $x_0$ and the 
	compactness of $M$, which allows to choose a uniform critical radius $R_*$. 
\end{proof}

We conclude this section by giving the proof
of a few Caccioppoli-type inequalities, which appeared in
the proof of Proposition~\ref{prop:small-ball-small-energy}.

\begin{lemma} \label{lemma:CaccioppoliL1}
 There exists a number~$R_*> 0$, depending only on~$M$,
 such that the following statement holds.
 Let~$x_0\in M$, $0 < r <  R < R_*$, $a\in L^\infty(\mathcal{B}_R(x_0))$
 and~$a_0 > 0$ be such that $a_0^{-1} \leq a(x) \leq a_0$
 for a.e.~$x\in\mathcal{B}_R(x_0)$. Let~$f\in L^2(\mathcal{B}_R(x_0), \, \T^*M)$
 and let~$\theta\in W^{1,2}(\mathcal{B}_R(x_0))$ be a weak solution
 of the equation
 \begin{equation} \label{elliptic1}
  \d^*\left(a \, \d\theta\right) = \d^* f
  \qquad \textrm{in } \mathcal{B}_R(x_0).
 \end{equation}
 Then,
 \begin{equation} \label{elliptic2}
  \norm{\d\theta}_{L^2(\mathcal{B}_{r}(x_0))}
   \lesssim (R-r)^{-\frac{n+2}{2}} \norm{\theta}_{L^1(\mathcal{B}_R(x_0))}
    + \norm{f}_{L^2(\mathcal{B}_R(x_0))}
 \end{equation}
\end{lemma}
 The estimate~\eqref{elliptic2} is a Caccioppoli inequality,
 except that the right-hand side contains the~$L^1$-norm of~$\theta$
 instead of the usual~$L^2$ norm. Therefore, we include a proof of 
 Lemma~\ref{lemma:CaccioppoliL1}, for the reader's convenience.
\begin{proof}[Proof of Lemma~\ref{lemma:CaccioppoliL1}]
 In the proof, we will apply the following inequality:
 for any~$x_0\in M$, $0 < \sigma < \inj(M)$
 and~$\theta\in W^{1,2}(\mathcal{B}_\sigma(x_0))$,
 there holds
 \begin{equation} \label{elliptic1,0}
  \norm{\theta}_{L^2(\mathcal{B}_\sigma(x_0))}
  \lesssim \norm{\d\theta}_{L^2(\mathcal{B}_\sigma(x_0))}^{\frac{n}{n+2}}
   \norm{\theta}_{L^1(\mathcal{B}_\sigma(x_0))}^{\frac{2}{n+2}}
   + \sigma^{-\frac{n}{2}}\norm{\theta}_{L^1(\mathcal{B}_\sigma(x_0))}
 \end{equation}
 where the implicit constant in front of the right-hand side
 does not depend on~$\sigma$, $x_0$.
 In case~$\mathcal{B}_\sigma(x_0)$ is a ball in~$\R^n$,
 equipped with the Euclidean metric, the
 estimate~\eqref{elliptic1,0} 
 is a special case of Gagliardo and Nirenberg's
 interpolation inequality, also known as Nash's inequality 
 (see~\cite{Nash1958}).
 A scaling argument shows that the implicit
 constant in front of the right-hand side is independent of~$\sigma$.
 When~$\mathcal{B}_\sigma(x_0)$ is a geodesic ball on~$M$,
 of radius~$\sigma$ smaller than the injectivity radius of~$M$,
 \eqref{elliptic1,0} follows from
 its Euclidean counterpart, up to composition with
 local (e.g., geodesic) coordinates.
 
 Let~$\rho$, $\sigma$ be positive numbers 
 such that~$r < \rho < \sigma < R$.
 We consider the Caccioppoli inequality
 \begin{equation} \label{elliptic1,1}
  \begin{split}
   \int_{B_\rho} \abs{\d\theta}^2\vol_g
   \lesssim (\sigma - \rho)^{-2} \int_{\mathcal{B}_\sigma} \theta^2 \vol_g
   + \int_{\mathcal{B}_\sigma} \abs{f}^2 \vol_g
  \end{split}
 \end{equation}
 For simplicity of notation, we have dropped the dependence 
 on~$x_0$ in all balls.
 The estimate~\eqref{elliptic1,1} is obtained
 by considering a suitable cut-off
 function~$\zeta\in C^\infty_{\mathrm{c}}(\mathcal{B}_\sigma)$
 and testing Equation~\eqref{elliptic1} against~$\zeta^2\theta$.
 The estimate~\eqref{elliptic1,1}, combined
 with~\eqref{elliptic1,0} and the Young inequality, implies
 \begin{equation} \label{elliptic1,2}
  \begin{split}
   \norm{\d\theta}_{L^2(\mathcal{B}_\rho)} 
   &\lesssim (\sigma - \rho)^{-1} 
    \norm{\d\theta}_{L^2(\mathcal{B}_\sigma)}^{\frac{n}{n+2}}
    \norm{\theta}_{L^1(\mathcal{B}_\sigma)}^{\frac{2}{n+2}}
     + (\sigma - \rho)^{-\frac{n+2}{2}} \norm{\theta}_{L^1(\mathcal{B}_\sigma)}
    + \norm{f}_{L^2(\mathcal{B}_R)} \\
   &\lesssim \delta \norm{\d\theta}_{L^2(\mathcal{B}_\sigma)}
    + C_\delta (\sigma - \rho)^{-\frac{n+2}{2}} 
    \norm{\theta}_{L^1(\mathcal{B}_R)}
    + \norm{f}_{L^2(\mathcal{B}_R)} 
  \end{split}
 \end{equation}
 for an arbitrary~$\delta\in (0, \, 1)$ and a constant~$C_\delta$
 that depends only on~$\delta$, $a_0$ and~$(M, \, g)$.
 As the inequality~\eqref{elliptic1,2} is valid for 
 arbitrary values of~$\rho$, $\sigma$ with~$r < \rho < \sigma < R$,
 it can be applied iteratively. The lemma follows by 
 an iteration argument (see~\cite[Lemma~V.3.1]{Giaquinta-MultipleIntegrals} 
 or~\cite[Lemma~B.1]{Beck}).
\end{proof}

There is another Caccioppoli-type inequality we need,
which follows along the same lines of Lemma~\ref{lemma:CaccioppoliL1}
but applies to~$k$-forms instead of scalar functions.
As a preliminary result, we first provide a version
of the Nash inequality~\eqref{elliptic1,0} for differential forms.

\begin{lemma} \label{lemma:Nash}
 For any~$x_0\in M$, any~$0 < r < R < \inj(M)/2$
 and any~$\omega\in W^{1,2}(\mathcal{B}_R(x_0), \, \Lambda^k\T^*M)$,
 there holds
 \begin{equation*}
  \begin{split}
   \norm{\omega}_{L^2(\mathcal{B}_r(x_0))}
   &\lesssim \norm{\d\omega}_{L^2(\mathcal{B}_R(x_0))}^{\frac{n}{n+2}}
    \norm{\omega}_{L^1(\mathcal{B}_R(x_0))}^{\frac{2}{n+2}}\\
   &\qquad + \norm{\d^*\omega}_{L^2(\mathcal{B}_R(x_0))}^{\frac{n}{n+2}}
    \norm{\omega}_{L^1(\mathcal{B}_R(x_0))}^{\frac{2}{n+2}}
    + (R - r)^{-\frac{n}{2}}\norm{\omega}_{L^1(\mathcal{B}_R(x_0))} 
  \end{split}
 \end{equation*}
 where the implicit constant in front of the right-hand side
 does not depend on~$r$, $R$, $x_0$.
\end{lemma}
\begin{proof}
 All the balls we consider are centered at~$x_0$,
 so we write~$\mathcal{B}_r$, $\mathcal{B}_R$
 instead of~$\mathcal{B}_r(x_0)$, $\mathcal{B}_R(x_0)$ and so on.
 Without loss of generality, we can assume that~$r \geq R/2$.
 Let~$\rho$, $\sigma$ be positive numbers,
 such that~$r < \rho < \sigma < R$.
 Let~$\omega\in W^{1,2}(\mathcal{B}_R, \, \Lambda^k\T^*M)$.
 Using geodesic normal coordinates~$(x^1, \, \ldots, \, x^n)$
 centered at~$x_0$, we can write~$\omega$ component-wise 
 as~$\omega = \sum_\alpha\omega_\alpha \, \d x^\alpha$,
 where the sum is taken over all multi-indices~$\alpha$
 of order~$k$.
 By applying the inequality~\eqref{elliptic1,0} 
 (on the ball~$\mathcal{B}_\rho$) to each component
 of~$\omega$, we obtain
 \begin{equation} \label{Nash1}
  \norm{\omega}_{L^2(\mathcal{B}_\rho)}
  \lesssim \norm{\omega}_{W^{1,2}(\mathcal{B}_\rho)}^{\frac{n}{n+2}}
   \norm{\omega}_{L^1(\mathcal{B}_\rho)}^{\frac{2}{n+2}}
   + \rho^{-\frac{n}{2}}\norm{\omega}_{L^1(\mathcal{B}_\rho)}
 \end{equation}
 The implicit constant in front of the right-hand side
 depend on the metric of~$g$, but it can be estimated uniformly
 with respect to~$x_0$ and~$\rho$. Now, let~$\zeta\in C^\infty_{\mathrm{c}}(\mathcal{B}_\sigma)$ be a cut-off function, such that~$\zeta = 1$
 in~$\mathcal{B}_\rho$ and~$\abs{\nabla\zeta} \lesssim (\sigma - \rho)^{-1}$.
 By applying the Gaffney inequality \cite[Proposition~4.10]{Scott}
 to~$\zeta\omega\in W^{1,2}(M, \, \Lambda^k\T^*M)$, we obtain
 \begin{equation} \label{Nash2}
  \begin{split}
   \norm{\omega}_{W^{1,2}(\mathcal{B}_\rho)}
   &\lesssim \norm{\d(\zeta\omega)}_{L^2(M)}
    + \norm{\d^*(\zeta\omega)}_{L^2(M)} + \norm{\zeta\omega}_{L^2(M)} \\
   &\lesssim \norm{\d\omega}_{L^2(\mathcal{B}_\sigma)}
    + \norm{\d^*\omega}_{L^2(\mathcal{B}_\sigma)} 
    + (\sigma - \rho)^{-1}\norm{\omega}_{L^2(\mathcal{B}_\sigma)}  
  \end{split}
 \end{equation}
 By combining~\eqref{Nash1} with~\eqref{Nash2}
 (and observing that~$\rho \geq \sigma - \rho$ because we 
 have assumed that~$r\geq R/2$), we obtain
 \begin{equation*}
  \begin{split}
   \norm{\omega}_{L^2(\mathcal{B}_\rho)}
   &\lesssim 
    \norm{\d\omega}_{L^2(\mathcal{B}_R)}^{\frac{n}{n+2}} 
    \norm{\omega}_{L^1(\mathcal{B}_R)}^{\frac{2}{n+2}}
    + \norm{\d^*\omega}_{L^2(\mathcal{B}_R)}^{\frac{n}{n+2}} 
    \norm{\omega}_{L^1(\mathcal{B}_R)}^{\frac{2}{n+2}} \\
   &\qquad + (\sigma - \rho)^{-\frac{n}{n+2}}
    \norm{\omega}_{L^2(\mathcal{B}_\sigma)}^{\frac{n}{n+2}}  
    \norm{\omega}_{L^1(\mathcal{B}_R)}^{\frac{2}{n+2}}
    + (\sigma - \rho)^{-\frac{n}{2}} \norm{\omega}_{L^1(\mathcal{B}_R)}
  \end{split}
 \end{equation*}
 By applying the Young inequality at the right-hand side,
 for each~$\delta > 0$ we find a constant~$C_\delta$ such that
 \begin{equation} \label{Nash3}
  \begin{split}
   \norm{\omega}_{L^2(\mathcal{B}_\rho)}
   &\lesssim 
    \norm{\d\omega}_{L^2(\mathcal{B}_R)}^{\frac{n}{n+2}} 
    \norm{\omega}_{L^1(\mathcal{B}_R)}^{\frac{2}{n+2}}
    + \norm{\d^*\omega}_{L^2(\mathcal{B}_R)}^{\frac{n}{n+2}} 
    \norm{\omega}_{L^1(\mathcal{B}_R)}^{\frac{2}{n+2}} \\
   &\qquad + C_\delta (\sigma - \rho)^{-\frac{n}{2}}
    \norm{\omega}_{L^1(\mathcal{B}_R)} 
    + \delta \norm{\omega}_{L^2(\mathcal{B}_\sigma)}
  \end{split}
 \end{equation}
 If we choose~$\delta$ small enough (uniformly with respect to~$r$,
 $R$, $x_0$), the term~$\delta \norm{\omega}_{L^2(\mathcal{B}_\sigma)}$
 at the right-hand side can be `absorbed into the left-hand side',
 in a manner of speaking, by means of an iteration argument
 (see~\cite[Lemma~V.3.1]{Giaquinta-MultipleIntegrals} 
 or~\cite[Lemma~B.1]{Beck}). The lemma follows.
\end{proof}

\begin{lemma} \label{lemma:CaccioppoliL1bis}
 There exists a number~$R_*> 0$, depending only on~$M$,
 such that the following statement holds.
 Let~$x_0\in M$, $0 < r <  R < R_*$ 
 and~$c\in L^\infty(\mathcal{B}_R(x_0))$
 be such that~$c\geq 0$.
 Let~$f\in L^2(\mathcal{B}_R(x_0), \, \Lambda^k\T^*M)$
 and let~$\omega\in W^{1,2}(\mathcal{B}_R(x_0))$ be a weak solution
 of the equation
 \begin{equation} \label{elliptic3}
  -\Delta\omega + c \, \omega = f
  \qquad \textrm{in } \mathcal{B}_R(x_0).
 \end{equation}
 Then,
 \begin{equation*} 
  \norm{\omega}_{W^{1,2}(\mathcal{B}_{r}(x_0))}
   \lesssim (R - r)^{-\frac{n+2}{2}} 
   \norm{\omega}_{L^1(\mathcal{B}_R(x_0))}
    + R \norm{f}_{L^2(\mathcal{B}_R(x_0))}
 \end{equation*}
\end{lemma}
\begin{proof}
 Once again, we write~$\mathcal{B}_\rho$
 for a generic (geodesic) ball of radius~$\rho$ and center~$x_0$.
 Let~$r^\prime := (r+R)/2$,
 let~$\rho$, $\sigma$ be positive numbers 
 such that~$r^\prime < \rho < \sigma < R$, 
 and let~$\tau := (\rho + \sigma)/2$.
 Let~$\zeta\in C^\infty_{\mathrm{c}}(\mathcal{B}_{\tau})$ 
 be a cut-off function such that~$\zeta=1$ 
 in~$\mathcal{B}_\rho$ and
 $\abs{\nabla\zeta}\lesssim (\tau-\rho)^{-1}$.
 By testing~\eqref{elliptic3} against~$\zeta^2 \, \omega$, 
 and keeping~\eqref{Cacc1}, \eqref{Cacc2} into account, we deduce
 \begin{equation*} 
  \int_{\mathcal{B}_\rho} \left(\abs{\d\omega}^2
   + \abs{\d^*\omega}^2 \right) \vol_g
  \lesssim (\tau - \rho)^{-2} 
   \int_{\mathcal{B}_\tau} \abs{\omega}^2 \, \vol_g
   + \tau^2 \int_{\mathcal{B}_\tau} \abs{f}^2 \, \vol_g
 \end{equation*}
 By applying Lemma~\ref{lemma:Nash} and the Young inequality,
 we obtain
 \begin{equation*}
  \begin{split}
   \norm{\d\omega}_{L^2(\mathcal{B}_\rho)}
   + \norm{\d\omega}_{L^2(\mathcal{B}_\rho)}
   &\lesssim (\sigma - \rho)^{-1} 
    \norm{\d\omega}_{L^2(\mathcal{B}_\sigma)}^{\frac{n}{n+2}}
    \norm{\omega}_{L^1(\mathcal{B}_\sigma)}^{\frac{2}{n+2}} \\
   &\qquad + (\sigma - \rho)^{-1} 
    \norm{\d^*\omega}_{L^2(\mathcal{B}_\sigma)}^{\frac{n}{n+2}}
    \norm{\omega}_{L^1(\mathcal{B}_\sigma)}^{\frac{2}{n+2}} \\
   &\qquad + (\sigma - \rho)^{-\frac{n+2}{2}} 
   \norm{\omega}_{L^1(\mathcal{B}_\sigma)}
    + R \norm{f}_{L^2(\mathcal{B}_R)} \\
   &\lesssim \delta \norm{\d\omega}_{L^2(\mathcal{B}_\sigma)}
   + \delta \norm{\d\omega}_{L^2(\mathcal{B}_\sigma)} \\ 
   &\qquad + C_\delta (\sigma - \rho)^{-\frac{n+2}{2}} 
   \norm{\omega}_{L^1(\mathcal{B}_\sigma)}
    + R\norm{f}_{L^2(\mathcal{B}_R)}  
  \end{split}
 \end{equation*}
 for an arbitrary~$\delta > 0$ and some constant~$C_\delta$
 depending on~$\delta$, but not on~$\rho$, $\sigma$.
 An iteration argument (see~\cite[Lemma~V.3.1]{Giaquinta-MultipleIntegrals} 
 or~\cite[Lemma~B.1]{Beck}) now gives
 \begin{equation} \label{elliptic3,2}
  \begin{split}
   \norm{\d\omega}_{L^2(\mathcal{B}_{r^\prime})}
   + \norm{\d\omega}_{L^2(\mathcal{B}_{r^\prime})}
   &\lesssim (R - r^\prime)^{-\frac{n+2}{2}} 
   \norm{\omega}_{L^1(\mathcal{B}_R)}
    + R \norm{f}_{L^2(\mathcal{B}_R)} 
  \end{split}
 \end{equation}
 where, we recall, $r^\prime = (r+R)/2$. 
 On the other hand, setting~$r^{\prime\prime} := (r + r^\prime)/2$, 
 Lemma~\ref{lemma:Nash} and the (weighted) Young inequality imply
 \begin{equation} \label{elliptic3,3}
  \begin{split}
   \norm{\omega}_{L^2(\mathcal{B}_{r^{\prime\prime}})}
   \lesssim (r^\prime - r^{\prime\prime})
    \norm{\d\omega}_{L^2(\mathcal{B}_{r^\prime})}
   + (r^\prime - r^{\prime\prime})
    \norm{\d\omega}_{L^2(\mathcal{B}_{r^\prime})}
   + (r^\prime - r^{\prime\prime})^{-\frac{n}{2}} 
    \norm{\omega}_{L^1(\mathcal{B}_{r^\prime})}^{2}
  \end{split}
 \end{equation}
 Finally, the Gaffney-type inequality~\eqref{Nash2} gives
 \begin{equation} \label{elliptic3,4}
  \begin{split}
   \norm{\omega}_{W^{1,2}(\mathcal{B}_r)}
   \lesssim \norm{\d\omega}_{L^2(\mathcal{B}_{r^{\prime\prime}})}
    + \norm{\d^*\omega}_{L^2(\mathcal{B}_{r^{\prime\prime}})} 
    + ({r^{\prime\prime}} - r)^{-1}\norm{\omega}_{L^2(\mathcal{B}_{r^{\prime\prime}})}  
  \end{split}
 \end{equation}
 Combining~\eqref{elliptic3,2} with~\eqref{elliptic3,3}
 and~\eqref{elliptic3,4}, the lemma follows.
\end{proof}

\paragraph{Acknowledgements}
F.~L.~D. has been supported by the project \textsc{Star Plus 2020 - Linea~1 
(21-UNINA-EPIG-172)} ``New perspectives in the Variational modelling of 
Continuum Mechanics''. G.~C. and F.~L.~D. thank the Hausdorff research Institute for Mathematics (HIM) for the warm hospitality during the Trimester 
Program ``Mathematics of complex materials'', funded by the Deutsche Forschungsgemeinschaft (DFG, German Research Foundation) under Germany's Excellence Strategy – EXC-2047/1 – 390685813, when part of this work was 
carried out. The authors have been supported by GNAMPA-INdAM.

\bibliographystyle{plain}
\bibliography{GL-bundles}

\newcommand{\noop}[1]{}
\begin{thebibliography}{10}

\bibitem{ABO1}
G.~Alberti, S.~Baldo, and G.~Orlandi.
\newblock Functions with prescribed singularities.
\newblock {\em J. Eur. Math. Soc. (JEMS)}, 5(3):275--311, 2003.

\bibitem{ABO2}
G.~Alberti, S.~Baldo, and G.~Orlandi.
\newblock Variational convergence for functionals of {G}inzburg-{L}andau type.
\newblock {\em Indiana Univ. Math. J.}, 54(5):1411--1472, 2005.

\bibitem{Almgren}
F.~J. Almgren.
\newblock {\em The theory of varifolds}.
\newblock Mimeographed notes. Princeton {U}niversity {P}ress, {P}rinceton,
  1965.

\bibitem{AmbrosioSoner}
L.~Ambrosio and H.-M. Soner.
\newblock A measure-theoretic approach to higher codimension mean curvature
  flows.
\newblock {\em Annali della Scuola Normale Superiore di Pisa - Classe di
  Scienze}, Ser. 4, 25(1-2):27--49, 1997.

\bibitem{Baraket}
S.~Baraket.
\newblock Critical points of the {G}inzburg-{L}andau system on a {R}iemannian
  surface.
\newblock {\em Asymptotic Anal.}, 13(3):277--317, 1996.

\bibitem{Beck}
L.~Beck.
\newblock {\em Elliptic regularity theory --- A first course}.
\newblock Lecture Notes of the Unione Matematica Italiana. Springer, 2016.

\bibitem{BBH}
F.~Bethuel, H.~Brezis, and F.~H{\'e}lein.
\newblock {\em Ginzburg-{L}andau {V}ortices}.
\newblock Progress in Nonlinear Differential Equations and their Applications,
  13. Birkh\"auser Boston Inc., Boston, MA, 1994.

\bibitem{BethuelBrezisOrlandi}
F.~Bethuel, H.~Brezis, and G.~Orlandi.
\newblock Asymptotics for the {G}inzburg-{L}andau equation in arbitrary
  dimensions.
\newblock {\em J. Funct. Anal.}, 186(2):432--520, 2001.

\bibitem{BethuelOrlandiSmets-Annals}
F.~Bethuel, G.~Orlandi, and D.~Smets.
\newblock Convergence of the parabolic {G}inzburg–{L}andau equation to motion
  by mean curvature.
\newblock {\em Ann. Math.}, 163(1):37--163, 2006.

\bibitem{BethuelRiviere}
F.~Bethuel and T.~Rivi{\`e}re.
\newblock Vortices for a variational problem related to superconductivity.
\newblock {\em Ann. Inst. H. Poincar\'e Anal. Non Lin\'eaire}, 12(3):243--303,
  1995.

\bibitem{BethuelZheng}
F.~Bethuel and X.~Zheng.
\newblock Density of smooth functions between two manifolds in {S}obolev
  spaces.
\newblock {\em J. Funct. Anal.}, 80(1):60 -- 75, 1988.

\bibitem{BrezisMironescu-book}
Haïm Brezis and Mironescu Petru.
\newblock {\em Sobolev Maps to the Circle: From the Perspective of Analysis,
  Geometry, and Topology}, volume~96.
\newblock 2021.

\bibitem{CDO1}
G.~Canevari, F.~Dipasquale, and G.~Orlandi.
\newblock The {Y}ang-{M}ills-{H}iggs functional on complex line bundles:
  ${\Gamma}$-convergence and the {L}ondon equation.
\newblock Preprint arXiv~2204.06491, 2022.

\bibitem{Cheng2020}
D.~R. Cheng.
\newblock Instability of solutions to the {G}inzburg-{L}andau equation on
  $\mathbb{S}^n$ and $\mathbb{C} \mathbb{P}^n$.
\newblock {\em J. Funct. Anal.}, 279(8):108669, 2020.

\bibitem{ColinetJerrardSternberg}
A.~Colinet, R.~Jerrard, and P.~Sternberg.
\newblock Solutions of the {G}inzburg-{L}andau equations with vorticity
  concentrating near a nondegenerate geodesic.
\newblock Preprint arXiv~2101.03575, 2021.

\bibitem{DePhilippisPigati}
G.~De~Philippis and A.~Pigati.
\newblock Non-degenerate minimal submanifolds as energy concentration sets: a
  variational approach.
\newblock Preprint arXiv~2205.12389, 2022.

\bibitem{deRhamKodaira}
G.~de~Rham and K.~Kodaira.
\newblock {\em Harmonic Integrals}.
\newblock Mimeographed notes, Institute for Advanced Study, 1954.

\bibitem{Giaquinta-MultipleIntegrals}
M.~Giaquinta.
\newblock {\em Multiple Integrals in the Calculus of Variations and Nonlinear
  Elliptic Systems}.
\newblock Princeton University Press, 1983.

\bibitem{IwaniecScottStroffolini}
T.~Iwaniec, C.~Scott, and B.~Stroffolini.
\newblock Nonlinear {H}odge theory on manifolds with boundary.
\newblock {\em Mat. Ann. Pura Appl.}, 177(1):37--115, 1999.

\bibitem{JaffeTaubes}
A.~Jaffe and C.~Taubes.
\newblock {\em Vortices and monopoles}, volume~2 of {\em Progress in Physics}.
\newblock Birkh\"{a}user, Boston, Mass., 1980.
\newblock Structure of static gauge theories.

\bibitem{JerrardSternberg}
R.~Jerrard and P.~Sternberg.
\newblock Critical points via {$\Gamma$}-convergence: general theory and
  applications.
\newblock {\em J. Eur. Math. Soc. (JEMS)}, 11(4):705--753, 2009.

\bibitem{JerrardSoner-GL}
R.~L. Jerrard and H.~M. Soner.
\newblock The {J}acobian and the {G}inzburg-{L}andau energy.
\newblock {\em Cal. Var. Partial Differential Equations}, 14(2):151--191, 2002.

\bibitem{LinRiviere2}
F.-H. Lin and T.~Rivi{\`e}re.
\newblock A quantization property for static {G}inzburg-{L}andau vortices.
\newblock {\em Comm. Pure Appl. Math.}, 54(2):69--84, 2001.

\bibitem{Monclair2021}
D.~Monclair.
\newblock Groups and geometry.
\newblock
  \url{https://www.imo.universite-paris-saclay.fr/~daniel.monclair/poly\{_}groupes\{_}geo.pdf},
  2021.

\bibitem{Nash1958}
J.~Nash.
\newblock Continuity of solutions of parabolic and elliptic equations.
\newblock {\em Am. J. Math}, 80(4):931--954, 1958.

\bibitem{Orlandi}
G.~Orlandi.
\newblock Asymptotic behavior of the {G}inzburg-{L}andau functional on complex
  line bundles over compact {R}iemann surfaces.
\newblock {\em Rev. Math. Phys.}, 08(03):457--486, 1996.

\bibitem{PigatiStern}
A.~Pigati and D.~Stern.
\newblock Minimal submanifolds from the abelian {H}iggs model.
\newblock {\em Invent. Math.}, 223(3):1027--1095, 2021.

\bibitem{PigatiStern-Integral}
A.~Pigati and D.~Stern.
\newblock Quantization and non-quantization of energy for higher-dimensional
  {G}inzburg--{L}andau vortices.
\newblock Preprint arXiv~2204.06491, 2022.

\bibitem{Pitts}
J.~Pitts.
\newblock {\em Existence and regularity of minimal surfaces in {R}iemannian
  manifolds}, volume~27 of {\em Mathematical Notes}.
\newblock Princeton {U}niversity {P}ress, {P}rinceton, N.J.; {U}niversity of
  {T}okyo {Press}, {T}okyo, 1981.

\bibitem{Qing}
J.~Qing.
\newblock Renormalized energy for {G}inzburg-{L}andau vortices on closed
  surfaces.
\newblock {\em Math. Z.}, 225(1):1--34, 1997.

\bibitem{SS-book}
{\'E}.~Sandier and S.~Serfaty.
\newblock {\em Vortices in the magnetic {G}inzburg-{L}andau model}.
\newblock Progress in Nonlinear Differential Equations and their Applications,
  70. Birkh\"auser Boston, Inc., Boston, MA, 2007.

\bibitem{SandierSerfaty-book}
{\'E}.~Sandier and S.~Serfaty.
\newblock {\em Vortices in the magnetic {G}inzburg-{L}andau model}.
\newblock Progress in Nonlinear Differential Equations and their Applications,
  70. Birkh\"auser Boston, Inc., Boston, MA, 2007.

\bibitem{Scott}
C.~Scott.
\newblock ${L}^p$ theory of differential forms on manifolds.
\newblock {\em Trans. Amer. Math. Soc.}, 347(6):2075--2096, 1995.

\bibitem{Simon-GMT}
L.~Simon.
\newblock {\em {L}ectures in {G}eometric {M}easure {T}heory}.
\newblock Centre for Mathematical Analysis, Australian National University,
  Canberra, 1984.

\bibitem{SmithUhlenbeck}
P.~Smith and K.~Uhlenbeck.
\newblock Removeability of a codimension four singular set for solutions of a
  yang mills higgs equation with small energy.
\newblock In {\em Surveys in differential geometry 2019. Differential geometry,
  Calabi-Yau theory, and general relativity. Part 2}, pages 257--291. Int.
  Press, Boston, MA, 2019.

\bibitem{Stern2020}
D.~Stern.
\newblock $p$-{H}armonic maps to $\mathbb{S}^1$ and stationary varifolds of
  codimension two.
\newblock {\em Calc. Var.}, 59(187), 2020.

\bibitem{Stern2021}
D.~Stern.
\newblock Existence and limiting behavior of min--max solutions of the
  {G}inzburg--{L}andau equations on compact manifolds.
\newblock {\em J. Diff. Geom.}, 118(2), 2021.

\end{thebibliography}

\Addresses

\end{document}